\documentclass[11pt]{article}
\usepackage{latexsym,amsmath,amssymb,amsfonts,mathrsfs,amsthm}
\usepackage{epsf,graphicx,epsfig,color,cite,cases}
\usepackage{subfigure,graphics,multirow,marginnote,enumerate,bm}
\usepackage[T1]{fontenc}
\usepackage{algorithm}

\newcommand{\R}{{\mat R}}

\newcommand{\be}{\begin{eqnarray}}
\newcommand{\ben}{\begin{eqnarray*}}
\newcommand{\en}{\end{eqnarray}}
\newcommand{\enn}{\end{eqnarray*}}

\newcommand{\pa}{\partial}

\newcommand{\G}{\Gamma}

\newcommand{\Om}{\Omega}

\newcommand{\mat}{\mathbb}
\newcommand{\se}{\setminus}

\newtheorem{theorem}{Theorem}[section]
\newtheorem{lemma}[theorem]{Lemma}

\begin{document}
\renewcommand{\theequation}{\arabic{section}.\arabic{equation}}

\title{\bf Reverse time migration for inverse acoustic scattering by locally rough surfaces}

\author{
Jianliang Li\thanks{Key Laboratory of Computing and Stochastic Mathematics, School of Mathematics and Statistics, Hunan Normal University, Changsha, Hunan 410081, China ({\tt lijianliang@hunnu.edu.cn, lijl@amss.ac.cn})}
\and
Hao Wu\thanks{School of Mathematics and Statistics, Xi'an Jiaotong University,
		Xi'an, Shaanxi 710049, China ({\tt wuhao2022@stu.xjtu.edu.cn})}
\and
Jiaqing Yang\thanks{School of Mathematics and Statistics, Xi'an Jiaotong University,
Xi'an, Shaanxi 710049, China ({\tt jiaq.yang@mail.xjtu.edu.cn})}
}
\date{}

\maketitle

\begin{abstract}
Consider the inverse scattering of time-harmonic acoustic scattering by an infinite rough surface which is supposed to be a 
local perturbation of a plane. A novel version of reverse time migration (RTM) is proposed to reconstruct the shape and location of the rough surface.  The method is based on a modified Helmholtz-Kirchhoff identity associated with a special rough surface, leading to a modified imaging functional which uses the near-field data generated by point sources as measurements. The modified imaging functional always reaches a peak on the boundary of the rough surface for sound-soft case and penetrable case, and hits a nadir on the boundary of the rough surface for sound-hard case. Furthermore, we also establish the RTM method associated with the far-field data generated by plane waves. As far as we know, this is the first result for the RTM method with the far-filed data. Numerical experiments are presented to show the powerful imaging quality.
\vspace{.2in}


\end{abstract}

\maketitle

\section{Introduction}\label{sec1}

Motivated by significant applications in medical imaging \cite{SA99} and exploration geophysics \cite{BA84}, we consider the two-dimensional 
inverse acoustic scattering of time-harmonic acoustic scattering by a locally rough surface, which aims to reconstruct the shape and location of the rough surface.

The scattering surface is described by a curve 
\begin{eqnarray}\label{a1}
	\Gamma:=\{(x_1,x_2)\in\mathbb R^2: x_2=f(x_1)\}
\end{eqnarray}
where $f$ is assumed to be a Lipschitz continuous function with compact support. This means that the surface $\Gamma$ 
is a local perturbation of the plane
$\Gamma_0:=\{(x_1,x_2)\in\mathbb R^2: x_2=0\}.$
The rough surface $\Gamma$ separates the whole space into two half-spaces denoted by
\begin{eqnarray*}\label{a3}
	\Omega_1:=\{(x_1,x_2)\in\mathbb R^2: x_2>f(x_1)\}\quad{\rm and}\quad \Omega_2:=\{(x_1,x_2)\in\mathbb R^2: x_2<f(x_1)\},
\end{eqnarray*}
which are filled with homogeneous mediums described by wave numbers $\kappa_1>0$ and $\kappa_2>0$, respectively.

For impenetrable cases, the incident wave $u^i$ is induced by the point source, which means 
\begin{eqnarray}\label{a4}
	u^i(x)=\Phi_{\kappa_1}(x,x_s):=\frac{\rm i}{4}H_0^{(1)}(\kappa_1|x-x_s|)\quad {\rm for}\;\; x_s\in\Omega_1.
\end{eqnarray}
Here, $H_0^{(1)}$ is the Hankel function of the first kind of order zero, and $\Phi_{\kappa_1}$ is the fundamental solution of the Helmholtz
equation satisfying $\Delta\Phi_{\kappa_1}(\cdot, x_s)+\kappa_1^2\Phi_{\kappa_1}(\cdot, x_s)=-\delta_{x_s}(\cdot)$ in $\mathbb R^2$, where $\delta$ is the Kronecker delta distribution. Then the scattering of $u^i$ by the rough surface $\Gamma$ can be modelled by
\begin{equation}\label{a5}
	\left\{\begin{aligned}
		&\Delta u_{\alpha}+\kappa_1^2u_{\alpha}=-\delta_{x_s} \qquad\qquad\textrm{in}\;\; \Omega_1, \\
		&\mathcal{B}u_{\alpha}=0\qquad\qquad\qquad\qquad\;\;\;\;\textrm{on}\;\; \Gamma,\\
		&\lim_{|x|\rightarrow \infty}|x|^{\frac{1}{2}}\left(\partial_{|x|} u_{\alpha}^s-{\rm i}\kappa_1 u_{\alpha}^s\right)=0,\quad
	\end{aligned}
	\right.
\end{equation}
where $u_{\alpha}:=u^i+u_{\alpha}^s$ for $\alpha\in\{D,N\}$ denotes the total field which is the sum of the incident field $u^i$ and the scattered field $u_{\alpha}^s$. ${\mathcal B}$ stands for the boundary condition on $\Gamma$ satisfying ${\mathcal B}u_D:=u_D$ if $\Gamma$ is a sound-soft rough surface, and ${\mathcal B}u_N=\partial_{\nu} u_N$ if $\Gamma$ is a sound-hard rough surface. Here and throughout, $\nu=\nu(x)$ is the upward normal vector directing into $\Omega_1$ for $x\in\Gamma$, and $\partial_{\nu}$ stands for the normal derivative. The last condition in (\ref{a5}) is the well-known Sommerfeld radiation condition which holds uniformly for all directions $\hat{x}:=x/|x|\in{\mathbb S}_+:=\{x\in{\mathbb R}^2: |x|=1,x_2>0\}$.

For the penetrable case, consider an incoming wave induced by the point source (\ref{a4})
to be incident on the scattering interface $\Gamma$ from the domain $\Omega_1$. Then the scattering of $u^i$ by $\Gamma$ can be modelled by 
\be\label{a7}
\left\{\begin{array}{lll}
	\Delta u^s+\kappa_1^2u^s=0& \textrm{in}\;\Om_1 \\[2mm]
	\Delta u^s+\kappa_2^2u^s =0&\textrm{in}\;\Om_2 \\[2mm]
	u^s|_+-u^s|_- =-u^i&{\rm on\;}\G\\[2mm]
	\pa_\nu u^s|_+-\pa_\nu u^s|_-=-\pa_\nu u^i& {\rm on\;}\G\\[2mm]
	\lim\limits_{|x|\rightarrow\infty}|x|^{\frac12}\left(\partial_{|x|} u^s-{\rm i}\kappa u^s\right)=0.
\end{array}
\right.
\en
Here the notation $\cdot|_\pm$ represents the limits of $\cdot$ approaching $\G$ from $\Om_1$ and $\Om_2$, respectively,
$\kappa$ is the wavenumber defined by $\kappa:=\kappa_1$ in $\Omega_1$ and $\kappa:=\kappa_2$ in $\Omega_2$.
The last condition in (\ref{a7}) is the Sommerfeld radiation condition which holds uniformly for all directions $\hat{x}:=x/|x|\in{\mathbb S}:=\{x\in{\mathbb R}^2: |x|=1\}$. It is worth pointing out that the scattered field $u^s$ satisfying the Sommerfeld radiation condition has the asymptotic behavior of an outgoing spherical wave 
\begin{eqnarray*}\label{a9}
	u^s(x)=\frac{e^{{\rm i}\kappa |x|}}{|x|^{\frac{1}{2}}}\left\{u^{\infty}(\hat{x})+O\left(\frac{1}{|x|}\right)\right\}\quad {\rm for}\; |x|\to\infty
\end{eqnarray*}
uniformly in all direction $\hat{x}\in{\mathbb S}$, where $u^{\infty}(\hat{x})$ is known as the far field pattern of $u^s$.

Given the incident wave, the rough surface, and the boundary condition, the direct scattering problem is to determine the distribution of the scattered wave, which is extensively studied by the variational method \cite{SE10,MT06} and the integral equation approach \cite{LYZ13,DTS03,ZS03} with employing a generalized Fredholm theory \cite{SZ97, SZ00}. Recently, a novel technique was proposed to prove the well-posedness of (\ref{a5}) for sound-soft case in \cite{DLLY17}, based on transferring the unbounded, locally rough surface scattering problem into an equivalent boundary value problem with the compactly supported boundary data, whose well-posedness follows from the classical Fredholm theory. The novel technique has been extended to deal with penetrable, locally rough surface in \cite{LYZ21}, and it has been extended in \cite{LYZ22} to investigate penetrable, locally rough surface with embedded obstacles in the lower half-space. 
While the inverse scattering problem is to determine the rough surface from the measured scattered field in some certain domain. For the time-harmonic case, there exists a large number of references on inversion methods such as Newton-type approaches \cite{GJP11,GJ11,GJ13,SP02,QZZ19,RM15,ZZ13}, the Kirsch-Kress schemes \cite{CR10,LZ13}, nonlinear integral equation methods \cite{L19,LB13}, reconstruction algorithms based on transformed field expansions \cite{GL13,GL14}, the factorization method \cite{RG08,AL08}, linear sampling methods \cite{DLLY17,LYZ21, ZY22}, and the direct imaging methods \cite{LLLL15, LZZ18,LZZ19, LYZZ23}.
For the time-domain case, a singular source method has been extended to solve the inverse rough surfae scattering problem\cite{C03}.

The RTM method is a sample-type method which are widely applied in exploration geophysics \cite{BA84, BCS01, JFC85} and seismic imaging \cite{BCS01}. The main idea of the RTM method consists of two steps. The first step is to back-propagate the complex conjugated data into the back-ground medium, and the second step is to compute the cross-correlation between the incident field and the back propagated field. Thus we can define the imaging functional as the imaginary part of the cross-correlation, which always peak on the boundary of the scatterer. Since the RTM method can provide an effective, stable and powerful reconstruction of the scatterer, it has gained considerable attention 
and has been extensively investigated by mathematicians and engineers. Mathematically, the justification of the RTM method has been proved rigorously for the inverse obstacle scattering problem for acoustic waves \cite{CCH131,L21,CH153,CH152}, elastic waves \cite{CH151}, and electromagnetic waves \cite{CCH132}.
It is worth mentioning that the RTM method in these references requires the full scattering data (both the intensity and phase information). On the other hand, in a variety of realistic applications, only the phaseless data is available. In this case, the RTM approach have been developed for acoustic waves \cite{CH17} and electromagnetic waves \cite{CH16}. It is shown in \cite{CH16, CH17} that the imaging functional with phaseless data has essentially the same asymptotic behavior as the case of full data. It is noticed that the almost all references about the RTM method are related to the inverse scattering associated with bounded obstacles. However, it is challenging to develop the RTM method for reconstructing an infinite rough surface since the general Helmholtz-Kirchhoff identity is not valid in this case.  No such a result is available so far. 

In this paper, we aim to investigate the RTM scheme to reconstruct infinite, locally rough surfaces where the key difficulty is that the usual Helmholtz-Kirchhoff identity presented in \cite{CCH131} is not applicable. The first novelty of this paper is to establish a modified Helmholtz-Kirchhoff identity by the Green's function associated with a special locally rough surface. Thus, a modified imaging functional is proposed from the near-field data generated by point sources, where its mathematical justification is proved rigorously. Specifically, we demonstrate that the modified imaging functional enjoys the nice feature that it always peaks on the boundary of the sound-soft and penetrable rough surfaces, and reaches a nadir on the boundary of the sound-hard rough surface. 
The second novelty is associated with the RTM with the far-field measurements generated by plane waves. An imaging functional is first proposed based on a novel mixed reciprocity relation, which is indeed the limit of the imaging functional  with the near-field data generated by point sources. It is further shown that the imaging functional has the same asymptotic property with the case of the near-filed measurements, and can be thus used to recover the the location and shape of the rough surfaces. To the best of our knowledge, this is the first result for the RTM method with using far-field measurements, especially for the reconstruction of infinite rough surfaces.
The numerical experiments shows that the two modified imaging functionals can provide a stable and powerful reconstruction for the rough surfaces. 

The rest of the paper is organized as follows. In section 2, we develop the RTM method for the locally  sound-soft and sound-hard rough surface, which includes the near-field reconstruction in the first subsection and the far-field reconstruction in the second subsection. Section 3 is devoted to the RTM approach for penetrable locally rough surfaces which consists of the near-field and far-field reconstructions, respectively. In section 4, we present some numerical experiments to demonstrate the validity of the RTM method. This paper concludes with some general remarks and discussions on the future work in section 5.

\section{The RTM for impenetrable locally rough surfaces}
\setcounter{equation}{0}

In this section, we investigate the RTM method for inverse acoustic scattering by impenetrable locally rough surfaces with the Dirichlet or Neumann boundary conditions. This section consists of two subsections. The first subsection will discuss the RTM method based on the near-field data which corresponds to the point source incidence. Based on a new mixed reciprocity relation, we will present the RTM method based on the far-field data generated by the plane wave incidence.

\subsection{The near-field reconstruction}
\setcounter{equation}{0}

\begin{figure}[htbp]
	\centering
	\includegraphics[width=5in, height=3in]{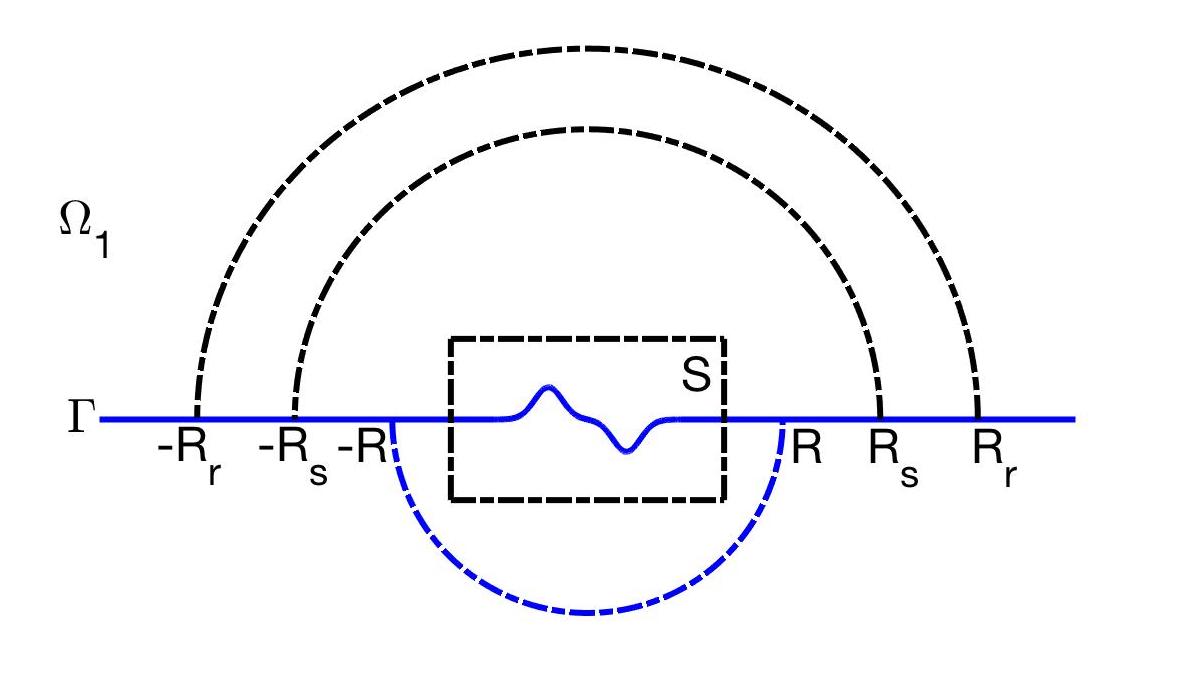}
	\caption{The setting of RTM method for sound-soft and sound-hard cases.}
	\label{f1} 
\end{figure}

As shown in Figure \ref{f1}, let $\Gamma$ be the locally rough surface defined by (\ref{a1}), whose local perturbation is contained 
in a rectangle sampling domain $S$. We choose a large enough $R$ and define a special locally rough surface $\Gamma_R$ as 
\begin{eqnarray}\label{b1}
	\Gamma_R:=\{x\in\mathbb R^2: x_2=0\;{\rm for }\;|x_1|\geq R\;{\rm and}\; x_2=-\sqrt{R^2-x_1^2}\;{\rm for}\;|x_1|\leq R\}
\end{eqnarray}
such that the sampling domain $S$ lies totally above $\Gamma_R$. We assume that there are $N_s$ point sources $x_s$ uniformly distributed on $\Gamma_s$ and $N_r$ receivers $x_r$ uniformly distributed on $\Gamma_r$. Here, $\Gamma_s$ and $\Gamma_r$ denote the upper semicircle with the origin as the center and $R_s$, $R_r$ as the radius, respectively. We assume a priori that $R<R_s\leq R_r$.

We first consider the scattering of the incident point source $\Phi_{\kappa_1}(x,x_s)$ given by (\ref{a4}) by the special locally rough surface $\Gamma_R$, which reads 
\begin{eqnarray}\label{b2}
	\left\{\begin{array}{lll}
		\Delta G_{\alpha}^s(x,x_s)+\kappa_1^2G_{\alpha}^s(x,x_s)=0 &\textrm{in}\;\; \Omega_R, \\[3mm]
		\mathcal{B}G_{\alpha}^s(x,x_s)=-\mathcal{B}\Phi_{\kappa_1}(x,x_s)&\textrm{on}\;\; \Gamma_R,\\[3mm]
		\lim\limits_{|x|\rightarrow \infty}|x|^{\frac{1}{2}}\left(\partial_{|x|} G_{\alpha}^s(x,x_s)-{\rm i}\kappa_1  G_{\alpha}^s(x,x_s)\right)=0.
	\end{array}
	\right.
\end{eqnarray}
Here $\alpha\in\{D,N\}$, $\alpha=D$ represents the Dirichlet boundary condition imposed on $\Gamma_R$ which means $G_D^s(x,x_s)=-\Phi_{\kappa_1}(x,x_s)$ on $\Gamma_R$, and $\alpha=N$ represents the Neumann boundary condition imposed on $\Gamma_R$ which means $\partial_{\nu}G_N^s(x,x_s)=-\partial_{\nu}\Phi_{\kappa_1}(x,x_s)$ on $\Gamma_R$. $G^s_{\alpha}(x,x_s)$ is the scattered field, $G_{\alpha}(x,x_s):=G_{\alpha}^s(x,x_s)+\Phi_{\kappa_1}(x,x_s)$ denotes the total field, and $\Omega_R$ is the upper half-space separated by $\Gamma_R$, which means 
\begin{eqnarray*}\label{b3}
	\Omega_R:=\{(x_1,x_2)\in\mathbb R^2: x_2>0\;\;{\rm for }\;|x_1|\geq R,\;\;{\rm and}\;\; x_2>-\sqrt{R^2-x_1^2}\;\;{\rm for}\;|x_1|\leq R\}.
\end{eqnarray*}
It follows from \cite{GJ11,DLLY17,QZZ19} that Problem (\ref{b2}) is well-posed in a standard Sobolev space. 

For $x_s\in\Gamma_s$, recall that $u_{\alpha}^s(x,x_s)$ is the solution of Problem (\ref{a5}) when the incident source is located at $x_s$. Define 
\begin{eqnarray}\label{b8}
	V_{\alpha}(x,x_s):=u_{\alpha}^s(x,x_s)-G^s_{\alpha}(x,x_s)
\end{eqnarray}
then it is easily checked that it solves 
\begin{eqnarray}\label{b9}
	\left\{\begin{aligned}
		&\Delta V_{\alpha}(x,x_s)+\kappa_1^2V_{\alpha}(x,x_s)=0 \qquad\qquad\qquad\;\; \textrm{in}\;\; \Omega_1, \\
		&\;\mathcal{B}V_{\alpha}(x,x_s)=-\mathcal{B}G_{\alpha}(x,x_s)\qquad\qquad\qquad\;\;\;\;\;\textrm{on}\;\; \Gamma\setminus\Gamma_R,\\
		&\;\mathcal{B}V_{\alpha}(x,x_s)=0\qquad\qquad\qquad\qquad\qquad\qquad\;\;\textrm{on}\;\;\Gamma\cap\Gamma_R,\\
		&\lim\limits_{|x|\rightarrow \infty}|x|^{\frac{1}{2}}\left(\partial_{|x|} V_{\alpha}(x,x_s)-{\rm i}\kappa_1  V_{\alpha}(x,x_s)\right)=0.
	\end{aligned}
	\right.
\end{eqnarray}
Noting that we can compute the scattered field $G_{\alpha}^s(x_r,x_s)$ by solving Problem (\ref{b2}) using Nystr\"{o}m method or finite element method. Thus, we can obtain $V_{\alpha}(x_r,x_s)$ from the measurement $u_{\alpha}^s(x_r,x_s)$ and (\ref{b8}). 

Now, we are able to introduce the RTM method which consists of two steps. The first step is to back-propagate the complex  conjugated data $\overline{V_{\alpha}(x_r,x_s)}$ into the domain $\Omega_R$; the second step is to calculate the imaginary part of the cross-correlation of $G_{\alpha}(\cdot, x_s)$ and the back-propagation field. More precisely, we summarize it in the following algorithm.

{\bf Algorithm 1 (RTM for impenetrable locally rough surfaces)}: Given the data $V_{\alpha}(x_r,x_s)$ for $r=1,2,...,N_r$ and $s=1,2,...,N_s$.
\begin{itemize}
	\item Back-propagation: for $s=1,2,...,N_s$, solve the problem 
	\begin{eqnarray}\label{b10}
		\left\{\begin{aligned}
			&\Delta W_{\alpha}(x,x_s)+\kappa_1^2W_{\alpha}(x,x_s)=\frac{|\Gamma_r|}{N_r}\sum_{r=1}^{N_r}\overline{V_{\alpha}(x_r,x_s)}\delta_{x_r}(x) \;\textrm{in}\;\; \Omega_R, \\
			& \;\mathcal{B}W_{\alpha}(x,x_s)=0\qquad\qquad\qquad\qquad\qquad\qquad\qquad\qquad\quad\textrm{on}\;\; \Gamma_R,\\
			&\; \lim\limits_{|x|\rightarrow \infty}|x|^{\frac{1}{2}}\left(\partial_{|x|} W_{\alpha}(x,x_s)-{\rm i}\kappa_1  W_{\alpha}(x,x_s)\right)=0,
		\end{aligned}
		\right.
	\end{eqnarray}
	to address the solution $W_{\alpha}$.
	\item Cross-correlation: for each sampling point $z\in S$, calculate the indicator function 
	\begin{eqnarray*}\label{b11}
		{\rm Ind}_{\alpha}(z)=\kappa_1^2{\rm Im}\left\{\frac{|\Gamma_s|}{N_s}\sum_{s=1}^{N_s}G_{\alpha}(z,x_s)W_{\alpha}(z,x_s)\right\}
	\end{eqnarray*}
	and then plot the mapping ${\rm Ind}_{\alpha}(z)$ against $z$.
\end{itemize}

It follows from the linearity that the solution of Problem (\ref{b10}) can be represented by 
\begin{eqnarray*}\label{b12}
	W_{\alpha}(x,x_s)=-\frac{|\Gamma_r|}{N_r}\sum_{r=1}^{N_r}\overline{V_{\alpha}(x_r,x_s)}G_{\alpha}(x,x_r),
\end{eqnarray*}
which leads to 
\begin{eqnarray}\label{b13}
	{\rm Ind}_{\alpha}(z)=-\kappa_1^2{\rm Im}\left\{\frac{|\Gamma_s|}{N_s}\frac{|\Gamma_r|}{N_r}\sum_{s=1}^{N_s}\sum_{r=1}^{N_r}G_{\alpha}(z,x_s)G_{\alpha}(z,x_r)\overline{V_{\alpha}(x_r,x_s)}\right\}\;\; z\in S.
\end{eqnarray}
Observing that the sampling point $z\in S$, the source location $x_s\in\Gamma_s$, the receiver $x_r\in\Gamma_r$, and $S$ is inside in $B_R$, we have $G_{\alpha}(z,x_s)$ and $G_{\alpha}(z,x_r)$ are smooth. Combining of the smoothness of $V_{\alpha}(x_r,x_s)$ and the trapezoid quadrature formula yields that ${\rm Ind}_{\alpha}(z)$ given by (\ref{b13}) is a discrete formula of the following continuous function:
\begin{eqnarray*}\label{b14}
	\widetilde{{\rm Ind}}_{\alpha}(z)=-\kappa_1^2{\rm Im}\int_{\Gamma_r}\int_{\Gamma_s}G_{\alpha}(z,x_s)G_{\alpha}(z,x_r)\overline{V_{\alpha}(x_r,x_s)}{\rm d}s(x_s){\rm d}s(x_r),\quad z\in S.
\end{eqnarray*}

The remaining part of this subsection aims to give a resolution analysis of the function $\widetilde{{\rm Ind}}_{\alpha}(z)$. Our destination is to show that $\widetilde{{\rm Ind}}_{\alpha}(z)$ will have contrast at the rough surface $\Gamma$ and decay away from $\Gamma$. To this end, we introduce the following Lemmas \ref{lem1}-\ref{lem3} and Theorem \ref{thm1} and, for simplicity, just give the proof for the Dirichlet boundary condition case, the results for the Neumann boundary condition case can be proven by similar arguments. We first introduce the following modified Helmholtz-Kirchhoff identity.
\begin{lemma}\label{lem1}
	Let $G_{\alpha}$ be the total field of the scattering problem (\ref{b2}), and $\nu$ be the unit upward normal to $\Gamma_p$ for $p\in\{r,s\}$, then we have 
	\begin{eqnarray*}\label{b4}
		\int_{\Gamma_p}\left(\overline{G_{\alpha}(\xi,x)}\frac{\partial G_{\alpha}(\xi,z)}{\partial \nu(\xi)}-\frac{\partial \overline{G_{\alpha}(\xi,x)}}{\partial\nu(\xi)}G_{\alpha}(\xi,z)\right){\rm d}s(\xi)=2{\rm i}{\rm Im}G_{\alpha}(x,z)
	\end{eqnarray*}
	for any $x,z\in B_{R_s}\cap\Omega_R$.
\end{lemma}
\begin{proof}
	For $\alpha=D$, $\Gamma_p=\Gamma_s$ and any $x,z\in B_{R_s}\cap\Omega_R$, we choose a sufficient small $\varepsilon>0$ such that the circles $B_{\varepsilon}(x)$, $B_{\varepsilon}(z)$ with $x,z$ as the center and $\varepsilon$ as the radius contains in the domain $B_{R_s}\cap\Omega_R$. A direct application of the Green theorem to $\overline{G_D(\cdot, x)}$ and $G_D(\cdot,z)$ in the domain $(B_{R_s}\cap\Omega_R)\setminus(\overline{B_{\varepsilon}(x)}\cup \overline{B_{\varepsilon}(z)})$ yields 
	\begin{eqnarray}\nonumber
		0&=&\int_{(B_{R_s}\cap\Omega_R)\setminus(\overline{B_{\varepsilon}(x)}\cup \overline{B_{\varepsilon}(z)})}\left(\overline{G_D(\xi,x)}\Delta G_D(\xi,z)-\Delta\overline{G_D(\xi,x)}G_D(\xi,z)\right){\rm d}\xi\\\nonumber
		&=&\int_{\Gamma_s\cup(B_{R_s}\cap\Gamma_R)\cup\partial B_{\varepsilon}(x)\cup\partial B_{\varepsilon}(z)}\left(\overline{G_D(\xi,x)}\frac{\partial G_D(\xi,z)}{\partial \nu(\xi)}-\frac{\partial\overline{G_D(\xi,x)}}{\partial \nu(\xi)}G_D(\xi,z)\right){\rm d}s(\xi)\\\label{b5}
		&=&I_1+I_2+I_3+I_4,
	\end{eqnarray}
	where $\nu(\xi)$ denotes the unit downward normal to $B_{R_s}\cap\Gamma_R$ when $\xi\in B_{R_s}\cap\Gamma_R$, and $\nu(\xi)$ denotes the unit normal to $\partial B_{\varepsilon}(x)$,  $\partial B_{\varepsilon}(z)$ into the interior of  $B_{\varepsilon}(x)$, $B_{\varepsilon}(z)$, respectively.  Since $G_D(\xi,x)$ and $G_D(\xi,z)$ vanish on $B_{R_s}\cap\Gamma_R$, we have $I_2=0$. For the item $I_3$, using $G_D(\xi,x)=G_D^s(\xi,x)+\Phi_{\kappa_1}(\xi,x)$ gives that 
	\begin{eqnarray*}\nonumber
		I_3 &=& \int_{\partial B_{\varepsilon}(x)}\left(\overline{\Phi_{\kappa_1}(\xi,x)}\frac{\partial G_D(\xi,z)}{\partial \nu(\xi)}-\frac{\partial\overline{\Phi_{\kappa_1}(\xi,x)}}{\partial \nu(\xi)}G_D(\xi,z)\right){\rm d}s(\xi)\\\nonumber
		&&+\int_{\partial B_{\varepsilon}(x)}\left(\overline{G^s_D(\xi,x)}\frac{\partial G_D(\xi,z)}{\partial \nu(\xi)}-\frac{\partial\overline{G^s_D(\xi,x)}}{\partial \nu(\xi)}G_D(\xi,z)\right){\rm d}s(\xi)\\\label{b6}
		&\to&-G_D(x,z),\quad{\rm as}\;\;\varepsilon\to 0.
	\end{eqnarray*}
	Similarly, we can obtain 
	\begin{eqnarray}\label{b7}
		\lim_{\varepsilon\to 0}I_4=\overline{G_D(z,x)}.
	\end{eqnarray}
	Combining (\ref{b5})--(\ref{b7}) implies that 
	\begin{eqnarray*}
		\int_{\Gamma_s}\left(\overline{G_D(\xi,x)}\frac{\partial G_D(\xi,z)}{\partial \nu(\xi)}-\frac{\partial\overline{G_D(\xi,x)}}{\partial \nu(\xi)}G_D(\xi,z)\right){\rm d}s(\xi)=2{\rm i}{\rm Im} G_D(x,z),
	\end{eqnarray*}
	where we use the reciprocity $G_D(x,z)=G_D(z,x)$ for any $x,z\in \Omega_R$, which is proven in the Lemma 3.1 of \cite{DLLY17}. We can obtain the result for the case $\Gamma_p=\Gamma_r$ by a similar argument and we omit it here. The proof is completed.
\end{proof}

With the help of the above Helmholtz-Kirchhoff identity, we can address the following lemma which plays a key role in the analysis of $\widetilde{{\rm Ind}}_{\alpha}(z)$.

\begin{lemma}\label{lem2}
	For any $x,z\in S$, we have
	\begin{eqnarray}\label{b15}
		&&\kappa_1\int_{\Gamma_s}\overline{G_{\alpha}(x,\xi)}G_{\alpha}(\xi,z){\rm d}s(\xi)={\rm Im}G_{\alpha}(x,z)+\zeta_{\alpha, s}(x,z)\\\label{b16}
	      &&\kappa_1\int_{\Gamma_r}\overline{G_{\alpha}(x,\xi)}G_{\alpha}(\xi,z){\rm d}s(\xi)={\rm Im}G_{\alpha}(x,z)+\zeta_{\alpha, r}(x,z)
	\end{eqnarray}
	where  $|\zeta_{\alpha,s}(x,z)|+|\nabla_x\zeta_{\alpha,s}(x,z)|\leq CR_s^{-1}$, $|\zeta_{\alpha,r}(x,z)|+|\nabla_x\zeta_{\alpha,r}(x,z)|\leq CR_r^{-1}$ uniformly for any $x,z\in S$.
\end{lemma}
\begin{proof}
	For $\alpha = D$ and any $x,z\in S$, it follows from Lemma \ref{lem1} that 
	\begin{eqnarray*}
		2{\rm i}{\rm Im}G_D(x,z)&=&\int_{\Gamma_s}\left(\overline{G_D(\xi,x)}\frac{\partial G_D(\xi,z)}{\partial \nu(\xi)}-\frac{\partial \overline{G_D(\xi,x)}}{\partial\nu(\xi)}G_D(\xi,z)\right){\rm d}s(\xi)\\
		&=&\int_{\Gamma_s}\Bigg\{\overline{G_D(\xi,x)}\left[\frac{\partial G_D(\xi,z)}{\partial \nu(\xi)}-{\rm i}\kappa_1 G_D(\xi,z)\right]\\
		&&\qquad\;\;\;-G_D(\xi,z)\left[\frac{\overline{\partial G_D(\xi,x)}}{\partial \nu(\xi)}+{\rm i}\kappa_1\overline{G_D(\xi,x)}\right]\Bigg\}{\rm d}s(\xi)\\
		&&\qquad\;\;\;+2{\rm i}\kappa_1\int_{\Gamma_s}\overline{G_D(\xi,x)}G_D(\xi,z){\rm d}s(\xi).
	\end{eqnarray*}
	Thus, a direct application of the reciprocity $G_D(\xi,x)=G_D(x,\xi)$ for $\xi\in \Gamma_s$ and $x\in S$ yields 
	\begin{eqnarray*}
		\kappa_1\int_{\Gamma_s}\overline{G_D(x,\xi)}G_D(\xi,z){\rm d}s(\xi)={\rm Im}G_D(x,z)+\zeta_{D,s}(x,z)\quad{\rm for}\;\;{\forall x,z\in S},
	\end{eqnarray*}
	with 
	\begin{eqnarray*}
		\zeta_{D, s}(x,z)&=&\frac{\rm i}{2}\int_{\Gamma_s}\Bigg\{\overline{G_D(\xi,x)}\left[\frac{\partial G_D(\xi,z)}{\partial \nu(\xi)}-{\rm i}\kappa_1 G_D(\xi,z)\right]\\
		&&\qquad\quad-G_D(\xi,z)\left[\frac{\overline{\partial G_D(\xi,x)}}{\partial \nu(\xi)}+{\rm i}\kappa_1\overline{G_D(\xi,x)}\right]\Bigg\}{\rm d}s(\xi).
	\end{eqnarray*}
Thus, the inequality (\ref{b15}) holds. Due to 
	\begin{eqnarray*}
		G_D(\xi,y)=O(|\xi|^{-\frac{1}{2}}),\qquad\frac{\partial G_D(\xi,y)}{\partial \nu(\xi)}-{\rm i}\kappa_1 G_D(\xi,y)=O(|\xi|^{-\frac{3}{2}})
	\end{eqnarray*}
	for $y\in\{x,z\}$ and $x,z\in S$, it follows that 
	\begin{eqnarray*}
		|\zeta_{D,s}(x,z)|\leq CR_s^{-1}
	\end{eqnarray*}
	uniformly for $x,z\in S$. Since 
	\begin{eqnarray*}
		\frac{\partial G_D(\xi,x)}{\partial x_j}=O(|\xi|^{-\frac{1}{2}}),\qquad\frac{\partial}{\partial x_j}\left[\frac{\partial G_D(\xi,x)}{\partial \nu(\xi)}-{\rm i}\kappa_1 G_D(\xi,x)\right]=O(|\xi|^{-\frac{3}{2}})
	\end{eqnarray*}
	for $j=1,2$ and $x\in S$, it follows that 
	\begin{eqnarray*}
		|\nabla_x\zeta_{D,s}(x,z)|\leq CR_s^{-1}
	\end{eqnarray*}
	uniformly for $x,z\in S$. Thus, we conclude that $|\zeta_{D,s}(x,z)|+|\nabla_x\zeta_{D,s}(x,z)|\leq CR_s^{-1}$ holds. It is obvious that we can prove (\ref{b16}) similarly. The proof is completed. 
\end{proof}
\begin{lemma}\label{lem3}
	Let $V_{\alpha}(x,x_s)$ be the solution to Problem (\ref{b9}), then we have the Green's formula 
	\begin{eqnarray*}\label{b28}
		V_{\alpha}(x,x_s)=\int_{\Gamma\setminus\Gamma_R}\left[\frac{\partial G_{\alpha}(\xi,x)}{\partial\nu(\xi)}V_{\alpha}(\xi,x_s)-\frac{\partial V_{\alpha}(\xi,x_s)}{\partial \nu(\xi)}G_{\alpha}(\xi,x)\right]{\rm d}s(\xi)\qquad {\rm for}\;\;x\in\Omega_1.
	\end{eqnarray*}
\end{lemma}
\begin{proof}
	For $\alpha=D$ and any $x\in\Omega_1$, we choose a sufficient large $\rho>0$ and a sufficiently small $\varepsilon>0$ such that the disc $B_{\varepsilon}(x)$ with $x$ as the center and $\varepsilon$ as the radius contains in $\Omega_{\rho}:=B_{\rho}\cap\Omega_1$, where $B_{\rho}$ denotes the disc with the origin as the center and $\rho$ as the radius. We now apply Green's theorem to the functions $V_D(\cdot, x_s)$ and $G_D(\cdot,x)$ in the domain $\Omega_{\rho}\setminus B_{\varepsilon}(x)$ to obtain 
	\begin{eqnarray}\nonumber
		0&=&\int_{\Omega_{\rho}\setminus B_{\varepsilon}(x)}\left[\Delta V_D(\xi,x_s)G_D(\xi,x)-\Delta G_D(\xi,x)V_D(\xi,x_s)\right]{\rm d}\xi\\\nonumber
		&=&\left\{\int_{\partial B_{\rho}^+}-\int_{\Gamma\cap B_{\rho}}+\int_{\partial B_{\varepsilon}(x)}\right\}\left[\frac{\partial V_D(\xi,x_s)}{\partial \nu(\xi)}G_D(\xi,x)-\frac{\partial G_D(\xi,x)}{\partial\nu(\xi)}V_D(\xi,x_s)\right]{\rm d}s(\xi)\\\label{b19}
		&:=&I_1-I_2+I_3,
	\end{eqnarray}
	where $\nu(\xi)$ denotes the unit normal which directs into the exterior of $B_{\rho}$ for $\xi\in\partial B_{\rho}^+$, and directs into the interior of $B_{\varepsilon}(x)$ for $\xi\in \partial B_{\varepsilon}(x)$. For the item $I_1$, our task is to show 
	\begin{eqnarray}\label{b20}
		\lim_{\rho\to \infty} I_1=0.
	\end{eqnarray}
	To accomplish this, using the Sommerfeld radiation condition gives that 
	\begin{eqnarray}\nonumber
		&&\int_{\partial B_{\rho}^+}\left[\left|\frac{\partial V_D(\xi,x_s)}{\partial \nu(\xi)}\right|^2+\kappa_1^2|V_D(\xi,x_s)|^2+2\kappa_1{\rm Im}\left(V_D(\xi,x_s)\frac{\partial \overline{V_D(\xi,x_s)}}{\partial\nu(\xi)}\right)\right]{\rm d}s(\xi)\\\label{b21}
		&&=\int_{\partial B_{\rho}^+}\left|\frac{\partial V_D(\xi,x_s)}{\partial\nu(\xi)}-{\rm i}\kappa_1 V_D(\xi,x_s)\right|^2{\rm d}s(\xi)\to 0\quad{\rm as}\;\;\rho\to\infty.
	\end{eqnarray}
	A direct application of Green's theorem leads to 
	\begin{align}\nonumber
		&\int_{\partial B_{\rho}^+}V_D(\xi,x_s)\frac{\partial\overline{V_D(\xi,x_s)}}{\partial\nu(\xi)}{\rm d}s(\xi)\\\label{b22}
		&=\int_{B_{\rho}\cap\Gamma}V_D(\xi,x_s)\frac{\partial\overline{V_D(\xi,x_s)}}{\partial\nu(\xi)}{\rm d}s(\xi)+\int_{B_{\rho}\cap\Omega_1}\left[|\nabla V_D(\xi,x_s)|^2-\kappa_1^2|V_D(\xi,x_s)|^2\right]{\rm d}\xi.
	\end{align}
	We now insert the imaginary part of (\ref{b22}) into (\ref{b21}) and find that 
	\begin{eqnarray*}\label{b23}
		&&\lim_{\rho\to\infty}\int_{\partial B_{\rho}^+}\left(\left|\frac{\partial V_D(\xi,x_s)}{\partial \nu(\xi)}\right|^2+\kappa_1^2|V_D(\xi,x_s)|^2\right){\rm d}s(\xi)\\
		&&=-2\kappa_1{\rm Im}\int_{\Gamma\setminus \Gamma_R}V_D(\xi,x_s)\frac{\partial\overline{V_D(\xi,x_s)}}{\partial\nu(\xi)}{\rm d}s(\xi),
	\end{eqnarray*}
	where we use the fact $V_D(\xi,x_s)=0$ for $\xi\in\Gamma\cap\Gamma_R$. Thus, we conclude 
	\begin{eqnarray}\label{b24}
		\int_{\partial B_{\rho}^+}|V_D(\xi,x_s)|^2{\rm d}s=O(1),\quad \rho\to\infty.
	\end{eqnarray}
	Similarly, we have 
	\begin{eqnarray}\label{b25}
		\int_{\partial B_{\rho}^+}|G_D(\xi,x)|^2{\rm d}s=O(1),\quad \rho\to\infty.
	\end{eqnarray}
	The item $I_1$ can be rewritten as 
	\begin{eqnarray*}
		I_1&=&\int_{\partial B_{\rho}^+}\left[\frac{\partial V_D(\xi,x_s)}{\partial\nu(\xi)}-{\rm i}\kappa_1 V_D(\xi,x_s)\right]G_D(\xi,x){\rm d}s(\xi)\\
		&&-\int_{\partial B_{\rho}^+}\left[\frac{\partial G_D(\xi,x)}{\partial\nu(\xi)}-{\rm i}\kappa_1 G_D(\xi,x)\right]V_D(\xi,x_s){\rm d}s(\xi)
	\end{eqnarray*}
	which, combining (\ref{b24}), (\ref{b25}), the Sommerfeld radiation condition, and Cauchy-Schwartz inequality, shows that (\ref{b20}) holds true.
	
	For the item $I_3$, by $G_D(\xi,x)=G^s_D(\xi,x)+\Phi_{\kappa_1}(\xi,x)$, we have
	\begin{eqnarray}\nonumber
		I_3&=&\int_{\partial B_{\varepsilon}(x)}\left[\frac{\partial V_D(\xi,x_s)}{\partial \nu(\xi)}G^s_D(\xi,x)-\frac{\partial G^s_D(\xi,x)}{\partial\nu(\xi)}V_D(\xi,x_s)\right]{\rm d}s(\xi)\\\nonumber
		&&+\int_{\partial B_{\varepsilon}(x)}\left[\frac{\partial V_D(\xi,x_s)}{\partial \nu(\xi)}\Phi_{\kappa_1}(\xi,x)-\frac{\partial \Phi_{\kappa_1}(\xi,x)}{\partial\nu(\xi)}V_D(\xi,x_s)\right]{\rm d}s(\xi)\\\label{b26}
		&:=&I_{31}+I_{32}.
	\end{eqnarray}
	Applying the Green's theorem in $B_{\varepsilon}(x)$ shows that $I_{31}=0$. A straightforward calculation with using the mean value theorem shows that 
	\begin{eqnarray}\label{b27}
		\lim_{\varepsilon\to 0}I_{32}= -V_D(x,x_s).
	\end{eqnarray}
	Thus, from (\ref{b19}), (\ref{b20}),(\ref{b26}), and (\ref{b27}), we have the following Green's formula
	\begin{eqnarray*}\label{b28}
		V_D(x,x_s)=\int_{\Gamma\setminus\Gamma_R}\left[\frac{\partial G_D(\xi,x)}{\partial\nu(\xi)}V_D(\xi,x_s)-\frac{\partial V_D(\xi,x_s)}{\partial \nu(\xi)}G_D(\xi,x)\right]{\rm d}s(\xi).
	\end{eqnarray*}
	The proof is completed. 
\end{proof}

Now, we are in position to present the resolution result of the RTM method for recovering an impenetrable locally rough surface.
\begin{theorem}\label{thm1}
	For any $z\in S$, let $\psi_{\alpha}(\xi,z)$ solve 
	\begin{eqnarray}\label{b29}
		\left\{\begin{aligned}
			&\Delta \psi_{\alpha}(\xi,z)+\kappa_1^2\psi_{\alpha}(\xi,z)=0 \qquad\qquad\qquad\;\;\;{\rm in}\;\; \Omega_1, \\
			&\;\mathcal{B}\psi_{\alpha}(\xi,z)=-\mathcal{B}{\rm Im}G_{\alpha}(\xi,z)\qquad\qquad\qquad\;\;{\rm on}\;\; \Gamma\setminus\Gamma_R,\\
			&\;\mathcal{B}\psi_{\alpha}(\xi,z)=0\qquad\qquad\qquad\qquad\qquad\qquad\;{\rm on}\;\; \Gamma\cap\Gamma_R,\\
			&\lim_{|\xi|\rightarrow \infty}|\xi|^{\frac{1}{2}}\left(\partial_{|\xi|} \psi_{\alpha}(\xi,z)-{\rm i}\kappa_1 \psi_{\alpha}(\xi,z)\right)=0,
		\end{aligned}
		\right.
	\end{eqnarray}
	and $\psi_{\alpha}^{\infty}(\hat{\xi},z)$ be the far-field pattern of $\psi_{\alpha}(\xi,z)$. Then for the indicator function $\widetilde{{\rm Ind}}_{\alpha}(z)$, we have
	\begin{eqnarray*}\label{b30}
		\widetilde{{\rm Ind}}_{\alpha}(z)=\kappa_1\int_{{\mathbb S}_+}|\psi_{\alpha}^{\infty}(\hat{\xi},z)|^2{\rm d}s(\hat{\xi})+\eta_{\alpha}(z),\qquad \forall z\in S,
	\end{eqnarray*}
	where $\|\eta_{\alpha}(z)\|_{L^{\infty}(S)}\leq C(R_s^{-1}+R_r^{-1})$ with some constant $C$ depending on $R$.
\end{theorem}
\begin{proof}
	For the case $\alpha=D$, recall that 
	\begin{eqnarray}\nonumber
		\widetilde{{\rm Ind}}_D(z)&=&-\kappa_1^2{\rm Im}\int_{\Gamma_r}\int_{\Gamma_s}G_D(z,x_s)G_D(z,x_r)\overline{V_D(x_r,x_s)}{\rm d}s(x_s){\rm d}s(x_r)\\\label{b31}
		&=&-\kappa_1{\rm Im}\int_{\Gamma_s}G_D(z,x_s)\widetilde{W}_D(z,x_s){\rm d}s(x_s),
	\end{eqnarray}
	where 
	\begin{eqnarray}\label{b32}
		\widetilde{W}_D(z,x_s):=\kappa_1\int_{\Gamma_r}G_D(z,x_r)\overline{V_D(x_r,x_s)}{\rm d}s(x_r).
	\end{eqnarray}
	Substituting the Green's formula presented by Lemma \ref{lem3} into (\ref{b32}) and exchanging the order of integration leads to 
	\begin{eqnarray}\nonumber
		\widetilde{W}_D(z,x_s)&=&\int_{\Gamma\setminus\Gamma_R}\bigg\{\overline{V_D(\xi,x_s)}\frac{\partial}{\partial\nu(\xi)}\left[\kappa_1\int_{\Gamma_r}G_D(z,x_r)\overline{G_D(\xi,x_r)}{\rm d}s(x_r)\right]\\\nonumber
		&&\qquad\qquad-\left[\kappa_1\int_{\Gamma_r}G_D(z,x_r)\overline{G_D(\xi,x_r)}ds(x_r)\right]\frac{\partial\overline{V_D(\xi,x_s)}}{\partial\nu(\xi)}\bigg\}{\rm d}s(\xi)\\\nonumber
		&=&\int_{\Gamma\setminus\Gamma_R}\bigg\{\overline{V_D(\xi,x_s)}\frac{\partial}{\partial\nu(\xi)}\left[{\rm Im}G_D(\xi,z)+\zeta_{D,r}(\xi,z)\right]\\\label{b33}
		&&\qquad\qquad-\left[{\rm Im}G_D(\xi,z)+\zeta_{D,r}(\xi,z)\right]\frac{\partial\overline{V_D(\xi,x_s)}}{\partial\nu(\xi)}\bigg\}{\rm d}s(\xi).
	\end{eqnarray}
	Substituting (\ref{b33}) into (\ref{b31}) gives 
	\begin{eqnarray*}\nonumber
		&&\widetilde{{\rm Ind}}_D(z)\\
		&&=-{\rm Im}\int_{\Gamma\setminus\Gamma_R}\bigg\{\left[\kappa_1\int_{\Gamma_s}G_D(z,x_s)\overline{V_D(\xi,x_s)}{\rm d}s(x_s)\right]\frac{\partial}{\partial\nu(\xi)}\left[{\rm Im}G_D(\xi,z)+\zeta_{D,r}(\xi,z)\right]\\\nonumber
		&&\qquad-\left[{\rm Im}G_D(\xi,z)+\zeta_{D,r}(\xi,z)\right]\frac{\partial}{\partial\nu(\xi)}\left[\kappa_1\int_{\Gamma_s}G_D(z,x_s)\overline{V_D(\xi,x_s)}{\rm d}s(x_s)\right]\bigg\}{\rm d}s(\xi)\\\nonumber
		&&=-{\rm Im}\int_{\Gamma\setminus\Gamma_R}\bigg\{\phi_D(\xi,z)\frac{\partial}{\partial\nu(\xi)}\left[{\rm Im}G_D(\xi,z)+\zeta_{D,r}(\xi,z)\right]\\\label{b34}
		&&\qquad\qquad-\left[{\rm Im}G_D(\xi,z)+\zeta_{D,r}(\xi,z)\right]\frac{\partial \phi_D(\xi,z)}{\partial\nu(\xi)}\bigg\}{\rm d}s(\xi),
	\end{eqnarray*}
	where $\phi_D(\xi,z)$ is defined by
	\begin{eqnarray*}\label{b35}
		\phi_D(\xi,z):=\kappa_1\int_{\Gamma_s}G_D(z,x_s)\overline{V_D(\xi,x_s)}{\rm d}s(x_s).
	\end{eqnarray*}
	Since $V_D(\xi,x_s)$ satisfies Problem (\ref{b9}), it follows from the linearity that $\overline{\phi_D}(\xi,z)$ solves the following problem
	\begin{eqnarray}\label{b36}
		\left\{\begin{aligned}
			&\Delta \overline{\phi_D}(\xi,z)+\kappa_1^2 \overline{\phi_D}(\xi,z)=0 \qquad\qquad\qquad\;\;{\rm in}\;\; \Omega_1, \\
			&\; \overline{\phi_D}(\xi,z)=-{\rm Im}G_D(\xi,z)-\zeta_{D,s}(\xi,z)\qquad\;\;{\rm on}\;\; \Gamma\setminus\Gamma_R,\\
			&\; \overline{\phi_D}(\xi,z)=0\qquad\qquad\qquad\qquad\qquad\qquad\;\;\;\;{\rm on}\;\; \Gamma\cap\Gamma_R,\\
			&\lim_{|\xi|\rightarrow \infty}|\xi|^{\frac{1}{2}}\left(\partial_{|\xi|}  \overline{\phi_D}(\xi,z)-{\rm i}\kappa_1  \overline{\phi_D}(\xi,z)\right)=0,
		\end{aligned}
		\right.
	\end{eqnarray}
	where we use Lemma \ref{lem2} to derive the boundary condition on $\Gamma\setminus\Gamma_R$. By the linearity, it is easy to obtain the following decomposition 
	\begin{eqnarray*}\label{b37}
		\overline{\phi_D}(\xi,z)=\psi_D(\xi,z)+\varphi_D(\xi,z),
	\end{eqnarray*}
	where $\psi_D(\xi,z)$ and $\varphi_D(\xi,z)$ solve Problem (\ref{b36}) with the boundary data $-{\rm Im} G_D(\xi,z)$ and $-\zeta_{D,s}(\xi,z)$ on $\Gamma\setminus\Gamma_R$, respectively. Thus, we conclude 
	\begin{eqnarray}\nonumber
		&&\widetilde{{\rm Ind}}_D(z)=-{\rm Im}\int_{\Gamma\setminus\Gamma_R}\bigg\{\left[\overline{\psi_D}(\xi,z)+\overline{\varphi_D}(\xi,z)\right]\frac{\partial}{\partial\nu(\xi)}\left[{\rm Im}G_D(\xi,z)+\zeta_{D,r}(\xi,z)\right]\\\nonumber
		&&\qquad\qquad-\left[{\rm Im}G_D(\xi,z)+\zeta_{D,r}(\xi,z)\right]\frac{\partial }{\partial\nu(\xi)}\left[\overline{\psi_D}(\xi,z)+\overline{\varphi_D}(\xi,z)\right]\bigg\}{\rm d}s(\xi)\\\nonumber
		&&=-{\rm Im}\int_{\Gamma\setminus\Gamma_R}\left[\overline{\psi_D}(\xi,z)\frac{\partial {\rm Im}G_D(\xi,z)}{\partial\nu(\xi)}-{\rm Im}G_D(\xi,z)\frac{\partial\overline{\psi_D}(\xi,z)}{\partial\nu(\xi)}\right]{\rm d}s(\xi)+\eta_D(z)\\\nonumber
		&&=-{\rm Im}\int_{\Gamma\setminus\Gamma_R}\left[-{\rm Im}G_D(\xi,z)\frac{\partial {\rm Im}G_D(\xi,z)}{\partial\nu(\xi)}+\psi_D(\xi,z)\frac{\partial\overline{\psi_D}(\xi,z)}{\partial\nu(\xi)}\right]{\rm d}s(\xi)+\eta_D(z)\\\label{b38}
		&&=-{\rm Im}\int_{\Gamma\setminus\Gamma_R}\psi_D(\xi,z)\frac{\partial\overline{\psi_D}(\xi,z)}{\partial\nu(\xi)}{\rm d}s(\xi)+\eta_D(z).
	\end{eqnarray}
	Here, $\eta_D(z)$ is defined by 
	\begin{eqnarray}\nonumber
		&&\eta_D(z)=-{\rm Im}\int_{\Gamma\setminus\Gamma_R}\left[\overline{\psi_D}(\xi,z)\frac{\partial \zeta_{D,r}(\xi,z)}{\partial\nu(\xi)}-\zeta_{D,r}(\xi,z)\frac{\partial\overline{\psi_D}(\xi,z)}{\partial\nu(\xi)}\right]{\rm d}s(\xi)\\\nonumber
		&&-{\rm Im}\int_{\Gamma\setminus\Gamma_R}\left[\overline{\varphi_D}(\xi,z)\frac{\partial \zeta_{D,r}(\xi,z)}{\partial\nu(\xi)}-\zeta_{D,r}(\xi,z)\frac{\partial\overline{\varphi_D}(\xi,z)}{\partial\nu(\xi)}\right]{\rm d}s(\xi)\\\label{b39}
		&&-{\rm Im}\int_{\Gamma\setminus\Gamma_R}\left[\overline{\varphi_D}(\xi,z)\frac{\partial {\rm Im}G_D(\xi,z)}{\partial\nu(\xi)}-{\rm Im}G_D(\xi,z)\frac{\partial\overline{\varphi_D}(\xi,z)}{\partial\nu(\xi)}\right]{\rm d}s(\xi).
	\end{eqnarray}
	Applying Green's theorem to the functions $\psi_D(\xi,z)$ and $\overline{\psi_D}(\xi,z)$ in the domain $\Omega\cap B_{\rho}$ yields 
	\begin{align}\nonumber
		&-{\rm Im}\int_{\Gamma\setminus\Gamma_R}\psi_D(\xi,z)\frac{\partial\overline{\psi_D}(\xi,z)}{\partial\nu(\xi)}{\rm d}s(\xi)
		=-{\rm Im}\int_{\partial B^+_{\rho}}\psi_D(\xi,z)\frac{\partial\overline{\psi_D}(\xi,z)}{\partial\nu(\xi)}{\rm d}s(\xi)\\\nonumber &+{\rm Im}\int_{\Omega\cap B_{\rho}}\left[\psi_D(\xi,z)\Delta\overline{ \psi_D}(\xi,z)+|\nabla\psi_D(\xi,z)|^2\right]{\rm d}\xi\\\label{b40}
		&=-{\rm Im}\int_{\partial B^+_{\rho}}\psi_D(\xi,z)\left[\frac{\partial\overline{\psi_D}(\xi,z)}{\partial\nu(\xi)}+{\rm i}\kappa_1\overline{\psi_D}(\xi,z)\right]{\rm d}s(\xi)+{\rm Im}\int_{\partial B^+_{\rho}}{\rm i}\kappa_1 |\psi_D(\xi,z)|^2{\rm d}s(\xi)
	\end{align}
	Note that 
	\begin{eqnarray*}
		\psi_D(\xi,z)=O(|\xi|^{-\frac{1}{2}})\quad{\rm and}\quad \frac{\partial\psi_D(\xi,z)}{\partial\nu(\xi)}-{\rm i}\kappa_1\psi_D(\xi,z)=o(|\xi|^{-\frac{1}{2}}),
	\end{eqnarray*}
	we have
	\begin{eqnarray*}\label{b41}
		\lim_{\rho\to\infty}\int_{\partial B^+_{\rho}}\psi_D(\xi,z)\left[\frac{\partial\overline{\psi_D}(\xi,z)}{\partial\nu(\xi)}+{\rm i}\kappa_1\overline{\psi_D}(\xi,z)\right]{\rm d}s(\xi)=0.
	\end{eqnarray*}
	Since $\psi_D(\xi,z)$ satisfies the Sommerfeld radiation condition, it admits the following asymptotic behavior 
	\begin{eqnarray*}\label{b42}
		\psi_D(\xi,z)=\frac{e^{{\rm i}\kappa_1 |\xi|}}{|\xi|^{\frac{1}{2}}}\left[\psi^{\infty}_D(\hat{\xi},z)+O(|\xi|^{-1})\right]
	\end{eqnarray*}
	which implies 
	\begin{eqnarray}\label{b43}
		\lim_{\rho\to\infty}\int_{\partial B_{\rho}^+}\kappa_1 |\psi_D(\xi,z)|^2{\rm d}s(\xi)=\kappa_1\int_{{\mathbb S}_+}|\psi_D^{\infty}(\hat{\xi},z)|^2{\rm d}s(\hat{\xi}).
	\end{eqnarray}
	Hence, with the help of (\ref{b38}) and (\ref{b40})--(\ref{b43}), we arrive at 
	\begin{eqnarray*}\label{b44}
		\widetilde{{\rm Ind}}_D(z)=\kappa_1\int_{{\mathbb S}_+}|\psi_D^{\infty}(\hat{\xi},z)|^2{\rm d}s(\hat{\xi})+\eta_D(z).
	\end{eqnarray*}
	
	The remaining part of the proof is the estimate of $\eta_D(z)$. For $\xi\in\Gamma\setminus\Gamma_R$ and $z\in S$, it follows from Lemma  \ref{lem2} that 
	\begin{eqnarray}\label{b45}
		|\zeta_{D,r}(\xi,z)|+|\nabla_{\xi}\zeta_{D,r}(\xi,z)|\leq CR_r^{-1}.
	\end{eqnarray}
	Observing that 
	\begin{eqnarray*}
		{\rm Im} G_D(\xi,z)={\rm Im}G^s_D(\xi,z)+\frac{1}{4}J_0(\kappa_1 |\xi-z|),
	\end{eqnarray*}
	where $J_0$ stands for the Bessel function of order zero, it follows from the smoothness of $G_D^s(\xi,z)$ and $J_0(\kappa_1 |\xi-z|)$ that 
	\begin{eqnarray}\label{b46}
		\left|{\rm Im} G_D(\xi,z)\right|+\left|\frac{\partial {\rm Im}G_D(\xi,z)}{\partial\nu(\xi)}\right|\leq C.
	\end{eqnarray}
	Since $\psi_D(\xi,z)$ and $\varphi_D(\xi,z)$ solve Problem (\ref{b36}) with the boundary data $-{\rm Im} G_D(\xi,z)$ and $-\zeta_{D,s}(\xi,z)$ on $\Gamma\setminus\Gamma_R$, respectively, a direct application of the well-posedness of Problem (\ref{b29}) (cf. \cite[Theorem 2.1]{DLLY17}) and the trace theorem shows that 
	\begin{align}\label{b47}
		\|\psi_D(\xi,z)\|_{H^{\frac{1}{2}}(\Gamma\setminus\Gamma_R)}+\left\|\frac{\partial\psi_D(\xi,z)}{\partial\nu(\xi)}\right\|_{H^{-\frac{1}{2}}(\Gamma\setminus\Gamma_R)}
		\lesssim \|{\rm Im}G_D(\xi,z)\|_{H^{\frac{1}{2}}(\Gamma\setminus\Gamma_R)}
		\leq C,
	\end{align}
and
\begin{align}\label{b48}
\|\varphi_D(\xi,z)\|_{H^{\frac{1}{2}}(\Gamma\setminus\Gamma_R)}+\left\|\frac{\partial\varphi_D(\xi,z)}{\partial\nu(\xi)}\right\|_{H^{-\frac{1}{2}}(\Gamma\setminus\Gamma_R)}\lesssim \|\zeta_{D,s}(\xi,z)\|_{H^{\frac{1}{2}}(\Gamma\setminus\Gamma_R)}\lesssim R_s^{-1},
\end{align}
	where the notation $a\lesssim b$ means $a\leq Cb$ for some generic constant $C>0$, which may change step by step.
	Thus, with the aid of (\ref{b39}) and  (\ref{b45})--(\ref{b48}), we can easily obtain
	\begin{eqnarray*}
		\|\eta_D(z)\|_{L^{\infty}(S)}\leq C(R_r^{-1}+R_s^{-1}),
	\end{eqnarray*} 
	with $C$ depending on $R$. The proof is finished.
\end{proof}

\subsection{The far-field reconstruction}
This subsection is devoted to the RTM method with the far-field measurement. To this 
end, we need to establish a mixed reciprocity relation. Let $w^i(x,d)=e^{{\rm i}\kappa_1x\cdot d}$ with $d=(d_1,d_2)^{\rm \top}\in{\mathbb S}_-:=\{x\in{\mathbb R}^2: |x|=1,x_2<0\}$ be the plane wave, then the reflected wave of $w^i$ by the infinite plane $\Gamma_0$ is given by $w_{\alpha}^r(x,d)=-e^{{\rm i}\kappa_1x\cdot d^r}$ for $\alpha=D$ and $w_{\alpha}^r(x,d)=e^{{\rm i}\kappa_1x\cdot d^r}$ for $\alpha=N$ with $d^r=(d_1,-d_2)^{\rm \top}$. We define 
\begin{equation*}\label{b49}
w_{0,\alpha}(x,d):=w^i(x,d)+w_{\alpha}^r(x,d)=\left\{\begin{array}{l}
              e^{{\rm i}\kappa_1x\cdot d}-e^{{\rm i}\kappa_1x\cdot d^r} \qquad\qquad\; \textrm{for}\;\;\; \alpha=D, \\[1mm]
              e^{{\rm i}\kappa_1x\cdot d}+e^{{\rm i}\kappa_1x\cdot d^r} \qquad\qquad\; \textrm{for}\;\; \;\alpha=N,
            \end{array}
\right.
\end{equation*}
which satisfies $w_{0,D}(x,d)=0$ on $\Gamma_0$ and $\partial_{\nu}w_{0,N}(x,d)=0$ on $\Gamma_0$. Then the 
scattering of $w_{0,\alpha}(x,d)$ by the locally rough surface $\Gamma$ can be modelled by
\begin{equation*}\label{b50}
	\left\{\begin{aligned}
		&\Delta w_{\alpha}^s(x,d)+\kappa_1^2w_{\alpha}^s(x,d)=0 \qquad\qquad\qquad\;\textrm{in}\;\; \Omega_1, \\
		&\mathcal{B}w_{\alpha}^s(x,d)=-\mathcal{B}w_{0,\alpha}(x,d)\qquad\qquad\qquad\;\;\;\;\textrm{on}\;\; \Gamma,\\
		&\lim_{|x|\rightarrow \infty}|x|^{\frac{1}{2}}\left(\partial_{|x|} w_{\alpha}^s(x,d)-{\rm i}\kappa_1 w_{\alpha}^s(x,d)\right)=0,\quad
	\end{aligned}
	\right.
\end{equation*}
where $w_{\alpha}^s$ denotes the scattered field and the Sommerfeld radiation condition holds uniformly for all directions $\hat{x}\in{\mathbb S}_+$. 

Let $v_{0,\alpha}(x,x_s)$ be the Dirichlet or Neumann Green's function with respect to $\Gamma_0$, which is given by 
\begin{equation*}\label{b51}
v_{0, \alpha}(x,x_s)=\left\{\begin{array}{l}
              \Phi(x, x_s)-\Phi(x,x_s') \qquad\qquad\; \textrm{for}\;\; \alpha=D, \\[1mm]
               \Phi(x, x_s)+\Phi(x,x_s') \qquad\qquad\; \textrm{for}\;\; \alpha=N. 
            \end{array}
\right.
\end{equation*}
Here, $x_s'$ is the image point of $x_s$ with respect to $\Gamma_0$. We have $v_{0, D}(x,x_s)=0$ on $\Gamma_0$ and $\partial_{\nu}v_{0, N}(x,x_s)=0$ on $\Gamma_0$. Let $v_{0,\alpha}^{\infty}(\hat{x},x_s)$ be the far-field of $v_{0,\alpha}(x,x_s)$, it is easy to see that 
\begin{eqnarray}\label{b52}
v_{0,\alpha}^{\infty}(\hat{x},x_s)=\gamma_1w_{0,\alpha}(x_s,-\hat{x})\qquad {\rm for}\;\; \hat{x}\in{\mathbb S}_+, x_s\in {\mathbb R}^2_+,
\end{eqnarray}
where $\gamma_1:=\frac{e^{\frac{\pi}{4}{\rm i}}}{\sqrt{8\kappa_1\pi}}$. Define $v_{\alpha}^s(x,x_s):=u_{\alpha}^s(x,x_s)-v_{0,\alpha}(x,x_s)+\Phi(x,x_s)$ where $u_{\alpha}^s(x,x_s)$ is the scattered field of Problem (\ref{a5}), then we have $v_{\alpha}^s$ solves
\begin{equation*}\label{bb52}
	\left\{\begin{aligned}
		&\Delta v_{\alpha}^s(x,x_s)+\kappa_1^2v_{\alpha}^s(x,x_s)=0 \qquad\qquad\qquad\;\;\textrm{in}\;\; \Omega_1, \\[1mm]
		&\mathcal{B}v_{\alpha}^s(x,x_s)=-\mathcal{B}v_{0,\alpha}(x,x_s)\qquad\qquad\qquad\;\;\;\;\;\textrm{on}\;\; \Gamma,\\[1mm]
		&\lim_{|x|\rightarrow \infty}|x|^{\frac{1}{2}}\left(\partial_{|x|} v_{\alpha}^s(x,x_s)-{\rm i}\kappa_1 v_{\alpha}^s(x,x_s)\right)=0,\quad
	\end{aligned}
	\right.
\end{equation*}
where the Sommerfeld radiation condition holds uniformly for all directions $\hat{x}\in{\mathbb S}_+$.

\begin{theorem}
For acoustic scattering of plane waves $w_{0,\alpha}(\cdot, -\hat{x})$, $\hat{x}\in{\mathbb S}_+$ and point sources $v_{0,\alpha}(\cdot, x_s)$, $x_s\in\Omega_1$ from a locally rough surface $\Gamma$ we have
\begin{eqnarray*}
v_{\alpha}^{\infty}(\hat{x}, x_s)=\gamma_1 w_{\alpha}^s(x_s, -\hat{x})
\end{eqnarray*}
for $\alpha\in\{D,N\}$, $\hat{x}\in{\mathbb S}_+, x_s\in \Omega_1$ and $x_s'\notin \Omega_1$.
\end{theorem}
\begin{proof}
Let $D_1:=\Omega_2\cap{\mathbb R}^2_+$ and $D_2:=\Omega_1\cap{\mathbb R}^2_-$. For $\alpha = D$, $x_s\in\Omega_1$ and $x_s'\notin\Omega_1$, a direct application of Green formula yields  
\begin{eqnarray}\nonumber
0&=&\int_{D_1}\left[v_{0,D}(\xi,x_s)\Delta w_{0,D}(\xi, -\hat{x})-w_{0,D}(\xi, -\hat{x})\Delta v_{0,D}(\xi,x_s) \right]{\rm d}\xi\\\label{b53}
&=&\int_{\partial D_1\setminus\Gamma_0}\left[v_{0,D}(\xi,x_s)\frac{\partial w_{0,D}(\xi, -\hat{x})}{\partial\nu(\xi)}-w_{0,D}(\xi, -\hat{x})\frac{\partial v_{0,D}(\xi,x_s)}{\partial\nu(\xi)}\right]{\rm d}s(\xi)
\end{eqnarray}
and 
\begin{eqnarray}\label{b54}
\int_{\partial D_2\setminus\Gamma_0}\left[v_{0,D}(\xi,x_s)\frac{\partial w_{0,D}(\xi, -\hat{x})}{\partial\nu(\xi)}-w_{0,D}(\xi, -\hat{x})\frac{\partial v_{0,D}(\xi,x_s)}{\partial\nu(\xi)}\right]{\rm d}s(\xi)=0.
\end{eqnarray}
Combining (\ref{b53}), (\ref{b54}) and the fact $v_{0,D}$ and $w_{0,D}$ vanish on $\Gamma_0$ gives that 
\begin{eqnarray}\label{b55}
\int_{
\Gamma}\left[v_{0,D}(\xi,x_s)\frac{\partial w_{0,D}(\xi, -\hat{x})}{\partial\nu(\xi)}-w_{0,D}(\xi, -\hat{x})\frac{\partial v_{0,D}(\xi,x_s)}{\partial\nu(\xi)}\right]{\rm d}s(\xi)=0.
\end{eqnarray}
It follows from the Green's formula, the Sommerfeld radiation condition and the Cauchy-Schwartz inequality that 
\begin{eqnarray}\label{b56}
\int_{
\Gamma}\left[v_{D}^s(\xi,x_s)\frac{\partial w^s_{D}(\xi, -\hat{x})}{\partial\nu(\xi)}-w^s_{D}(\xi, -\hat{x})\frac{\partial v^s_{D}(\xi,x_s)}{\partial\nu(\xi)}\right]{\rm d}s(\xi)=0.
\end{eqnarray}
For $x\in\Omega_1$, we choose a sufficient large $\rho>0$ and a sufficient small $\varepsilon>0$ such that $B_{\varepsilon}(x)\subset\Omega_{\rho}:=B_{\rho}\cap\Omega_1$. By the Green's formula, we obtain
\begin{eqnarray*}\nonumber
0&=&\int_{\Omega_{\rho}\setminus\overline{B_{\varepsilon}(x)}}\left[v_{0,D}(\xi,x)\Delta v^s_{D}(\xi, x_s)-v^s_{D}(\xi, x_s)\Delta v_{0,D}(\xi,x) \right]{\rm d}\xi\\\label{b57}
&=&\left(\int_{\partial B_{\rho}^+}+\int_{\partial B_{\varepsilon}(x)}-\int_{\Gamma\cap B_{\rho}}\right)\left[v_{0,D}(\xi,x)\frac{\partial v_{D}^s(\xi, x_s)}{\partial\nu(\xi)}-v_{D}^s(\xi, x_s)\frac{\partial v_{0,D}(\xi,x)}{\partial\nu(\xi)}\right]{\rm d}s(\xi)\\
&=&I_1+I_2-I_3.
\end{eqnarray*}
By a similar argument with (\ref{b20}) and (\ref{b27}), we have
\begin{eqnarray*}\label{b58}
\lim_{\rho\to\infty}I_1=0\qquad{\rm and}\qquad \lim_{\varepsilon\to 0}I_2=-v_D^s(x,x_s).
\end{eqnarray*}
Hence, we obtain
\begin{eqnarray*}\label{b59}
v_D^s(x,x_s)=\int_{\Gamma}\left[v_D^s(\xi,x_s)\frac{\partial v_{0,D}(\xi,x)}{\partial\nu(\xi)}-v_{0,D}(\xi,x)\frac{\partial v_D^s(\xi,x_s)}{\partial\nu(\xi)}\right]{\rm d}s(\xi),
\end{eqnarray*}
which implies 
\begin{align}\label{b60}
v_D^{\infty}(\hat{x},x_s)=\gamma_1\int_{\Gamma}\left[v_D^s(\xi,x_s)\frac{\partial w_{0,D}(\xi,-\hat{x})}{\partial\nu(\xi)}-w_{0,D}(\xi,-\hat{x})\frac{\partial v_D^s(\xi,x_s)}{\partial\nu(\xi)}\right]{\rm d}s(\xi),
\end{align}
where we use (\ref{b52}). Similarly, we have
\begin{align}\label{b61}
w_D^s(x_s,-\hat{x})=\int_{\Gamma}\left[w_D^s(\xi,-\hat{x})\frac{\partial v_{0,D}(\xi,x_s)}{\partial\nu(\xi)}-v_{0,D}(\xi,x_s)\frac{\partial w_D^s(\xi,-\hat{x})}{\partial\nu(\xi)}\right]{\rm d}s(\xi),
\end{align}
Combining (\ref{b55}),(\ref{b56}),(\ref{b60}) and (\ref{b61}) leads to
\begin{eqnarray*}\label{b62}
&&\gamma_1^{-1}v_D^{\infty}(\hat{x},x_s)-w_D^s(x_s,-\hat{x})\\
&&=\int_{\Gamma}\left[v_D(\xi,x_s)\frac{\partial w_{D}(\xi,-\hat{x})}{\partial\nu(\xi)}-w_{D}(\xi,-\hat{x})\frac{\partial v_D(\xi,x_s)}{\partial\nu(\xi)}\right]{\rm d}s(\xi)\\
&&=0,
\end{eqnarray*}
where the fact has been used that the total fields $v_D(\xi,x_s):=v_{0,D}(\xi,x_s)+v_D^s(\xi,x_s)$ and $w_D(\xi,-\hat{x}):=w_{0,\alpha}(\xi,-\hat{x})+w_D^s(\xi,-\hat{x})$ vanish on $\Gamma$. The proof is thus complete.
\end{proof}

With the above mixed reciprocity relation, we are in a position to present the RTM method based on far-field measurements. To this end, we first introduce some notations. Note that the wave fields $w_{\alpha}^s(x,d)$, $w_{\alpha}^{\infty}(\hat{x},d)$, $w_{\alpha}(x,d)$, $v_{\alpha}^s(x,x_s)$, $v_{\alpha}^{\infty}(\hat{x},x_s)$, $v_{\alpha}(x,x_s)$ depend on the rough surface $\Gamma$. For clarity, we write 
$w_{\alpha}^s(x,d,\Gamma)$, $w_{\alpha}^{\infty}(\hat{x},d,\Gamma)$, $w_{\alpha}(x,d,\Gamma)$, $v_{\alpha}^s(x,x_s,\Gamma)$, $v_{\alpha}^{\infty}(\hat{x},x_s,\Gamma)$, $v_{\alpha}(x,x_s,\Gamma)$ to express explicitly the dependence of the wave fields on $\Gamma$.

\begin{theorem}\label{thm26}
For the indicator function $\widetilde{{\rm Ind}}_{\alpha}(z)$ with $\alpha\in\{D,N\}$, we have the following limit identity
\begin{eqnarray}\nonumber
&&\lim_{R_s\to\infty}\lim_{R_r\to\infty}\widetilde{{\rm Ind}}_{\alpha}(z)=-\kappa_1^2|\gamma_1|^2{\rm Im} \bigg\{\gamma_1\int_{{\mathbb S_+}}\int_{{\mathbb S_+}}w_{\alpha}(z,-\hat{x}_r,\Gamma_R)w_{\alpha}(z,-\hat{x}_s,\Gamma_R)\\\label{b63}
&&\qquad\times\left[\overline{w_{\alpha}^{\infty}(\hat{x}_s,-\hat{x}_r,\Gamma)}-\overline{w_{\alpha}^{\infty}(\hat{x}_s,-\hat{x}_r,\Gamma_R)}\right]{\rm d}s(\hat{x}_r){\rm d}s(\hat{x}_s)\bigg\}:=\widehat{{\rm Ind}}_{\alpha}(z).
\end{eqnarray}
\end{theorem}
\begin{proof}
For $\alpha=D$ and $x\in\{x_r,x_s\}$, it follows from the Sommerfeld radiation condition, the mixed reciprocity relation, and (\ref{b52}) that 
\begin{eqnarray}\nonumber
G_D(z,x)&=&v_{0,D}(x,z)+v_D^s(x,z,\Gamma_R)\\\nonumber
&=&\frac{e^{{\rm i}\kappa_1|x|}}{|x|^{\frac{1}{2}}}\left[v_{0,D}^{\infty}(\hat{x},z)+v_D^{\infty}(\hat{x},z,\Gamma_R)+O\left(\frac{1}{|x|}\right)\right]\\\nonumber
&=&\gamma_1\frac{e^{{\rm i}\kappa_1|x|}}{|x|^{\frac{1}{2}}}\left[w_{0,D}(z,-\hat{x})+w_D^s(z, -\hat{x},\Gamma_R)+O\left(\frac{1}{|x|}\right)\right]\\\label{b64}
&=&\gamma_1\frac{e^{{\rm i}\kappa_1|x|}}{|x|^{\frac{1}{2}}}\left[w_D(z, -\hat{x},\Gamma_R)+O\left(\frac{1}{|x|}\right)\right]
\end{eqnarray}
where we use the reciprocity $G_D(z,x)=G_D(x,z)$ for $x,z\in\Omega_{1,R}$. For the data $V_D(x_r,x_s)$, we have
\begin{eqnarray}\nonumber
V_D(x_r,x_s)&=&u_D^s(x_r,x_s,\Gamma)-G_D^s(x_r,x_s,\Gamma_R)\\\nonumber
&=&v_D^s(x_r,x_s,\Gamma)-v_D^s(x_r,x_s,\Gamma_R)\\\nonumber
&=&\frac{e^{{\rm i}\kappa_1|x_r|}}{|x_r|^{\frac{1}{2}}}\left[v_{D}^{\infty}(\hat{x}_r,x_s,\Gamma)-v_D^{\infty}(\hat{x}_r,x_s,\Gamma_R)+O\left(\frac{1}{|x_r|}\right)\right]\\\label{b65}
&=&\gamma_1\frac{e^{{\rm i}\kappa_1|x_r|}}{|x_r|^{\frac{1}{2}}}\left[w_{D}^s(x_s,-\hat{x}_r,\Gamma)-w_D^s(x_s, -\hat{x}_r,\Gamma_R)+O\left(\frac{1}{|x_r|}\right)\right].
\end{eqnarray}
By (\ref{b64}) and (\ref{b65}), we obtain
\begin{eqnarray}\nonumber
&&\lim_{R_r\to\infty}\int_{\Gamma_r}G_D(z,x_r)\overline{V_D(x_r,x_s)}{\rm d}s(x_r)\\\nonumber
&&=|\gamma_1|^2\lim_{R_r\to\infty}\int_{\Gamma_r}\frac{1}{|x_r|}\left[w_D(z, -\hat{x}_r,\Gamma_R)+O\left(\frac{1}{|x_r|}\right)\right]\\\nonumber
&&\times\left[\overline{w_{D}^s(x_s,-\hat{x}_r,\Gamma)}-\overline{w_D^s(x_s, -\hat{x}_r,\Gamma_R)}+O\left(\frac{1}{|x_r|}\right)\right]{\rm d}s(x_r)\\\label{b66}
&&=|\gamma_1|^2\int_{{\mathbb S}_+}w_D(z, -\hat{x}_r,\Gamma_R)\left[\overline{w_{D}^s(x_s,-\hat{x}_r,\Gamma)}-\overline{w_D^s(x_s, -\hat{x}_r,\Gamma_R)}\right]{\rm d}s(\hat{x}_r)
\end{eqnarray}
Since 
\begin{eqnarray*}\nonumber
&&w_{D}^s(x_s,-\hat{x}_r,\Gamma)-w_D^s(x_s, -\hat{x}_r,\Gamma_R)\\\label{b67}
&&=\frac{e^{{\rm i}\kappa_1|x_s|}}{|x_s|^{\frac{1}{2}}}\left[w_D^{\infty}(\hat{x}_s,-\hat{x}_r,\Gamma)-w_D^{\infty}(\hat{x}_s,-\hat{x}_r,\Gamma_R)+O\left(\frac{1}{|x_s|}\right)\right],
\end{eqnarray*}
combining with (\ref{b64}) yields that 
\begin{eqnarray}\nonumber
&&\lim_{R_s\to\infty}\int_{\Gamma_s}G_D(z,x_s)\left[\overline{w_D^s(x_s,-\hat{x}_r,\Gamma)}-\overline{w_D^s(x_s,-\hat{x}_r,\Gamma_R)}\right]{\rm d}s(x_s)\\\nonumber
&&=\gamma_1\lim_{R_s\to\infty}\int_{\Gamma_s}\frac{1}{|x_s|}\left[w_D(z, -\hat{x}_s,\Gamma_R)+O\left(\frac{1}{|x_s|}\right)\right]\\\nonumber
&&\quad\times\left[\overline{w_{D}^{\infty}(\hat{x}_s,-\hat{x}_r,\Gamma)}-\overline{w_D^{\infty}(\hat{x}_s, -\hat{x}_r,\Gamma_R)}+O\left(\frac{1}{|x_s|}\right)\right]{\rm d}s(x_s)\\\label{b68}
&&=\gamma_1\int_{{\mathbb S}_+}w_D(z, -\hat{x}_s,\Gamma_R)\left[\overline{w_{D}^{\infty}(\hat{x}_s,-\hat{x}_r,\Gamma)}-\overline{w_D^{\infty}(\hat{x}_s, -\hat{x}_r,\Gamma_R)}\right]{\rm d}s(\hat{x}_s).
\end{eqnarray}
A direct application of (\ref{b66}) and (\ref{b68}) shows that the limit identity (\ref{b63}) holds for $\alpha=D$. This completes the proof.
\end{proof}

\section{The RTM for penetrable locally rough interfaces}
\setcounter{equation}{0}
The destination of this section is to develop the RTM method for penetrable, locally rough interfaces. This section consists of two subsections. In the first subsection, we will introduce the RTM method based on near-field associated with point sources incidence, and in the second subsection we will present a mixed reciprocity relation which leads to the RTM method based on far-field data which corresponds to the plane wave incidence.

\subsection{The near-filed reconstruction}
As shown in Figure \ref{f2}, 
let $\Gamma$ be the locally rough interface and $S$ be the sampling domain which contains the local perturbation of $\Gamma$. We choose a sufficient large $R$ and define $\Gamma_R$ by (\ref{b1}) so that the sampling domain $S$ lies totally above $\Gamma_R$. We suppose that there are $N_s$ point sources $x_s$ uniformly distributed on $\Gamma_s$ and $N_r$ receivers $x_r$ uniformly distributed on $\Gamma_r$. Here, $\Gamma_s$ and $\Gamma_r$ denote the circle with the origin as the center and $R_s$, $R_r$ as the radius, respectively. We suppose a priori that $R<R_s\leq R_r$.

\begin{figure}[htbp]
	\centering
	\includegraphics[width=5in, height=3in]{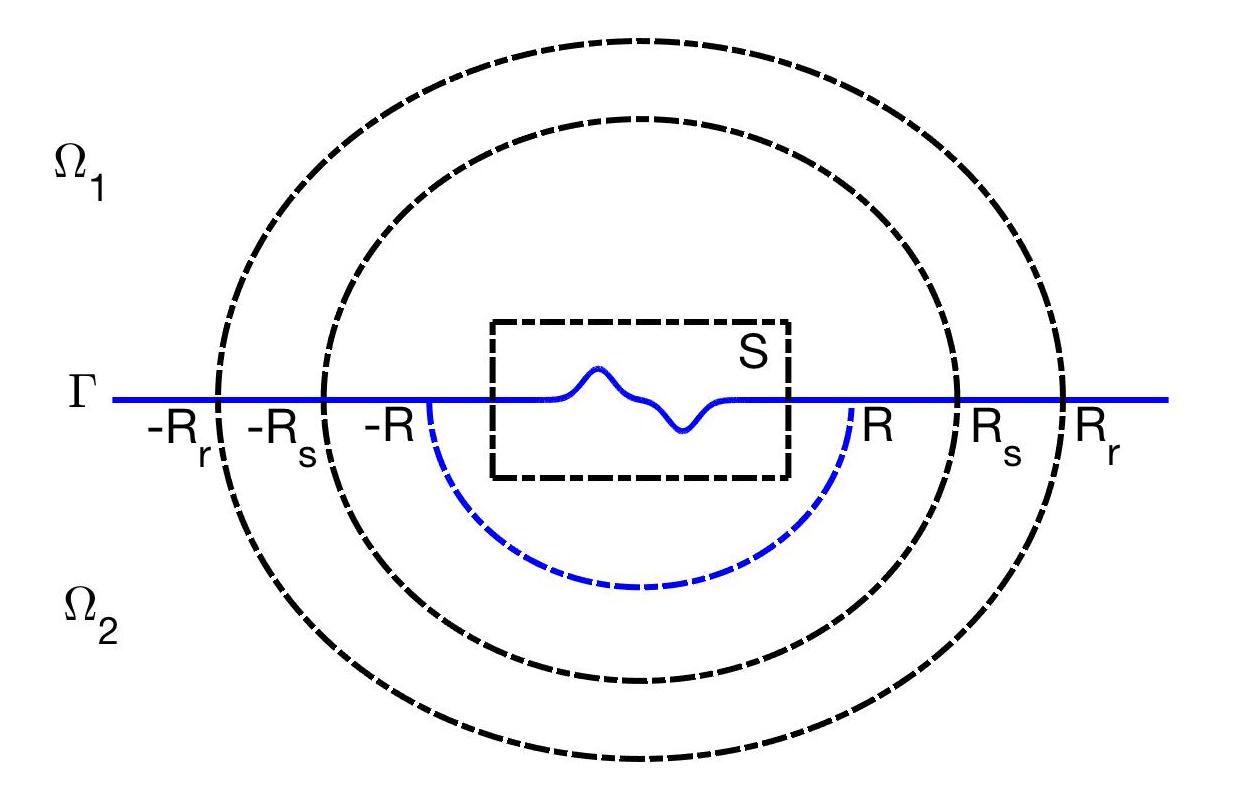}
	\caption{The setting of RTM method for the penetrable case.}
	\label{f2} 
\end{figure}

To establish the mathematic justification of the RTM method, we first introduce the Green's function $G_P$ associated with the two-dimensional Helmholtz equation in a two-layered medium separated by $\Gamma_R$, which satisfies 
\be\label{d1}
\left\{\begin{array}{lll}
	\Delta G_P(x,x_s)+\kappa_{P}^2(x)G_P(x,x_s)=-\delta_{x_s}(x) & \textrm{in}\;\;\R^2, \\[2mm]
	\lim\limits_{|x|\rightarrow \infty}|x|^{\frac{1}{2}}\left(\partial_{|x|} G_P(x,x_s)
	-{\rm i}\kappa_{P}(x)G_P(x,x_s)\right)=0 
\end{array}
\right.
\en
in the distributional sense and the Sommerfeld radiation condition uniformly for all directions 
$\hat{x}\in\mathbb{S}$. Here, $x_s\in\R^2\se\G_R$ and the wave number $\kappa_P(x)$ is defined by $\kappa_P(x):=\kappa_1$ in $\Omega_{1,R}$ and $\kappa_P(x):=\kappa_2$ in $\Omega_{2,R}$, with $\Omega_{1,R}$ and $\Omega_{2,R}$ being the upper and lower half-space separated by $\Gamma_R$, respectively. We refer to Theorem 2.1 and Theorem 2.2 in \cite{YLZ22} for the well-posedness of the background Green's function $G_P(x,x_s)$ for $x_s\in\R^2\se\G_R$.

Define 
\begin{eqnarray*}\label{d2}
	V_P(x,x_s):=u(x,x_s)-G_P(x,x_s),
\end{eqnarray*}
which, from (\ref{a7}) and (\ref{d1}), satisfies 
\begin{eqnarray}\label{d3}
	\left\{\begin{array}{lll}
		\Delta V_P(x,x_s)+\kappa^2(x)V_P(x,x_s)=g(x,x_s) &\textrm{in}\;\; \R^2, \\[3mm]
		\lim\limits_{|x|\rightarrow \infty}|x|^{\frac{1}{2}}\left(\partial_{|x|} V_P(x,x_s)-{\rm i}\kappa(x)  V_P(x,x_s)\right)=0.
	\end{array}
	\right.
\end{eqnarray}
Here $g(x,x_s)$ is a function with compact support and given by $g(x,x_s):=\sigma G_P(x,x_s)$ in $D_R$ and $g(x,x_s)=0$ in $\R^2\setminus\overline{D_R}$, where $\sigma:=\kappa_1^2-\kappa_2^2$ and $D_R:=\Omega_2\cap\Omega_{1,R}$.

The main idea of the RTM algorithm is to break up the reconstruction of $\Gamma$ into two parts: the first part is to back-propagate the complex conjugated data $\overline{V_P(x_r,x_s)}$ and the second part is to calculate the imaginary part of the cross-correlation of $G_P(z,x_s)$ and the back-propagation field. We summarize it in the following algorithm.

{\bf Algorithm 2 (RTM for penetrable locally rough interface)}: Given the data $V_P(x_r,x_s)$ for $r=1,2,...,N_r$ and $s=1,2,...,N_s$.
\begin{itemize}
	\item Back-propagation: for $s=1,2,...,N_s$, solve the problem 
	\begin{eqnarray}\label{d4}
		\left\{\begin{aligned}
			&\Delta W_P(x,x_s)+\kappa_P^2(x)W_P(x,x_s)=\frac{|\Gamma_r|}{N_r}\sum_{r=1}^{N_r}\overline{V_P(x_r,x_s)}\delta_{x_r}(x) \;\;\textrm{in}\;\; \R^2, \\
			&\; \lim_{|x|\rightarrow \infty}|x|^{\frac{1}{2}}\left(\partial_{|x|} W_P(x,x_s)-{\rm i}\kappa_P(x)  W_P(x,x_s)\right)=0,
		\end{aligned}
		\right.
	\end{eqnarray}
	to obtain the solution $W_P$.
	\item Cross-correlation: for each sampling point $z\in S$, calculate the indicator function 
	\begin{eqnarray*}\label{d5}
		{\rm Ind}_P(z)=\kappa(x_r){\rm Im}\left\{\frac{|\Gamma_s|}{N_s}\sum_{s=1}^{N_s}\kappa(x_s)G_P(z,x_s)W_P(z,x_s)\right\}
	\end{eqnarray*}
	and then plot the mapping ${\rm Ind}_P(z)$ against $z$.
\end{itemize}

Combination of (\ref{d1}) and the linearity shows that the solution of (\ref{d4}) can be expressed by
\begin{eqnarray*}\label{d6}
	W_P(x,x_s)=-\frac{|\Gamma_r|}{N_r}\sum_{r=1}^{N_r}G_P(x,x_r)\overline{V_P(x_r,x_s)}.
\end{eqnarray*}
Hence, we obtain
\begin{eqnarray*}\label{d7}
	{\rm Ind}_P(z)=-{\rm Im}\left\{\frac{|\Gamma_s|}{N_s}\frac{|\Gamma_r|}{N_r}\sum_{s=1}^{N_s}\sum_{r=1}^{N_r}\kappa(x_r)\kappa(x_s)G_P(z,x_s)G_P(z,x_r)\overline{V_P(x_r,x_s)}\right\}\quad z\in S,
\end{eqnarray*}
which is a discrete formula of the following continuous function 
\begin{eqnarray*}\label{d8}
	\widetilde{{\rm Ind}}_P(z)=-{\rm Im}\int_{\Gamma_r}\int_{\Gamma_s}\kappa(x_r)\kappa(x_s)G_P(z,x_s)G_P(z,x_r)\overline{V_P(x_r,x_s)}{\rm d}s(x_s){\rm d}s(x_r),\quad z\in S.
\end{eqnarray*}

In the remaining part of this section, we restrict us to show that the function $\widetilde{{\rm Ind}}_P(z)$ will have contrast at the rough interface $\Gamma$ and decay away from $\Gamma$. To this end, we first introduce the following modified Helmholtz-Kirchhoff identity. It can be shown by a direct application of the Green's theorem along with the continuity of $G_P$ and its normal derivative across $\Gamma_R$, which is similar to the proof of Lemma \ref{lem1} and we omit it here.
\begin{lemma}\label{lem7}
	Let $G_P$ be the background Green's function defined by (\ref{d1}). Then for  any $x,z\in B_{R_s}\setminus\Gamma_R$, we have 
	\begin{eqnarray*}\label{d9}
		\int_{\Gamma_p}\left(\overline{G_P(\xi,x)}\frac{\partial G_P(\xi,z)}{\partial \nu(\xi)}-\frac{\partial \overline{G_P(\xi,x)}}{\partial\nu(\xi)}G_P(\xi,z)\right){\rm d}s(\xi)=2{\rm i}{\rm Im}G_P(x,z),
	\end{eqnarray*}
where $\Gamma_p$ denotes $\Gamma_s$ or $\Gamma_r$.
\end{lemma}

With the above modified Helmholtz-Kirchhoff identity and the Sommerfeld radiation condition, it is easy to obtain the following Lemma which plays an important role in the analysis of $\widetilde{{\rm Ind}}_P(z)$.
\begin{lemma}\label{lem8}
	For any $x,z\in S$, we have
	\begin{eqnarray}\label{d10}
		&&\int_{\Gamma_s}\kappa(\xi)\overline{G_P(x,\xi)}G_P(\xi,z){\rm d}s(\xi)={\rm Im}G_P(x,z)+\zeta_{P,s}(x,z)\\\label{d11}
	       &&\int_{\Gamma_r}\kappa(\xi)\overline{G_P(x,\xi)}G_P(\xi,z){\rm d}s(\xi)={\rm Im}G_P(x,z)+\zeta_{P,r}(x,z)
	\end{eqnarray}
	where $|\zeta_{P,s}(x,z)|+|\nabla_x\zeta_{P,s}(x,z)|\leq CR_s^{-\frac{1}{4}}$ and $|\zeta_{P,r}(x,z)|+|\nabla_x\zeta_{P,r}(x,z)|\leq CR_r^{-\frac{1}{4}}$ uniformly for any $x,z\in S$.
\end{lemma}
\begin{proof}
	For any $x,z\in S$, it follows from Lemma \ref{lem7} that 
	\begin{eqnarray*}
		2{\rm i}{\rm Im}G_P(x,z)&=&\int_{\Gamma_s}\left(\overline{G_P(\xi,x)}\frac{\partial G_P(\xi,z)}{\partial \nu(\xi)}-\frac{\partial \overline{G_P(\xi,x)}}{\partial\nu(\xi)}G_P(\xi,z)\right){\rm d}s(\xi)\\
		&=&\int_{\Gamma_s}\Bigg\{\overline{G_P(\xi,x)}\left[\frac{\partial G_P(\xi,z)}{\partial \nu(\xi)}-{\rm i}\kappa(\xi) G_P(\xi,z)\right]\\
		&&\qquad\;\;\;-G_P(\xi,z)\left[\frac{\overline{\partial G_P(\xi,x)}}{\partial \nu(\xi)}+{\rm i}\kappa(\xi)\overline{G_P(\xi,x)}\right]\Bigg\}{\rm d}s(\xi)\\
		&&\qquad\;\;\;+2{\rm i}\int_{\Gamma_s}\kappa(\xi)\overline{G_P(\xi,x)}G_P(\xi,z){\rm d}s(\xi).
	\end{eqnarray*}
	Thus, a direct application of the reciprocity $G_P(\xi,x)=G_P(x,\xi)$ for $\xi\in \partial B_r$ and $x\in S$ yields 
	\begin{eqnarray*}
		\int_{\Gamma_s}\kappa(\xi)\overline{G_P(x,\xi)}G_P(\xi,z){\rm d}s(\xi)={\rm Im}G_P(x,z)+\zeta_{P,s}(x,z)\quad{\rm for}\;\;{\forall x,z\in S},
	\end{eqnarray*}
	with 
	\begin{eqnarray*}
		\zeta_{P,s}(x,z)&=&\frac{\rm i}{2}\int_{\Gamma_s}\Bigg\{\overline{G_P(\xi,x)}\left[\frac{\partial G_P(\xi,z)}{\partial \nu(\xi)}-{\rm i}\kappa(\xi) G_P(\xi,z)\right]\\
		&&\qquad\quad-G_P(\xi,z)\left[\frac{\overline{\partial G_P(\xi,x)}}{\partial \nu(\xi)}+{\rm i}\kappa(\xi)\overline{G_P(\xi,x)}\right]\Bigg\}{\rm d}s(\xi).
	\end{eqnarray*}
Thus, the inequality (\ref{d10}) holds. By a similar argument with Theorem 1 and Theorems 9-11 in \cite{LYZZ22}, we have 
	\begin{eqnarray*}
		G_P(\xi,y)=O(|\xi|^{-\frac{1}{2}}),\qquad\frac{\partial G_P(\xi,y)}{\partial \nu(\xi)}-{\rm i}\kappa(\xi) G_P(\xi,y)=O(|\xi|^{-\frac{3}{4}})
	\end{eqnarray*}
	for $y\in\{x,z\}$, $x,z\in S$, which implies that 
	\begin{eqnarray*}
		|\zeta_{P,s}(x,z)|\leq CR_s^{-\frac{1}{4}}
	\end{eqnarray*}
	uniformly for $x,z\in S$. Since 
	\begin{eqnarray*}
		\frac{\partial G_P(\xi,x)}{\partial x_j}=O(|\xi|^{-\frac{1}{2}}),\qquad\frac{\partial}{\partial x_j}\left[\frac{\partial G_P(\xi,x)}{\partial \nu(\xi)}-{\rm i}\kappa(\xi) G_P(\xi,x)\right]=O(|\xi|^{-\frac{3}{4}})
	\end{eqnarray*}
	for $j=1,2$ and $x\in S$, it follows that 
	\begin{eqnarray*}
		|\nabla_x\zeta_{P,s}(x,z)|\leq CR_s^{-\frac{1}{4}}
	\end{eqnarray*}
	uniformly for $x,z\in S$. Thus, we conclude that $|\zeta_{P,s}(x,z)|+|\nabla_x\zeta_{P,s}(x,z)|\leq CR_s^{-\frac{1}{4}}$ hold. A similar argument shows that (\ref{d11}) holds. The proof is complete. 
\end{proof}

To analyze the indicator function $\widetilde{{\rm Ind}}_P(z)$, we also need the following Green's formula which is shown by Theorem 2.1 in \cite{LYZ21}.
\begin{lemma}\label{lem9}
	Let $V_P(x,x_s)$ be the solution of Problem (\ref{d3}). Then we have
	\begin{eqnarray*}\label{d12}
		V_P(x,x_s)=-\sigma\int_{D_R}G_P(x,\xi)u(\xi,x_s){\rm d}\xi,\quad{\rm for}\;\;x\in\R^2.
	\end{eqnarray*}
\end{lemma}
Now we are in position to present the main result of this section.
\begin{theorem}\label{thm3}
	For any $z\in S$, let $\psi_P(\xi,z)$ be the solution of 
	\begin{eqnarray}\label{d13}
		\left\{\begin{aligned}
			&\Delta \psi_P(\xi,z)+\kappa^2(\xi)\psi_P(\xi,z)=-\sigma\chi_{D_R}(\xi){\rm Im}G_P(\xi,z) \qquad{\rm in}\;\; \R^2, \\
			&\; \lim_{|\xi|\rightarrow \infty}|\xi|^{\frac{1}{2}}\left(\partial_{|\xi|} \psi_P(\xi,z)-{\rm i}\kappa(\xi)  \psi_P(\xi,z)\right)=0,
		\end{aligned}
		\right.
	\end{eqnarray}
	where $\chi_{D_R}$ is the characterization function of the domain $D_R$ given by $\chi_{D_R}=1$ in $D_R$ and vanishes outside $D_R$, and $\psi_P^{\infty}(\hat{\xi},z)$ be the corresponding far field pattern. Then we have
	\begin{eqnarray*}\label{d14}
		\widetilde{{\rm Ind}}_P(z)=\kappa_1\int_{{\mathbb S}_+}|\psi_P^{\infty}(\hat{\xi},z)|^2{\rm d}s(\hat{\xi})+\kappa_2\int_{{\mathbb S}_-}|\psi_P^{\infty}(\hat{\xi},z)|^2{\rm d}s(\hat{\xi})+\eta_P(z)\quad{\forall z\in S}
	\end{eqnarray*}
	where $\|\eta_P(z)\|_{L^{\infty}(S)}\leq C(R_s^{-\frac{1}{4}}+R_r^{-\frac{1}{4}})$ with some constant $C$ depending on $R$.
\end{theorem}
\begin{proof}
	Note that 
	\begin{eqnarray*}\label{d15}
		\widetilde{{\rm Ind}}_P(z)&=&-{\rm Im}\int_{\Gamma_r}\int_{\Gamma_s}\kappa(x_s)\kappa(x_r)G_P(z,x_s)G_P(z,x_r)\overline{V_P(x_r,x_s)}{\rm d}s(x_s){\rm d}s(x_r)\\
		&=&-{\rm Im}\int_{\Gamma_s}\kappa(x_s)G_P(z,x_s)\widetilde{W}_P(z,x_s){\rm d}s(x_s)
	\end{eqnarray*}
	where 
	\begin{eqnarray*}\label{d16}
		\widetilde{W}_P(z,x_s):=\int_{\Gamma_r}\kappa(x_r)G_P(z,x_r)\overline{V_P(x_r,x_s)}{\rm d}s(x_r).
	\end{eqnarray*}
	With the help of Lemma \ref{lem9} and Lemma \ref{lem8}, we can rewrite $\widetilde{W}_P(z,x_s)$ as 
	\begin{eqnarray*}\label{d17}
		\widetilde{W}_P(z,x_s)&=&-\sigma\int_{D_R}\left[\int_{\Gamma_r}\kappa(x_r)G_P(z,x_r)\overline{G_P(x_r,\xi)}{\rm d}s(x_r)\right]\overline{u(\xi,x_s)}{\rm d}\xi\\
		&=&-\sigma\int_{D_R}\left[{\rm Im}G_P(\xi,z)+\zeta_{P,r}(\xi,z)\right]\overline{u(\xi,x_s)}{\rm d}\xi
	\end{eqnarray*}
	which leads to 
	\begin{eqnarray*}\label{d18}
		\widetilde{{\rm Ind}}_P(z)=\sigma{\rm Im}\int_{D_R}\left[{\rm Im}G_P(\xi,z)+\zeta_{P,r}(\xi,z)\right]\phi_P(\xi,z){\rm d}\xi,
	\end{eqnarray*}
	with
	\begin{eqnarray*}\label{d19}
		\phi_P(\xi,z):=\int_{\Gamma_s}\kappa(x_s)G_P(z,x_s)\overline{u(\xi,x_s)}{\rm d}s(x_s).
	\end{eqnarray*}
	Due to the following Lippmann-Schwinger integral equation  (cf. \cite[Theorem 2.1]{LYZ21})
	\begin{eqnarray*}\label{d20}
		u(\xi,x_s)+\sigma\int_{D_R}G_P(\xi,y)u(y,x_s){\rm d}y=G_P(\xi,x_s)\quad{\rm for}\;\;\xi\in\R^2,
	\end{eqnarray*}
	we arrive at
	\begin{eqnarray*}\nonumber
		\overline{\phi_P(\xi,z)}=\int_{\Gamma_s}\kappa(x_s)\overline{G_P(z,x_s)}G_P(\xi,x_s){\rm d}s(x_s)-\sigma\int_{D_R}G_P(\xi,y)\overline{\phi_P(y,z)}{\rm d}y\\\label{d21}
		={\rm Im}G_P(\xi,z)+\zeta_{P,s}(\xi,z)-\sigma\int_{D_R}G_P(\xi,y)\overline{\phi_P(y,z)}{\rm d}y.
	\end{eqnarray*}
	Let $\theta(\xi,z)=\overline{\phi_P(\xi,z)}-\left[{\rm Im}G_P(\xi,z)+\zeta_{P,s}(\xi,z)\right]$, then 
	\begin{eqnarray*}\label{d22}
		\theta(\xi,z)=-\sigma\int_{D_R}G_P(\xi,y)\left[\theta(y,z)+{\rm Im}G_P(y,z)+\zeta_{P,s}(y,z)\right]{\rm d}y.
	\end{eqnarray*}
	Hence, we conclude that $\theta(\xi,z)$ satisfies the Sommerfeld radiation condition and 
	\begin{eqnarray*}\label{d23}
		\left\{\begin{aligned}
			&\Delta \theta(\xi,z)+\kappa^2(\xi)\theta(\xi,z)=0\qquad &&{\rm in}\;\; \R^2\setminus\overline{D_R}, \\
			&\Delta \theta(\xi,z)+\kappa_1^2\theta(\xi,z)=\sigma\left[\theta(\xi,z)+{\rm Im}G_P(\xi,z)+\zeta_{P,s}(\xi,z)\right] &&{\rm in}\;\; D_R,
		\end{aligned}
		\right.
	\end{eqnarray*}
	which is equivalent to 
	\begin{eqnarray}\label{d24}
		\left\{\begin{aligned}
			&\Delta \theta(\xi,z)+\kappa^2(\xi)\theta(\xi,z)=\sigma\chi_{D_R}(\xi)\left[{\rm Im}G_P(\xi,z)+\zeta_{P,s}(\xi,z)\right] \;\;{\rm in}\;\; \R^2, \\
			&\; \lim_{|\xi|\rightarrow \infty}|\xi|^{\frac{1}{2}}\left(\partial_{|\xi|} \theta(\xi,z)-{\rm i}\kappa(\xi)  \theta(\xi,z)\right)=0,
		\end{aligned}
		\right.
	\end{eqnarray}
	Let $\psi_P(\xi,z)$ and $\varphi_P(\xi,z)$ solve the same scattering problem (\ref{d24}) expect that the right hand term are replaced by $\sigma\chi_{D_R}(\xi){\rm Im}G_P(\xi,z)$ and $\sigma\chi_{D_R}(\xi)\zeta_{P,s}(\xi,z)$. Then by the linearity we have 
	\begin{eqnarray*}\label{d25}
		\theta(\xi,z)=\psi_P(\xi,z)+\varphi_P(\xi,z)
	\end{eqnarray*}
	which yields 
	\begin{eqnarray*}\label{d26}
		\phi_P(\xi,z)=\overline{\psi_P(\xi,z)}+\overline{\varphi_P(\xi,z)}+{\rm Im}G_P(\xi,z)+\overline{\zeta_{P,s}(\xi,z)}.
	\end{eqnarray*}
	Hence,
	\begin{eqnarray*}
		\widetilde{{\rm Ind}}_P(z)&=&\sigma{\rm Im}\int_{D_R}\left[{\rm Im}G_P(\xi,z)+\zeta_{P,r}(\xi,z)\right]\\
		&&\times\left[\overline{\psi_P(\xi,z)}+\overline{\varphi_P(\xi,z)}+{\rm Im}G_P(\xi,z)+\overline{\zeta_{P,s}(\xi,z)}\right]{\rm d}\xi\\
		&=&{\rm Im}\int_{D_R}\left[\Delta\psi_P(\xi,z)+\kappa_2^2\psi_P(\xi,z)\right]\overline{\psi_P(\xi,z)}{\rm d}\xi+\eta_P(z)\\
		&=&{\rm Im}\int_{\partial D_R}\frac{\partial \psi_P(\xi,z)}{\partial\nu(\xi)}\overline{\psi_P(\xi,z)}{\rm d}s(\xi)+\eta_P(z)\\
		&=&\kappa_1\int_{{\mathbb S}_+}|\psi_P^{\infty}(\hat{\xi},z)|^2{\rm d}s(\hat{\xi})+\kappa_2\int_{{\mathbb S}_-}|\psi_P^{\infty}(\hat{\xi},z)|^2{\rm d}s(\hat{\xi})+\eta_P(z)
	\end{eqnarray*}
	where we use the Green's theorem and the Sommerfeld radiation condition in the last step, and $\eta_P(z)$ is defined by
	\begin{align}\nonumber
		\eta_P(z)=\sigma{\rm Im}\int_{D_R}&\bigg[\zeta_{P,r}(\xi,z)\overline{\psi_P(\xi,z)}+\zeta_{P,r}(\xi,z)\overline{\varphi_P(\xi,z)}\\\nonumber
		&+\zeta_{P,r}(\xi,z){\rm Im}G_P(\xi,z)+\zeta_{P,r}(\xi,z)\overline{\zeta_{P,s}(\xi,z)}\\\label{dd}
		&+{\rm Im}G_P(\xi,z)\overline{\varphi_P(\xi,z)}+{\rm Im}G_P(\xi,z)\overline{\zeta_{P,s}(\xi,z)}\bigg]{\rm d}\xi.
	\end{align}
	
	Now we are in position to  show $\|\eta_P(z)\|_{L^{\infty}(S)}\leq C(R_s^{-\frac{1}{4}}+R_r^{-\frac{1}{4}})$ with $C$ depending on $R$. Recall that $\psi_P(\xi,z)$ and $\varphi_P(\xi,z)$ solve Problem (\ref{d24}) expect that the right hand term are replaced by $\sigma\chi_{D_R}(\xi){\rm Im}G_P(\xi,z)$ and $\sigma\chi_{D_R}(\xi)\zeta_{P,s}(\xi,z)$, thus we can write $\psi_P(\xi,z)$ and $\varphi_P(\xi,z)$ in the following form 
	\begin{eqnarray}\label{d27}
		&&\psi_P(\xi,z)=-\sigma\int_{D_R}G_{\Gamma}(\xi,y){\rm Im}G_P(y,z){\rm d}y\qquad\;{\rm for}\;\xi\in\R^2,\\\label{d28}
		&&\varphi_P(\xi,z)=-\sigma\int_{D_R}G_{\Gamma}(\xi,y)\zeta_{P,s}(y,z){\rm d}y\qquad\qquad{\rm for}\;\xi\in\R^2.
	\end{eqnarray}
	Here $G_{\Gamma}(\xi,y)$ is the Green's function associated with the two-dimensional Helmholtz equation in a two-layered medium 
	separated by $\Gamma$, which satisfies 
	\be\label{d29}
	\left\{\begin{array}{lll}
		\Delta_{\xi}G_{\Gamma}(\xi,y)+\kappa^2(\xi)G_{\Gamma}(\xi,y)=-\delta_y(\xi) & \textrm{in}\;\;\R^2, \\[2mm]
		\lim\limits_{|\xi|\rightarrow \infty}|\xi|^{\frac{1}{2}}\left(\partial_{|\xi|} G_{\Gamma}(\xi,y)
		-{\rm i}\kappa(\xi)G_{\Gamma}(\xi,y)\right)=0 
	\end{array}
	\right.
	\en
	in the distributional sense and the Sommerfeld radiation condition uniformly for all directions 
	$\hat{\xi}\in\mathbb{S}$. We refer to Theorem 3.2 and Theorem 3.3 in \cite{YLZ22} for the well-posedness of 
	Problem (\ref{d29}). It follows from (\ref{d27}) and (\ref{d28}) that
	\begin{eqnarray}\label{d30}
		\|\psi_P(\cdot,z)\|_{H^2(D_R)}\lesssim \|{\rm Im}G_P(\cdot,z)\|_{L^2(D_R)}\leq C\\\label{d31}
		\|\varphi_P(\cdot,z)\|_{H^2(D_R)}\lesssim \|\zeta_{P,s}(\cdot,z)\|_{L^2(D_R)}\leq CR_s^{-\frac{1}{4}}
	\end{eqnarray}
	where we use (\ref{d11}) and $C$ depends on $R$. A direct application of the smoothness of ${\rm Im}G_P(\cdot,z)$, (\ref{d11}), (\ref{dd}), (\ref{d30}), and (\ref{d31}), we can obtain 
	\begin{eqnarray*}
		\|\eta_P(z)\|_{L^{\infty}(S)}\leq C(R_s^{-\frac{1}{4}}+R_r^{-\frac{1}{4}})
	\end{eqnarray*}
	with $C$ depending on $R$. The proof is completed.
\end{proof}

\subsection{The far-filed reconstruction}
In this subsection, we present the RTM method based on far-field data to reconstruct the penetrable, locally rough surface. It requires to develop a mixed reciprocity relation. Throughout this subsection, for simplicity, we restrict ourselves to the case $\kappa_1>\kappa_2$. The case $\kappa_1<\kappa_2$ can be dealt in a similar manner. For the case $\kappa_1>\kappa_2$, let $n:=\kappa_2/\kappa_1$ and $\theta_c:=\arccos(n)\in (0,\pi)$ be the critical incident angle and $d:=(\cos(\theta), \sin(\theta))^{\rm \top}$ be the incident direction with $\theta$ being the incident angle. Denote by $d^r:=(\cos(\theta), -\sin(\theta))^{\rm \top}$ the reflected direction and denote by $d^t$ the transmitted direction which is defined by
\begin{equation*}
d^t=\left\{\begin{array}{l}
              (\cos(\varphi),\sin(\varphi))^{\rm\top}\qquad\qquad\qquad \textrm{for}\;\;\; \theta\in (0,\pi)\cup(\pi+\theta_c,2\pi-\theta_c), \\[1mm]
              n^{-1}(\cos(\theta), {\rm i}p)^{\rm\top}\qquad\qquad\qquad\; \textrm{for}\;\;\; \theta\in(\pi,\pi+\theta_c)\cup(2\pi-\theta_c, 2\pi), 
                    \end{array}
\right.
\end{equation*}
where $\varphi:=2\pi-\arccos(\kappa_1\cos(\theta)/\kappa_2)$ for $\theta\in (\pi+\theta_c,2\pi-\theta_c)$, $\varphi:=\arccos(\kappa_2\cos(\theta)/\kappa_1)$ for $\theta\in (0, \pi)$, and $p:=\sqrt{\cos^2(\theta)-n^2}$. 
 Let $w_0(x,d)$ be the total field of the scattering of plane waves $w^i(x,d)$ from the infinite plane $\Gamma_0$, it follows from the Fresnel formula and \cite{LYZZ22} that the field $w_0(x,d)$ is given by
\begin{equation}\label{d31}
w_{0}(x,d)=\left\{\begin{array}{l}
              e^{{\rm i}\kappa_1x\cdot d}+R(\kappa_1,\kappa_2,\theta)e^{{\rm i}\kappa_1 x\cdot d^r} \qquad\qquad\; \textrm{for}\;\;\; x\in{\mathbb R}^2_+, d\in{\mathbb S}_-, \\[1mm]
              T(\kappa_1,\kappa_2,\theta)e^{{\rm i}\kappa_2x\cdot d^t} \qquad\qquad\;\qquad\qquad \textrm{for}\;\;\; x\in{\mathbb R}^2_-, d\in{\mathbb S}_-, \\[1mm]
              T(\kappa_2,\kappa_1,\theta)e^{{\rm i}\kappa_1x\cdot d^t} \qquad\qquad\;\qquad\qquad \textrm{for}\;\;\; x\in{\mathbb R}^2_+, d\in{\mathbb S}_+, \\[1mm]
e^{{\rm i}\kappa_2x\cdot d}+R(\kappa_2,\kappa_1,\theta)e^{{\rm i}\kappa_2 x\cdot d^r} \qquad\qquad\; \textrm{for}\;\;\; x\in{\mathbb R}^2_-, d\in{\mathbb S}_+, 
            \end{array}
\right.
\end{equation}
which satisfies $w_0(x,d)|_+-w_0(x,d)|_-=\partial_{\nu}w_0(x,d)|_+-\partial_{\nu}w_0(x,d)|_-=0$ on $\Gamma_0$. Here, the coefficients $R$ and $T$ are defined by
\begin{equation*}\label{d32}
R(\lambda,\mu,\theta):=\left\{\begin{aligned}
              \frac{\lambda\sin(\theta)-\mu\sin(\varphi)}{\lambda\sin(\theta)+\mu\sin(\varphi)}\qquad\qquad\qquad\;\;&\textrm{for}\;\; \theta\in (0,\pi)\cup(\pi+\theta_c,2\pi-\theta_c), \\[2mm]
              \frac{{\rm i}\sin(\theta)+p}{{\rm i}\sin(\theta)-p}\qquad\qquad\qquad\qquad\;\;\;\;\;\;\; &\textrm{for}\;\; \theta\in(\pi,\pi+\theta_c)\cup(2\pi-\theta_c, 2\pi), 
                    \end{aligned}
\right.
\end{equation*}
and $T(\lambda,\mu,\theta):=R(\lambda,\mu,\theta)+1$.
Then the scattering of $w_0(x,d)$ by the locally rough surface $\Gamma$ can be modelled by 
\begin{equation*}\label{d33}
	\left\{\begin{aligned}
		&\Delta w^s(x,d)+\kappa^2w^s(x,d)=g(x,d) \qquad\qquad\qquad\;\textrm{in}\;\; {\mathbb R}^2, \\
		&\lim_{|x|\rightarrow \infty}|x|^{\frac{1}{2}}\left(\partial_{|x|} w^s(x,d)-{\rm i}\kappa w^s(x,d)\right)=0.\quad
	\end{aligned}
	\right.
\end{equation*}
Where $g(x,d):=\sigma(\chi_1-\chi_2)w_0(x,d)$ with $\chi_j(j=1,2)$ being the characterization function of the domain $D_j$ given by $\chi_j=1$ in $D_j$ and vanishes outside $D_j$, the domain $D_j$ is defined by $D_1:={\mathbb R}^2_+\cap\Omega_2$ and $D_2:={\mathbb R}^2_-\cap\Omega_1$, $w^s$ denotes the scattered field and the Sommerfield radiation condition holds uniformly for all directions $\hat{x}\in{\mathbb S}$.

Let $v_0(x,x_s)$ be the background Green's function in a two-layered medium seperated by $\Gamma_0$, which solves 
\begin{equation*}\label{d34}
	\left\{\begin{aligned}
		&\Delta v_0(x,x_s)+\kappa_0^2v_0(x,x_s)=-\delta_{x_s}(x) \qquad\qquad\qquad\;\textrm{in}\;\; {\mathbb R}^2, \\
		&\lim_{|x|\rightarrow \infty}|x|^{\frac{1}{2}}\left(\partial_{|x|} v_0(x,x_s)-{\rm i}\kappa_0 v_0(x,x_s)\right)=0,\quad
	\end{aligned}
	\right.
\end{equation*}
in the distributional sense with the Sommerfeld radiation condition uniformly for all $\hat{x}\in{\mathbb S}$. Here, the wavenumber $\kappa_0$ is defined by $\kappa_0:=\kappa_1$ in ${\mathbb R}^2_+$ and $\kappa_0:=\kappa_2$ in ${\mathbb R}^2_-$. Denote by $v_0^{\infty}(\hat{x},x_s)$ the far-field of $v_0(x,x_s)$, observing from the formulation of $v_0^{\infty}(\hat{x},x_s)$ in Proposition 2.1 in \cite{AIL05} and Theorem1, Theorem 10 in \cite{LYZZ22}, it is easily seen that  
\begin{equation}\label{d35}
v_0^{\infty}(\hat{x},x_s)=\gamma(\hat{x})w_0(x_s,-\hat{x})
\end{equation}
where $\gamma(\hat{x})=\gamma_1$ for $\hat{x}\in{\mathbb S}_+$ and $\gamma(\hat{x})=\gamma_2:=\frac{e^{\frac{\pi}{4}{\rm i}}}{\sqrt{8\kappa_2\pi}}$ for $\hat{x}\in{\mathbb S}_-$. Define $v^s(x,x_s):=u(x,x_s)-v_0(x,x_s)$ where $u(x,x_s)$ is the total field of Problem (\ref{a7}), then we have $v(x,x_s)$ solves 
\begin{equation*}\label{d36}
	\left\{\begin{aligned}
		&\Delta v^s(x,x_s)+\kappa^2v^s(x,x_s)=h(x,x_s) \qquad\qquad\qquad\;\textrm{in}\;\; {\mathbb R}^2, \\
		&\lim_{|x|\rightarrow \infty}|x|^{\frac{1}{2}}\left(\partial_{|x|} v^s(x,x_s)-{\rm i}\kappa v^s(x,x_s)\right)=0,\quad
	\end{aligned}
	\right.
\end{equation*}
with $h(x,x_s)=\sigma(\chi_1-\chi_2)v_0(x,x_s)$, where the Sommerfeld radiation condition holds uniformly for all $\hat{x}\in{\mathbb S}$.

\begin{theorem}
For acoustic scattering of plane waves $w_0(\cdot,-\hat{x})$ and point sources $v_0(\cdot,x_s)$ from a penetrable, locally rough surface $\Gamma$ we have
\begin{equation*}\label{d37}
v^{\infty}(\hat{x},x_s)=\gamma(\hat{x})w^s(x_s,-\hat{x})
\end{equation*}
for all $x, x_s\in (\Omega_1\cup\Omega_2)\setminus(\overline{D_1}\cup\overline{D_2})$.
\end{theorem}
\begin{proof}
We restrict ourselves to the proof of the case $x, x_s\in\Omega_1\cap{\mathbb R}^2_+$ and the proof can be easily extended to other cases. We choose a sufficient large $\rho>0$ and a sufficient small $\varepsilon>0$ such that $B_{\varepsilon}(x)\subset D^+_{\rho}:=B_{\rho}\cap\Omega_1$. Applying the Green's formula to $v^s(\xi,x_s)$ and $v_0(\xi,x)$ in the domain $D_{\rho}^+\setminus\overline{B_{\varepsilon}(x)}$ and in the domain $D_{\rho}^-:=B_{\rho}\cap\Omega_2$ gives that 
\begin{equation}\label{d38}
v^s(x,x_s)=-\sigma\left[\int_{D_1}v_0(x,\xi)u(\xi,x_s){\rm d}\xi-\int_{D_2}v_0(x,\xi)u(\xi,x_s){\rm d}\xi\right]
\end{equation}
where we use the continuity of $v_0(\xi,x)$, $\partial_{\nu}v_0(\xi,x)$, $v^s(\xi,x_s)$, $\partial_{\nu}v^s(\xi,x_s)$ across $\Gamma$, the reciprocity $v_0(\xi,x)=v_0(x,\xi)$, and the Sommerfeld radiation condition. It follows from (\ref{d35}) and (\ref{d38}) that the far-field  $v^{\infty}(\hat{x},x_s)$ is given by
\begin{align}\label{d39}
v^{\infty}(\hat{x},x_s)=-\gamma_1\sigma\left[\int_{D_1}w_0(\xi,-\hat{x})u(\xi,x_s){\rm d}\xi-\int_{D_2}w_0(\xi,-\hat{x})u(\xi,x_s){\rm d}\xi\right].
\end{align}
A similar argument with (\ref{d38}) implies that 
\begin{equation}\label{d40}
w^s(x_s,-\hat{x})=-\sigma\left[\int_{D_1}v_0(\xi,x_s)w(\xi,-\hat{x}){\rm d}\xi-\int_{D_2}v_0(\xi,x_s)w(\xi,-\hat{x}){\rm d}\xi\right].
\end{equation}
 Noting that $w^s$, $v$ satisfy the Sommerfeld radiation condition, and $w^s$, $\partial_{\nu}w^s$, $v$, $\partial_{\nu}v$ are continuous across $\Gamma$, using the Green's formula for $w^s$ and $v$ in the domain $D_{\rho}^+$ and in the domain $D_{\rho}^-$ yields that 
 \begin{eqnarray}\nonumber
 0=-\sigma\left[\int_{D_1}w^s(\xi,-\hat{x})v_0(\xi,x_s){\rm d}\xi-\int_{D_2}w^s(\xi,-\hat{x})v_0(\xi,x_s){\rm d}\xi\right]\\\label{d41}
+\sigma\left[\int_{D_1}w_0(\xi,-\hat{x})v^s(\xi,x_s){\rm d}\xi-\int_{D_2}w_0(\xi,-\hat{x})v^s(\xi,x_s){\rm d}\xi\right].
 \end{eqnarray}
 The difference between (\ref{d40}) and (\ref{d41}) yields that 
 \begin{equation*}\label{d42}
 w^s(x_s,-\hat{x})=-\sigma\left[\int_{D_1}w_0(\xi,-\hat{x})u(\xi,x_s){\rm d}\xi-\int_{D_2}w_0(\xi,-\hat{x})u(\xi,x_s){\rm d}\xi\right].
 \end{equation*}
 Compared with (\ref{d39}), we conclude that $v^{\infty}(\hat{x},x_s)=\gamma(\hat{x})w^s(x_s,-\hat{x})$ for $x,x_s\in \Omega_1\cap{\mathbb R}^2_+$. The proof is finished.
\end{proof}

With the above mixed reciprocity relation, we can establish the main result of this subsection in the following theorem. Its proof is similar to Theorem \ref{thm26} and we omit it here. 
\begin{theorem}
For the indicator function $\widetilde{{\rm Ind}}_P(z)$, we have the following limit identity
\begin{eqnarray}\nonumber
&&\lim_{R_s\to\infty}\lim_{R_r\to\infty}\widetilde{\rm Ind}_P(z)=-{\rm Im} \int_{{\mathbb S}}\int_{{\mathbb S}}\kappa(\hat{x}_r)\kappa(\hat{x}_s)|\gamma(\hat{x}_r)|^2\gamma(\hat{x}_s)w(z,-\hat{x}_r,\Gamma_R)w(z,-\hat{x}_s,\Gamma_R)\\\label{b63}
&&\qquad\times\left[\overline{w^{\infty}(\hat{x}_s,-\hat{x}_r,\Gamma)}-\overline{w^{\infty}(\hat{x}_s,-\hat{x}_r,\Gamma_R)}\right]{\rm d}s(\hat{x}_r){\rm d}s(\hat{x}_s):=\widehat{\rm Ind}_P(z).
\end{eqnarray}
\end{theorem}

\section{Numerical experiments}
\setcounter{equation}{0}
In this section, we first give an analysis of the indicator function $\widetilde{\rm Ind}_{\alpha}(z)$ with $\alpha=D, N, P$ and then present several numerical experiments to demonstrate the effectiveness of the RTM method. 

According to Theorem \ref{thm1} and Theorem \ref{thm3}, it is easy to see that the behavior of the indicator function $\widetilde{\rm Ind}_{\alpha}(z)$ depends on $\psi_{\alpha}(\xi,z)$ when the source radius $R_s$ and measurement radius 
$R_r$ are large enough, where $\alpha=D, N, P$. Notice that the function $\psi_{\alpha}(\xi,z)$ satisfies Problem (\ref{b29}) and Problem (\ref{d13}) with boundary data $-{\rm Im}G_D(\xi, z)$, $-\partial_{\nu}{\rm Im}G_N(\xi,z)$, and $-\sigma\chi_{D_R}(\xi){\rm Im} G_P(\xi, z)$, respectively. Observe that 
\begin{eqnarray*}
	&&{\rm Im}G_D(\xi,z)=\frac{1}{4}J_0(\kappa_1 |\xi-z|)+{\rm Im}G_D^s(\xi,z)\\
	&&\partial_{x_j}{\rm Im}G_N(\xi,z)=\frac{\kappa}{4}J'_0(\kappa_1 |\xi-z|)\frac{\xi_j-z_j}{|\xi-z|}+\partial_{x_j}{\rm Im}G^s_N(\xi,z)\quad\;{\rm for}\;\;j=1,2\\
	&&{\rm Im}G_P(\xi,z)=\frac{1}{4}J_0(\kappa_1 |\xi-z|)+{\rm Im}G_P^s(\xi,z)\qquad\qquad\qquad\quad\;\;{\rm for}\;\;\xi, z\in\Omega_{1,R}
\end{eqnarray*}
where $G_{\alpha}^s(\xi,z)$ $(\alpha=D, N, P)$ denotes the corresponding scattered fields associated with $\Gamma_R$. It is shown numerically that $G_{\alpha}^s(\xi,z)$ can be sufficiently small for $\xi,z\in S$ when $R$ is large enough, see \cite{DLLY17, LYZ21} for details. It is shown in (a) and (d) of Figure \ref{f7} that $J_0(\kappa_1 |\xi-z|)$ achieves a maximum at $\xi=z$, which implies that $J'_0(\kappa_1 |\xi-z|)=0$ when $\xi=z$. Hence, we can obtain that the functions ${\rm Im}G_D(\xi,z)$ and ${\rm Im}G_P(\xi,z)$ will achieve a maximum at $\xi=z$ and the function $\partial_{x_j}{\rm Im}G_N(\xi,z)$ will 
achieve a minimum at $\xi=z$ for a sufficient large $R$. This property can be easily observed in Figure \ref{f7}. Based on this observation, we can expect that $\widetilde{\rm Ind}_D(z)$ and $\widetilde{\rm Ind}_P(z)$ will reach a peak on $\Gamma$, and $\widetilde{\rm Ind}_N(z)$ will hit a nadir on $\Gamma$.

\begin{figure}[htbp]
	\begin{minipage}[t]{0.28\linewidth}
		\centering
		\includegraphics[width=1.7in]{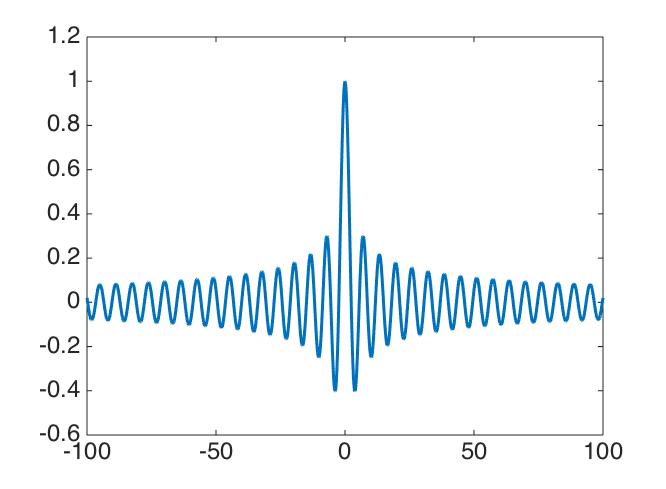}
		(a) $J_0$
	\end{minipage}\qquad
	\begin{minipage}[t]{0.28\linewidth}
		\centering
		\includegraphics[width=1.7in]{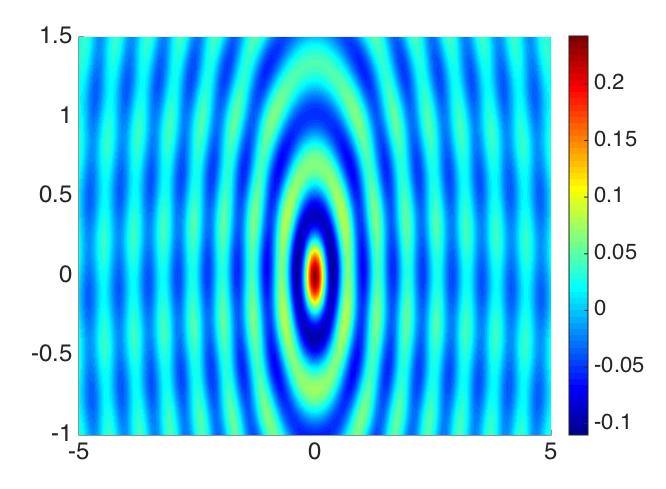}
		(b) ${\rm Im}G_D(x,z)$
	\end{minipage}\qquad
	\begin{minipage}[t]{0.28\linewidth}
		\centering
		\includegraphics[width=1.7in]{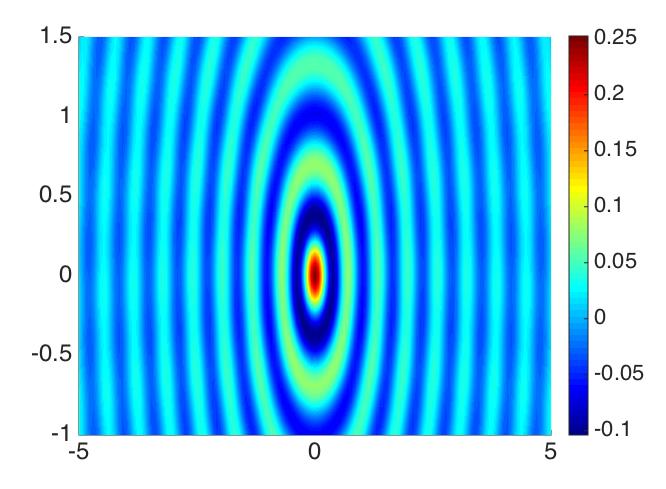}
		(c) ${\rm Im}G_P(x,z)$
	\end{minipage}
	\begin{minipage}[t]{0.28\linewidth}
		\centering
		\includegraphics[width=1.7in]{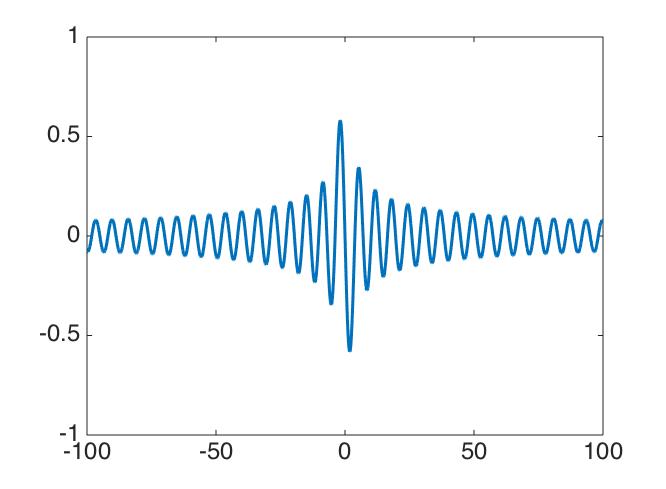}
		(d) $J'_0$
	\end{minipage}\qquad
	\begin{minipage}[t]{0.28\linewidth}
		\centering
		\includegraphics[width=1.7in]{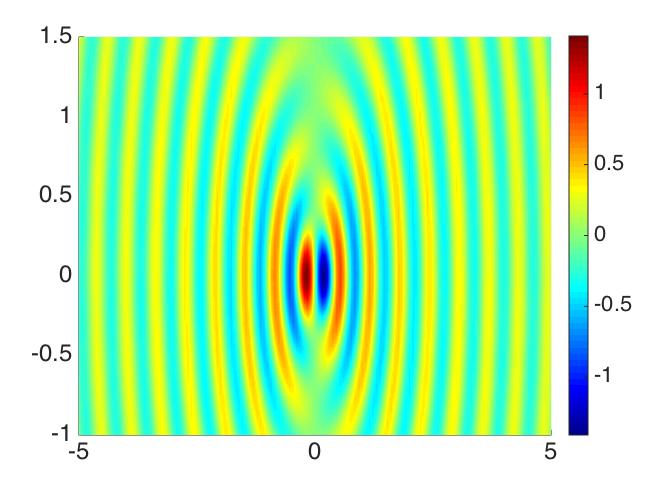}
		(e) $\partial_{x_1}{\rm Im}G_N(x,z)$
	\end{minipage}\qquad
	\begin{minipage}[t]{0.28\linewidth}
		\centering
		\includegraphics[width=1.7in]{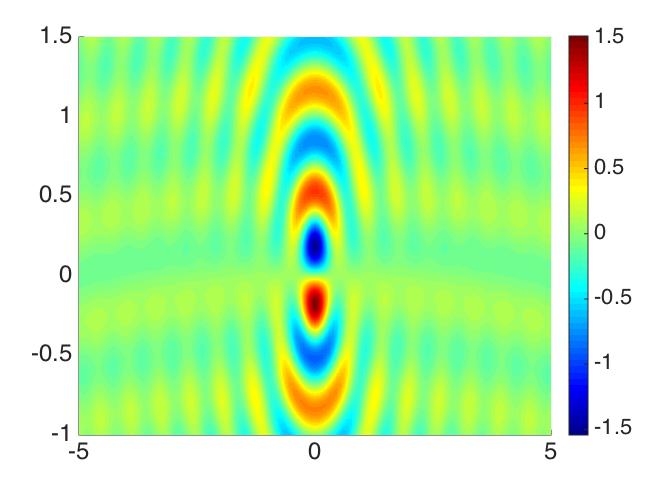}
		(f) $\partial_{x_2}{\rm Im}G_N(x,z)$
	\end{minipage}
	\qquad\qquad
	\caption{The image of functions $J_0$, $J'_0$, ${\rm Im}G_D(x,z)$, ${\rm Im}G_P(x,z)$, $\partial_{x_1}{\rm Im}G_N(x,z)$ and $\partial_{x_2}{\rm Im}G_N(x,z)$ with $R=95$, $x\in [-5, 5]\times[-1,1.5]$ and the source $z=(0,0)$. The wavenumber $\kappa_1=10$ in (b),(e), (f) and $\kappa=(10,5)$ in (c).
	}\label{f7}
\end{figure}

In all examples, we assume that the locally rough surface function $f$ is supported in $[-5, 5]$, the sample domain $S=[-5, 5]\times [-1, 1.5]$, and $N_s=N_r=1024$ for impenetrable locally rough surfaces and $N_s=N_r=2048$ for penetrable locally rough surfaces, and we set $R=95$ for the special locally rough surface $\Gamma_R$. In addition, we take the wave number $\kappa_1=10$ for impenetrable case and $\kappa_1=10$, $\kappa_2=5$ for penetrable case. The synthetic data is generated by applying the Nystr\"{o}m method to solve the corresponding direct scattering problem,  see \cite{LYZ13, LB13} for details. 

To test the stability of the RTM method, we consider the performance of this method with noisy data. For some relative error $\tau>0$, we inject some noise into the data by defining 
\begin{eqnarray*}
	u_{\tau}^s(x)=u^s(x)+\tau\frac{\beta}{\|\beta\|_2}\|u^s(x)\|_2
\end{eqnarray*}
where $\beta=\beta_1+{\rm i}\beta_2$ is complex-valued with $\beta_1$ and $\beta_2$ consisting of random numbers obeying standard normal distribution $N(0,1)$.

{\bf Example 1.} In this example, the locally rough surface $\Gamma$ is described as 
\begin{equation}\nonumber
	f_1(x_1)=\left\{\begin{array}{c}
		0.5+0.6\sin(0.6\pi x_1)\exp(16/(x_1^2-16)) \qquad\;\; |x_1|<4, \\[1mm]
		\qquad\qquad\qquad\qquad\qquad\qquad\qquad\;\;\;\; 0.5\;\qquad\quad|x_1|\geq 4.
	\end{array}\right.
\end{equation}
The reconstruction results from exact data are presented in Figure \ref{f4}, where the top row is the results from near-field data, and the bottom row is the reconstruction from far-field data. As shown in Figure \ref{f4}, the RTM approach can present a satisfactory reconstruction.

{\bf Example 2.} In this example, the locally rough surface $\Gamma$ is a multiscale profile given by 
\begin{equation}\nonumber
	f_2(x_1)=\left\{\begin{array}{c}
		0.5+(0.5+0.05\sin(3\pi x_1))\exp(4/(x_1^2-16)) \qquad\;\; |x_1|<4, \\[1mm]
		\qquad\qquad\qquad\qquad\qquad\qquad\qquad\qquad\qquad 0.5\;\qquad\;\; |x_1|\geq 4.
	\end{array}\right.
\end{equation}
The reconstructions with $5\%$ noise are illustrated in Figure \ref{f5}, which shows that the RTM method can provide a satisfactory imaging quality at $5\%$ noise level.

{\bf Example 3.} In the last example, the locally rough surface $\Gamma$ is described by a piecewise continuous function given by 
\begin{equation}\nonumber
	f_3(x_1)=\left\{\begin{array}{lll}
		0.2 \qquad\;\; |x_1|\leq 1, \\[1mm]
		0.3 \qquad\;\; 3\leq|x_1|\leq 4,\\[1mm]
		0.5  \qquad\;\;{\rm others}.
	\end{array}\right.
\end{equation}
The numerical results are shown in Figure \ref{f6}, which demonstrates that the RTM method can provide satisfactory reconstructions for piecewise continuous surfaces.

\begin{figure}[htbp]
	\begin{minipage}[t]{0.28\linewidth}
		\centering
		\includegraphics[width=1.7in]{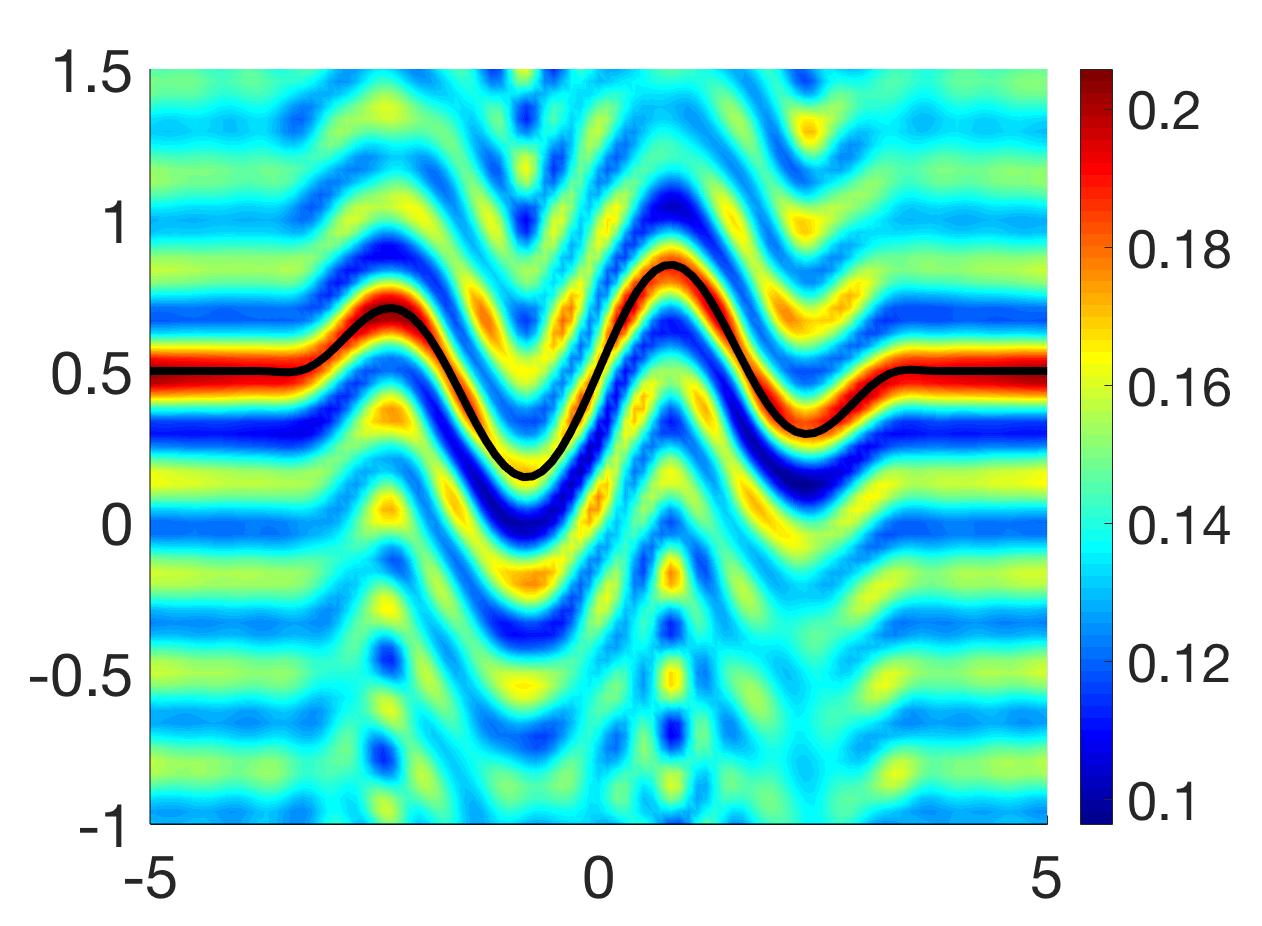}
		(a) Dirichlet
	\end{minipage}\qquad
	\begin{minipage}[t]{0.28\linewidth}
		\centering
		\includegraphics[width=1.7in]{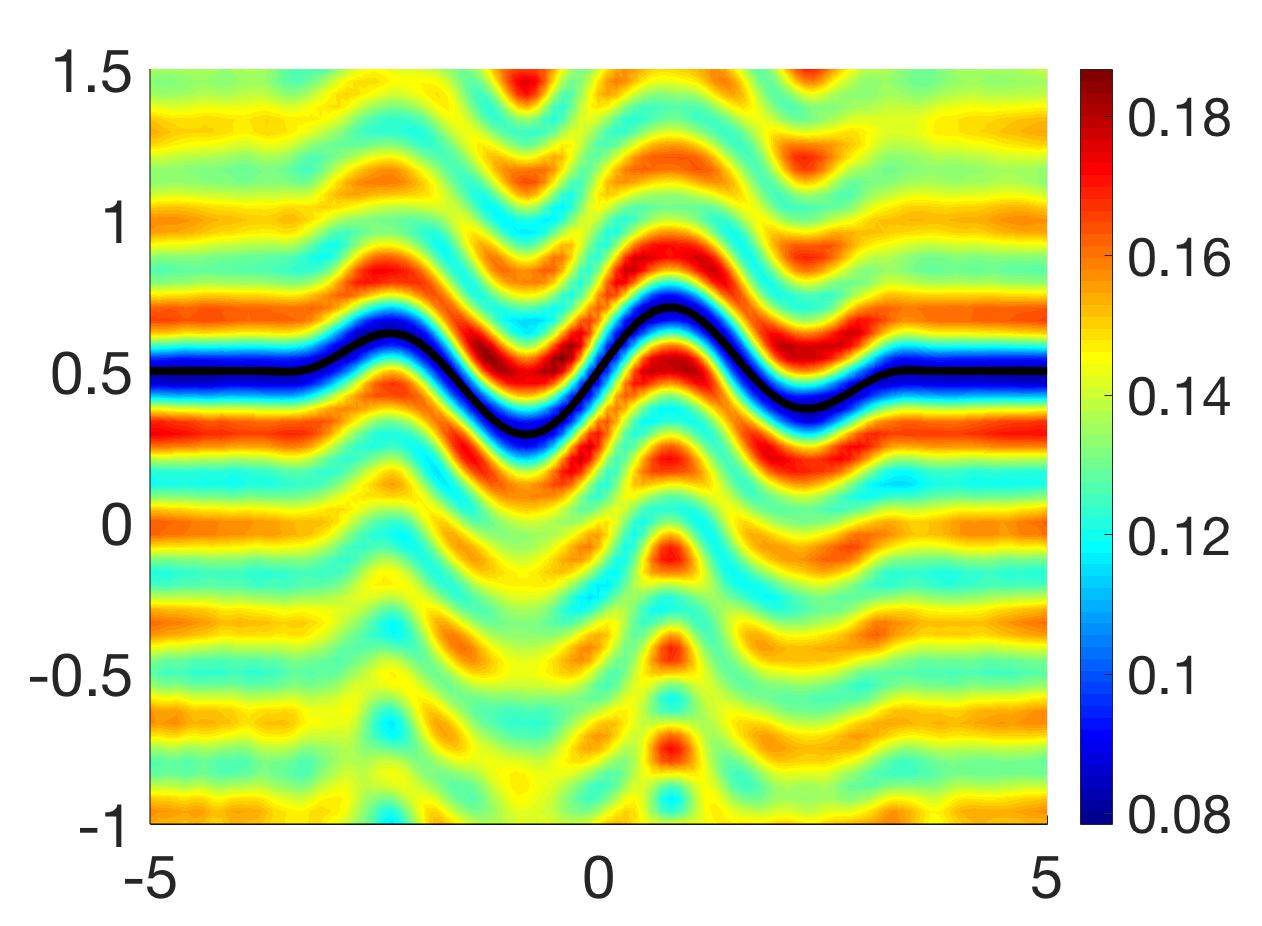}
		(b) Neumann
	\end{minipage}\qquad
	\begin{minipage}[t]{0.28\linewidth}
		\centering
		\includegraphics[width=1.7in]{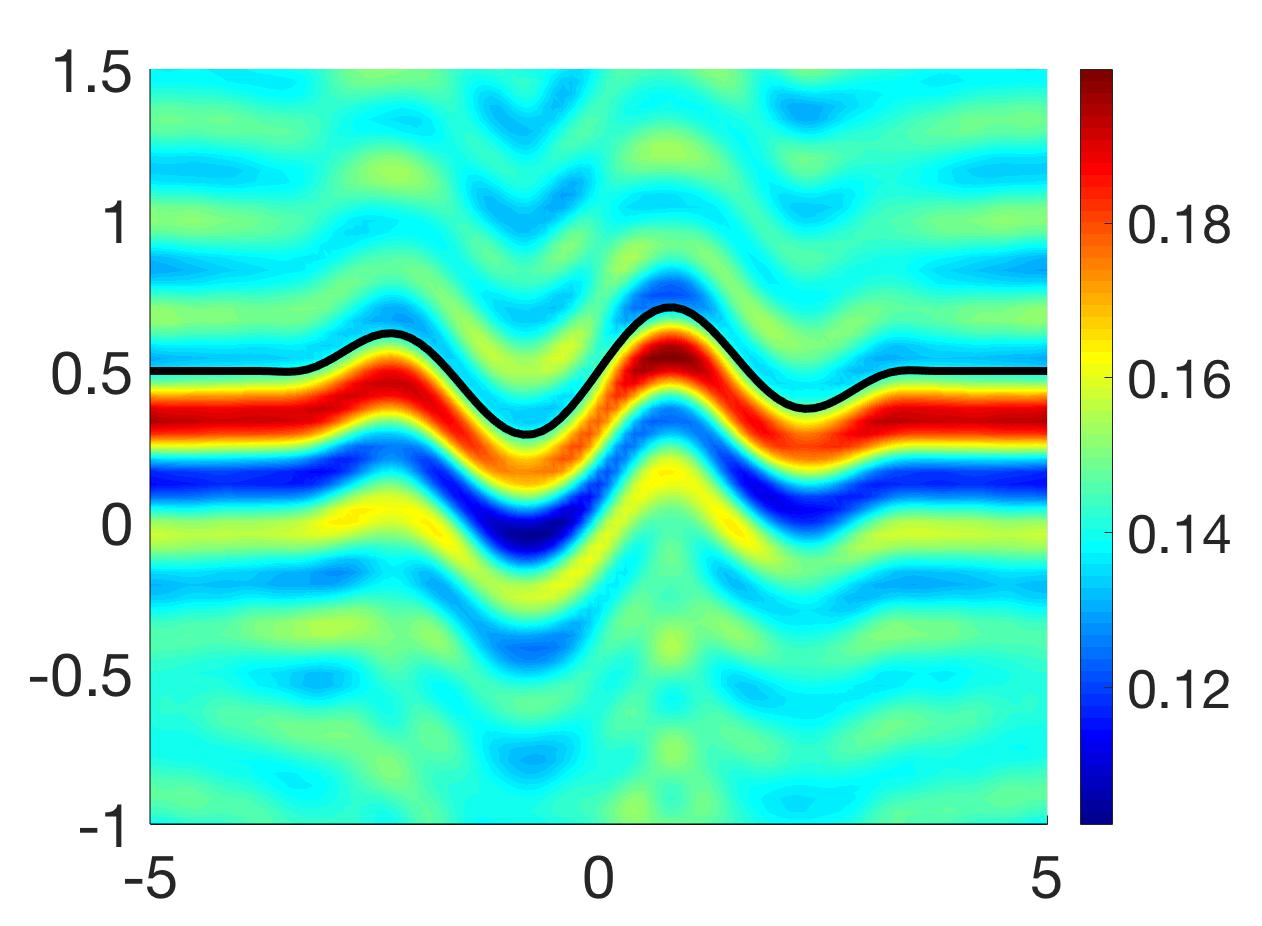}
		(c) Penetrable
	\end{minipage}
	\begin{minipage}[t]{0.28\linewidth}
		\centering
		\includegraphics[width=1.7in]{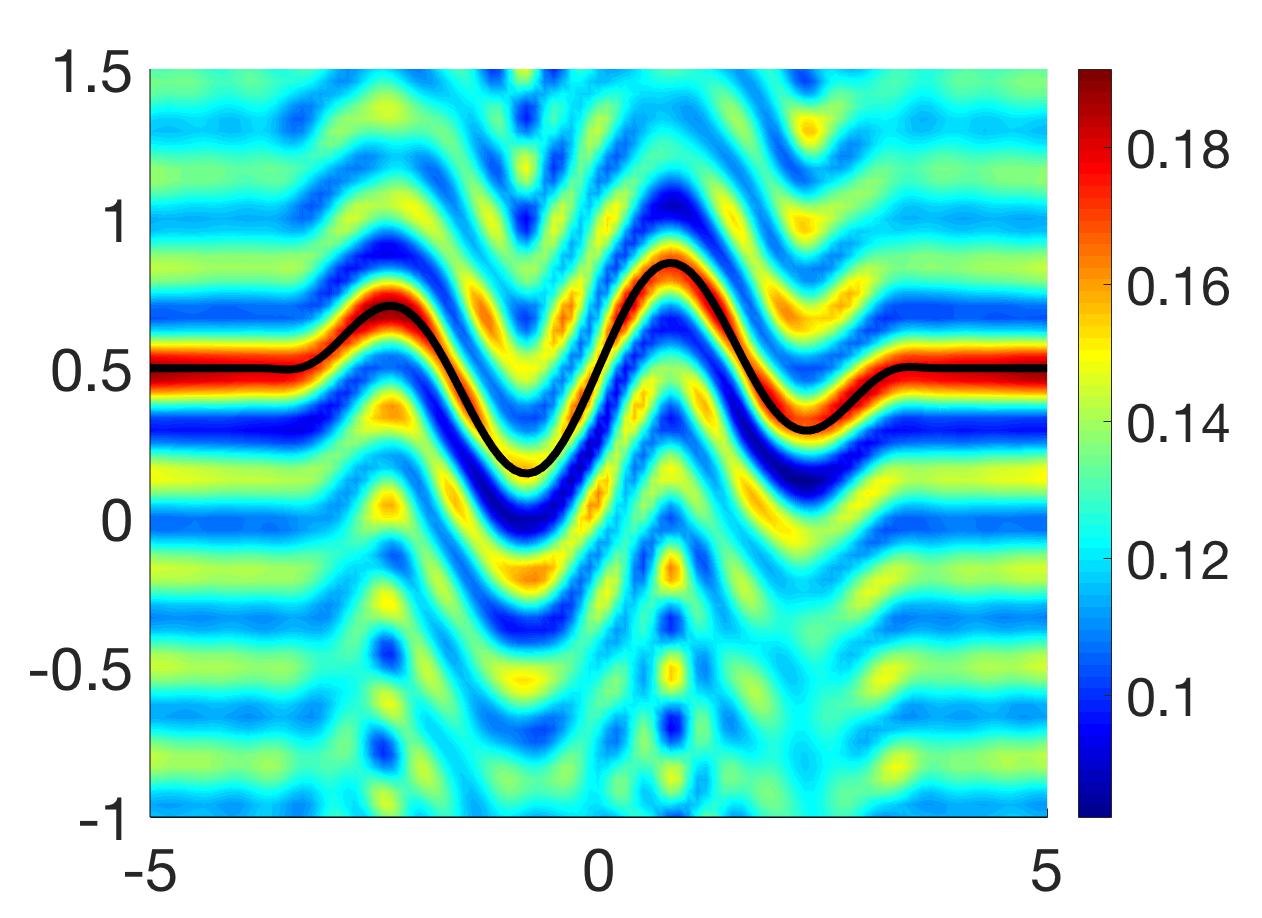}
		(d) Dirichlet
	\end{minipage}\qquad
	\begin{minipage}[t]{0.28\linewidth}
		\centering
		\includegraphics[width=1.7in]{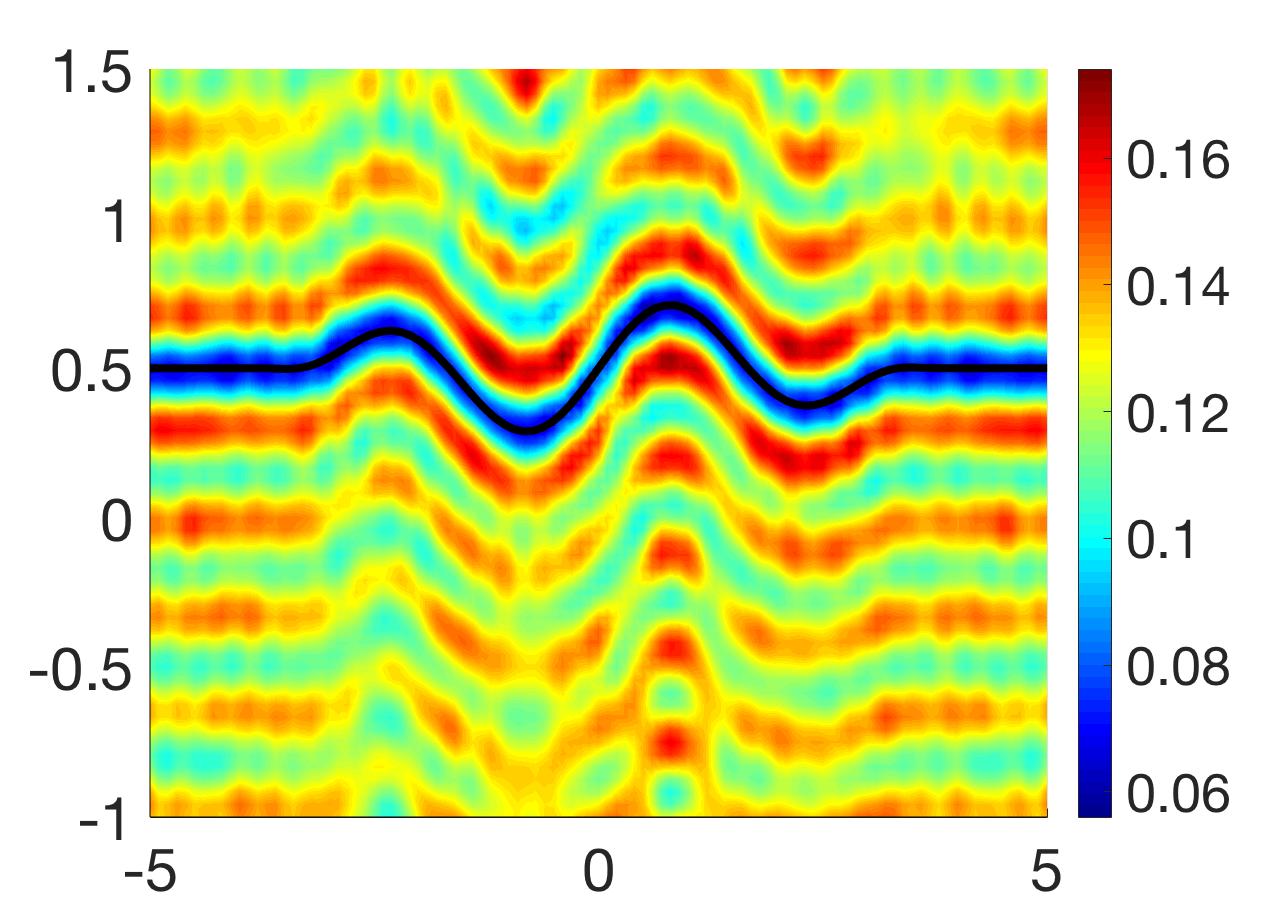}
		(e) Neumann
	\end{minipage}\qquad
	\begin{minipage}[t]{0.28\linewidth}
		\centering
		\includegraphics[width=1.7in]{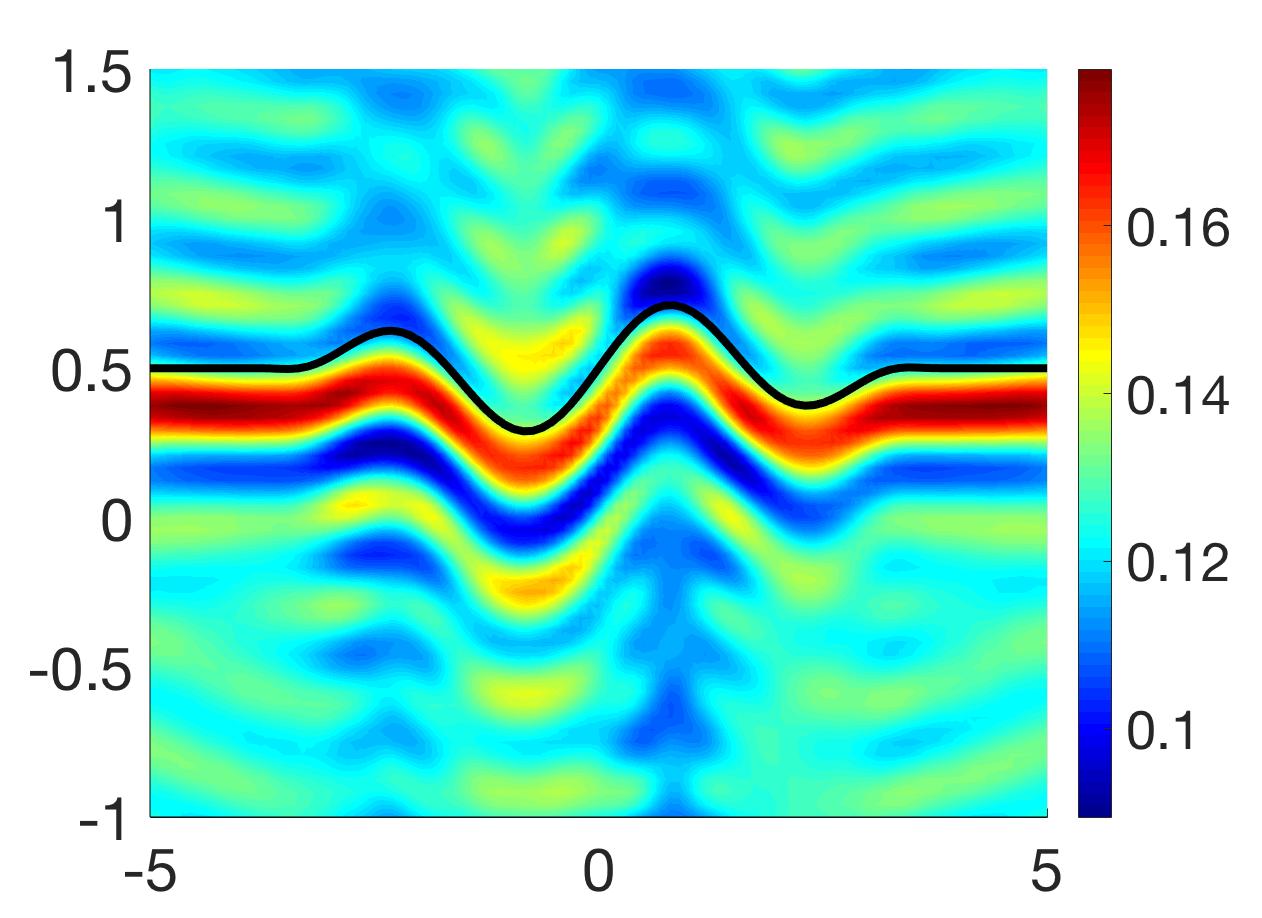}
		(f) Penetrable
	\end{minipage}
	\qquad\qquad
\caption{Reconstructions of the locally rough surface given in Example 1 from data with no noise.The first and second rows are the reconstructions from near-field and far-field, respectively.}\label{f4}
\end{figure}

\begin{figure}[htbp]
	\begin{minipage}[t]{0.28\linewidth}
		\centering
		\includegraphics[width=1.7in]{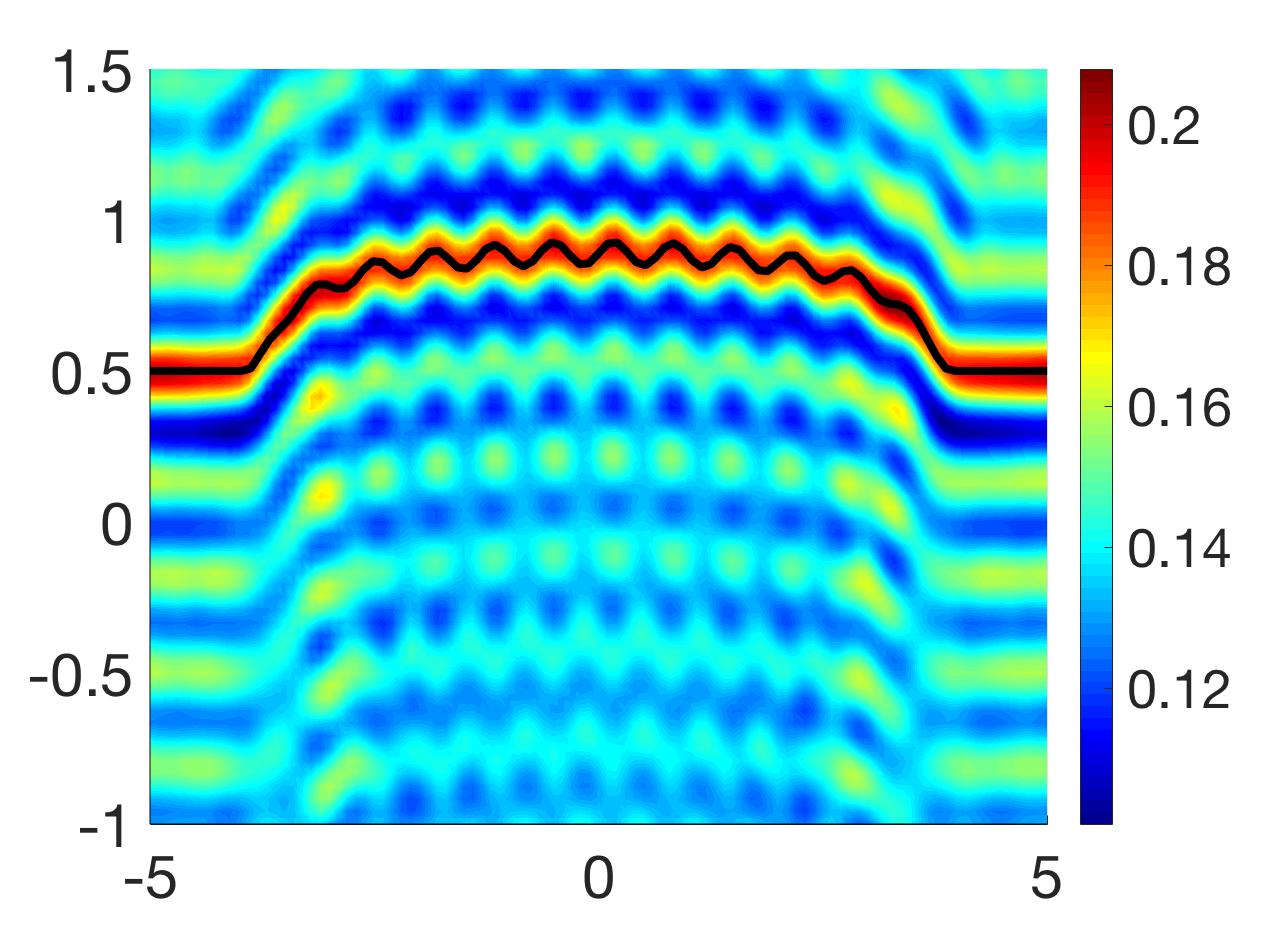}
		(a) Dirichlet
	\end{minipage}\qquad
	\begin{minipage}[t]{0.28\linewidth}
		\centering
		\includegraphics[width=1.7in]{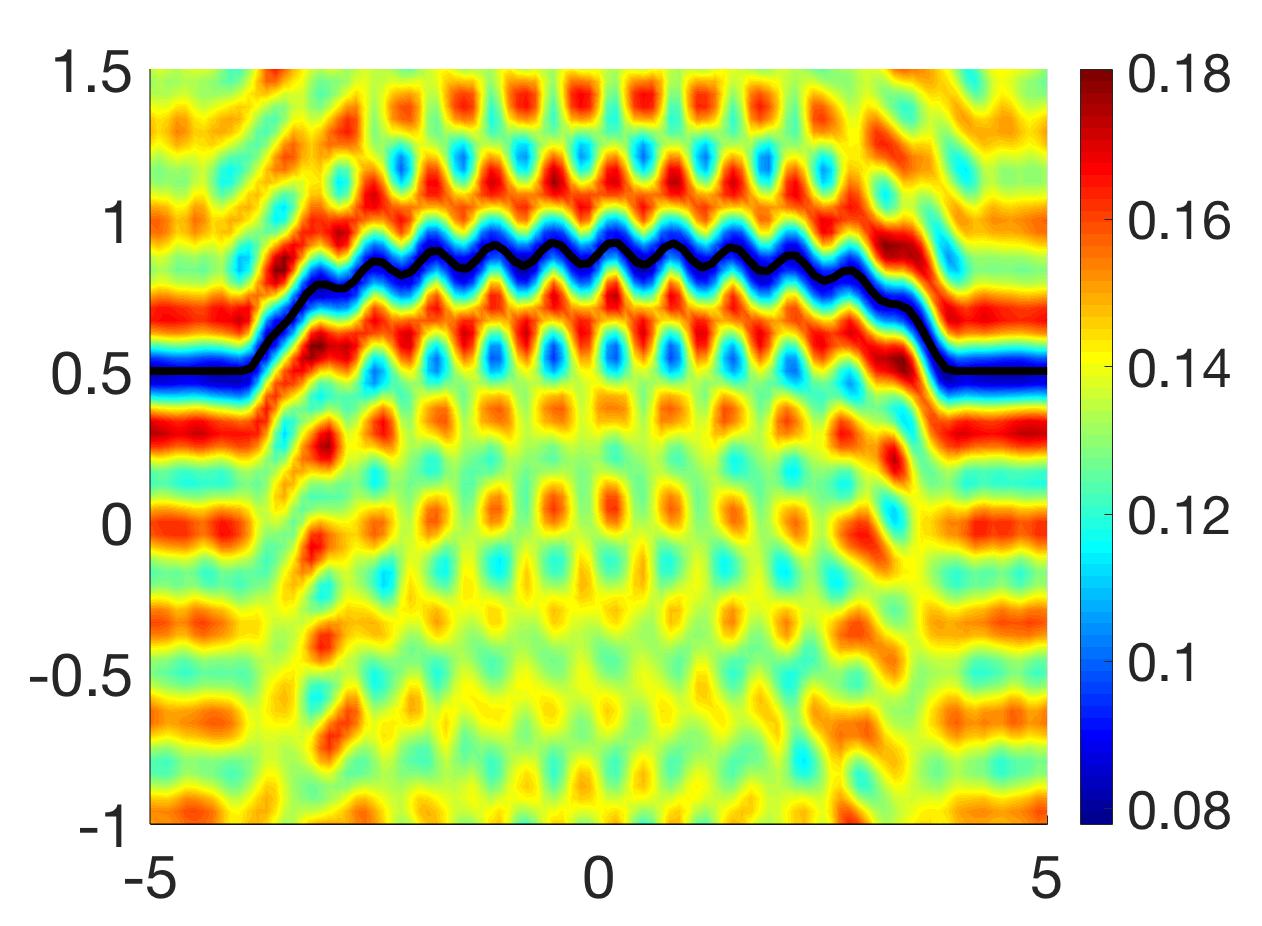}
		(b) Neumann
	\end{minipage}\qquad
	\begin{minipage}[t]{0.28\linewidth}
		\centering
		\includegraphics[width=1.7in]{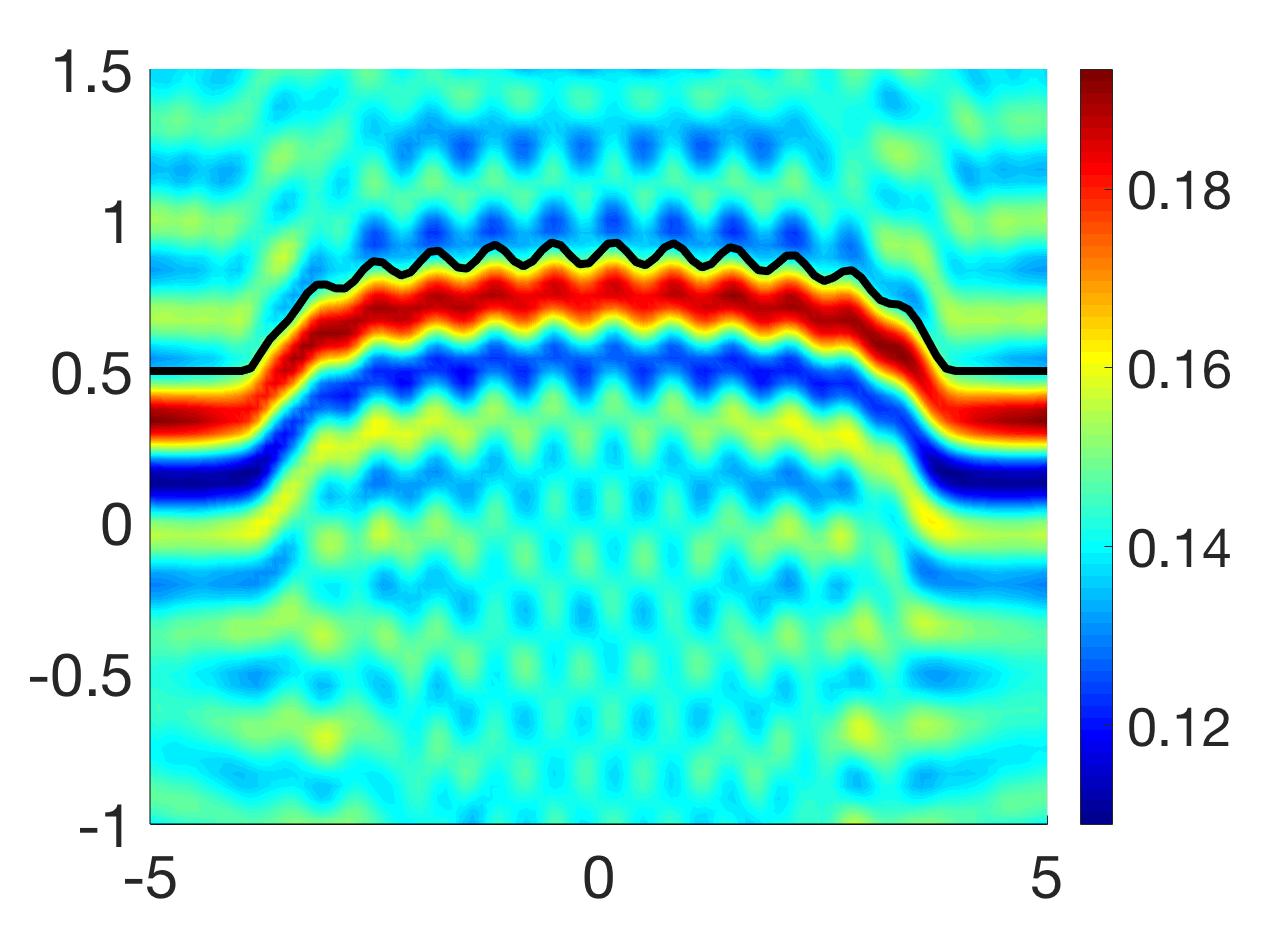}
		(c) Penetrable
	\end{minipage}
	\begin{minipage}[t]{0.28\linewidth}
		\centering
		\includegraphics[width=1.7in]{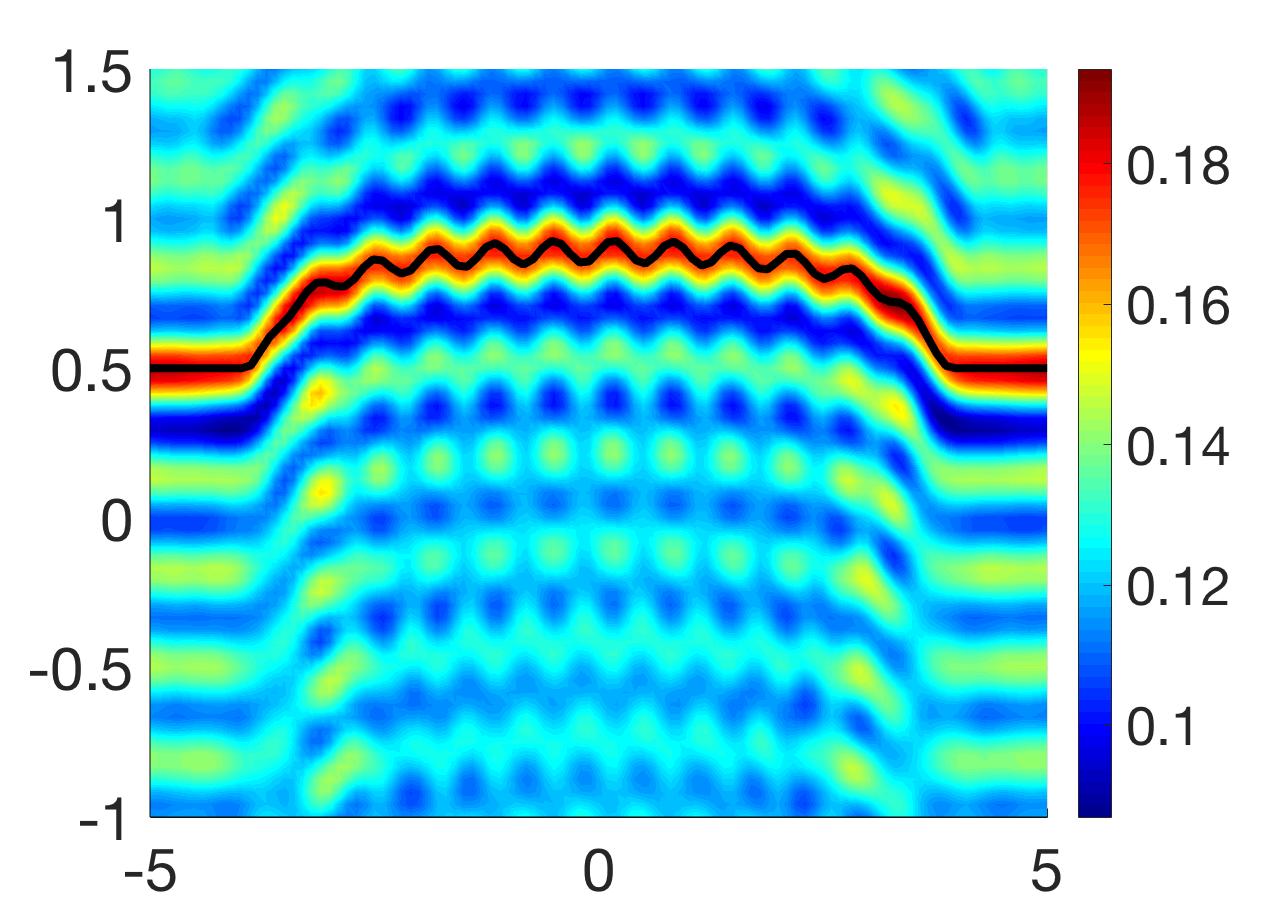}
		(d) Dirichlet
	\end{minipage}\qquad
	\begin{minipage}[t]{0.28\linewidth}
		\centering
		\includegraphics[width=1.7in]{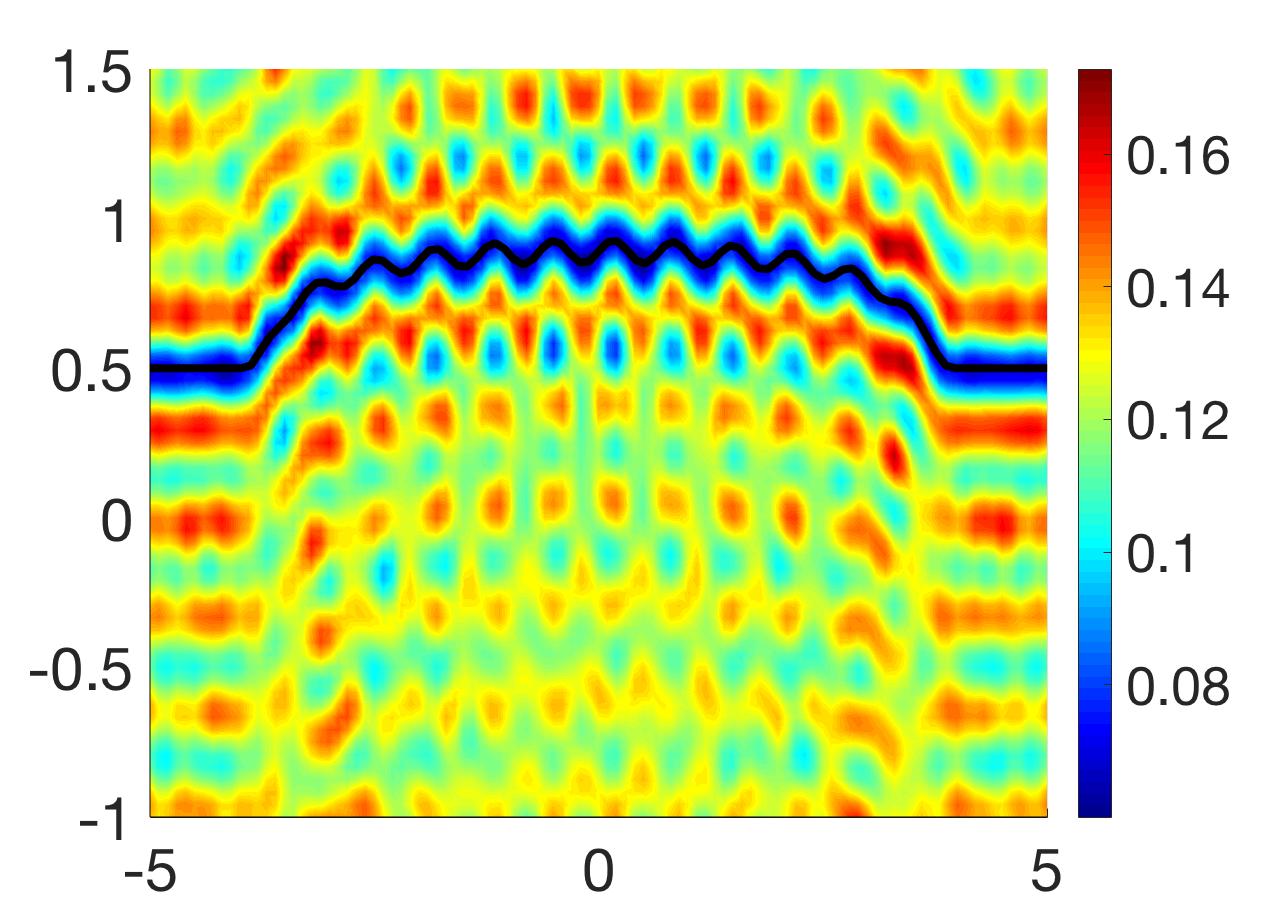}
		(e) Neumann
	\end{minipage}\qquad
	\begin{minipage}[t]{0.28\linewidth}
		\centering
		\includegraphics[width=1.7in]{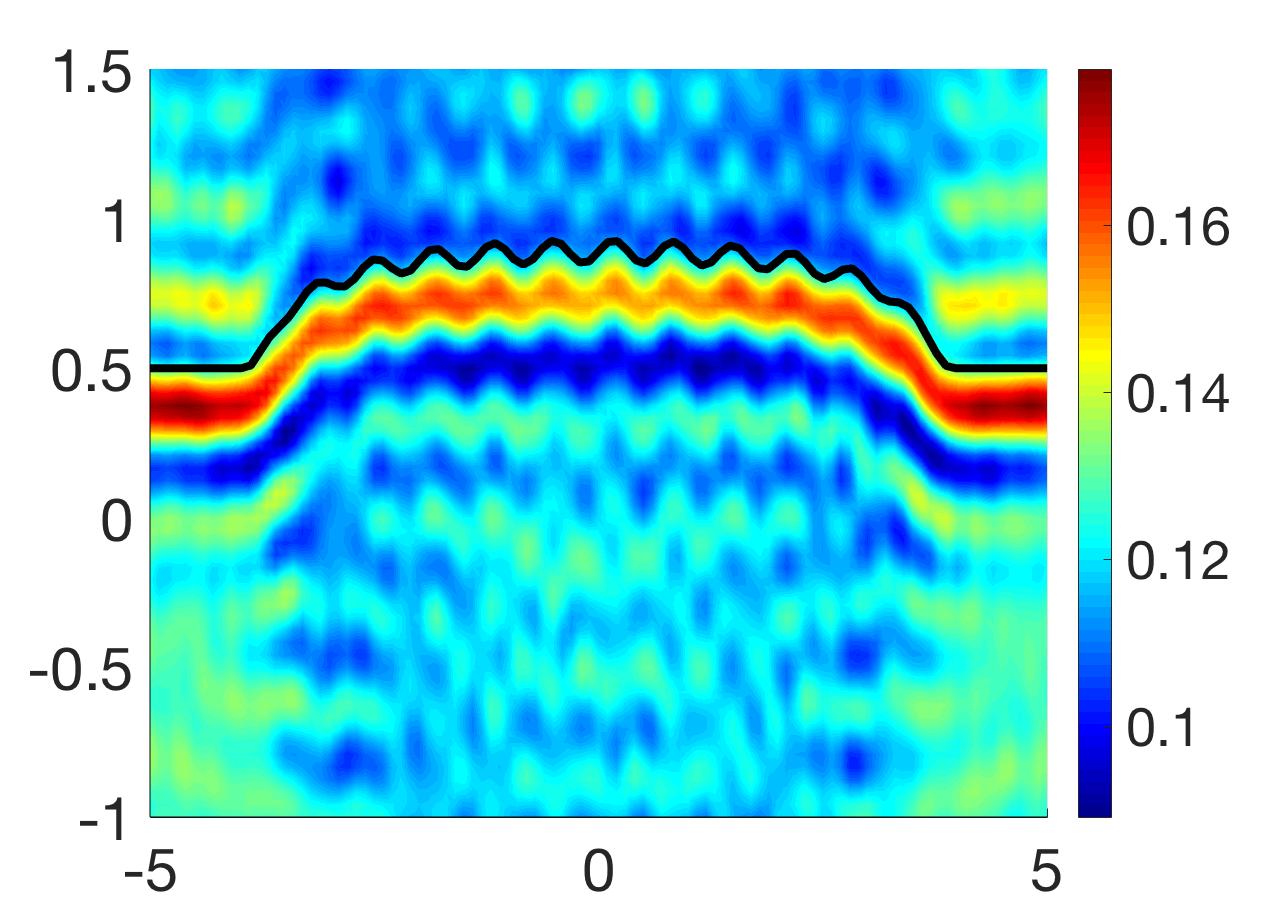}
		(f) Penetrable
	\end{minipage}
	\qquad\qquad
		\caption{Reconstructions of the locally rough surface given in Example 2 from data with 5\% noise. The first, second rows are the reconstructions from near-field, far-field, respectively.}\label{f5} 
\end{figure}

\begin{figure}[htbp]
	\begin{minipage}[t]{0.28\linewidth}
		\centering
		\includegraphics[width=1.7in]{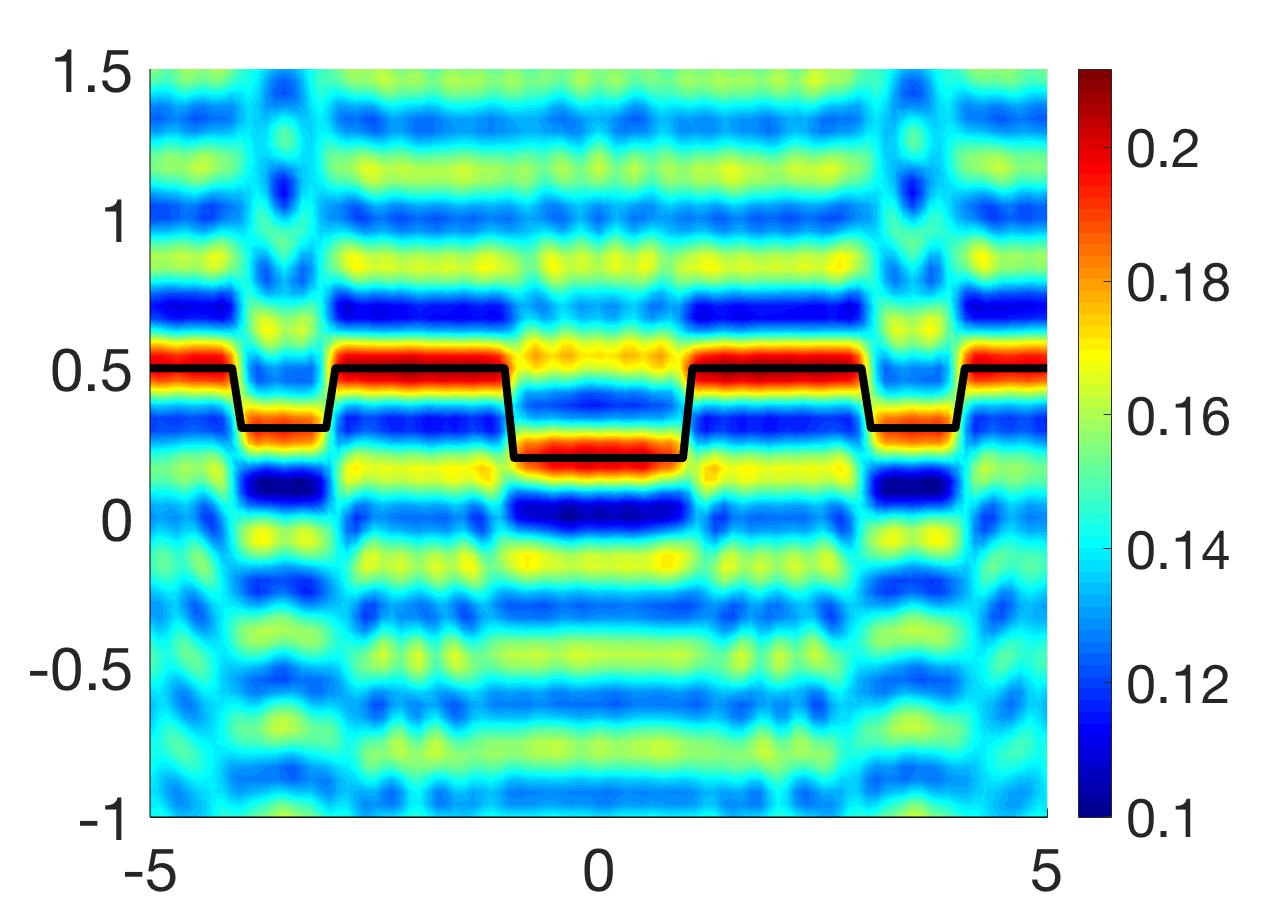}
		(a) Dirichlet
	\end{minipage}\qquad
	\begin{minipage}[t]{0.28\linewidth}
		\centering
		\includegraphics[width=1.7in]{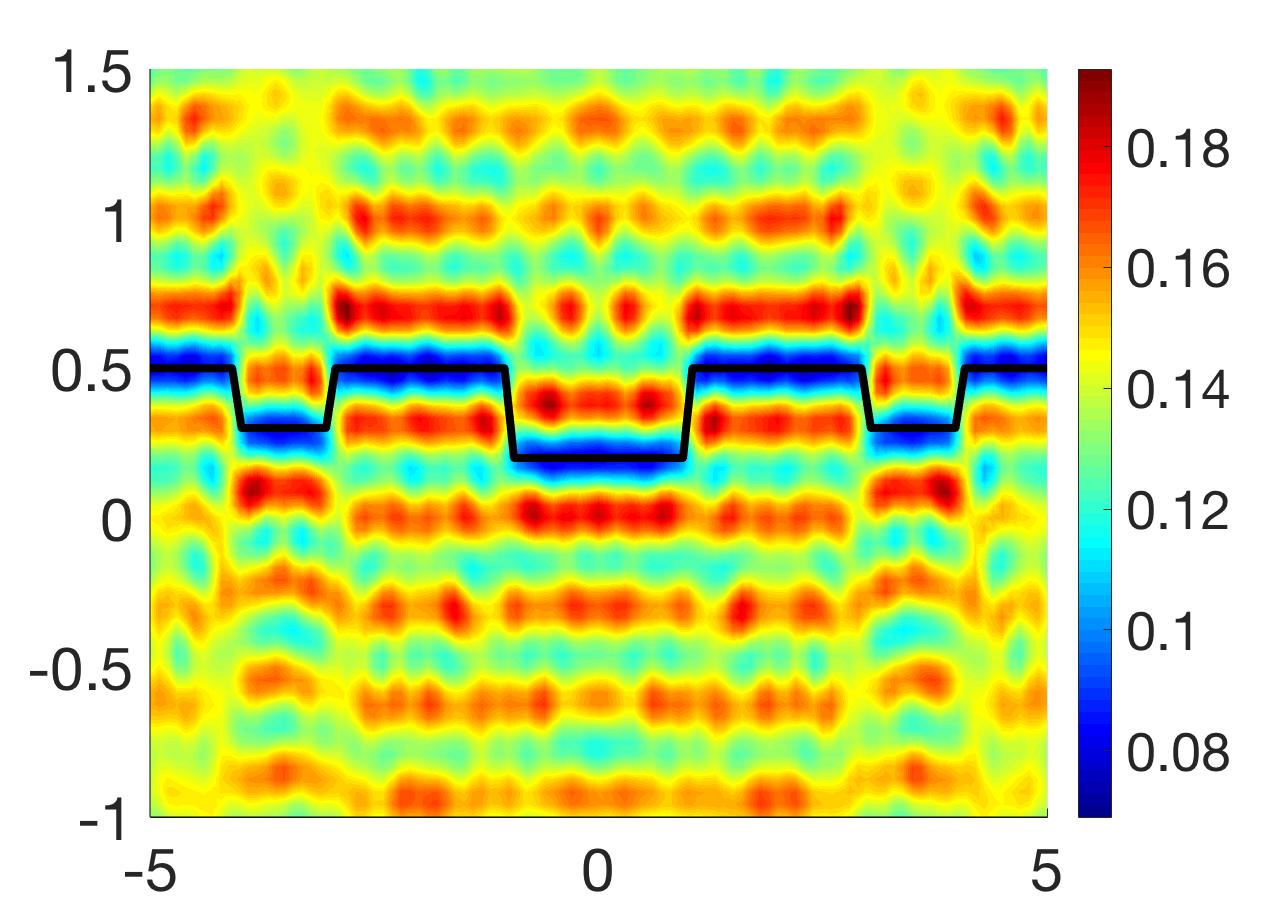}
		(b) Neumann
	\end{minipage}\qquad
	\begin{minipage}[t]{0.28\linewidth}
		\centering
		\includegraphics[width=1.7in]{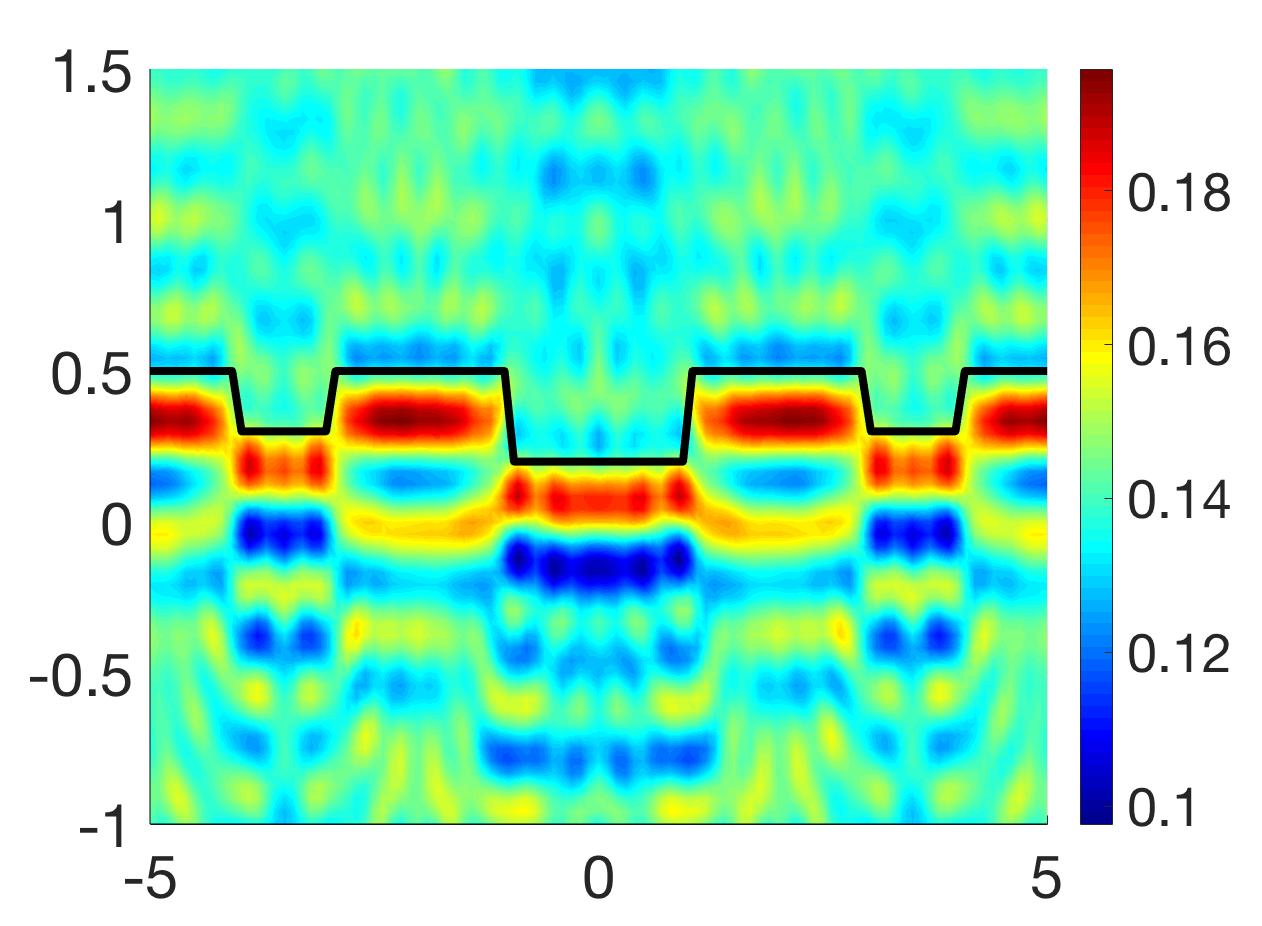}
		(c) Penetrable
	\end{minipage}
	\begin{minipage}[t]{0.28\linewidth}
		\centering
		\includegraphics[width=1.7in]{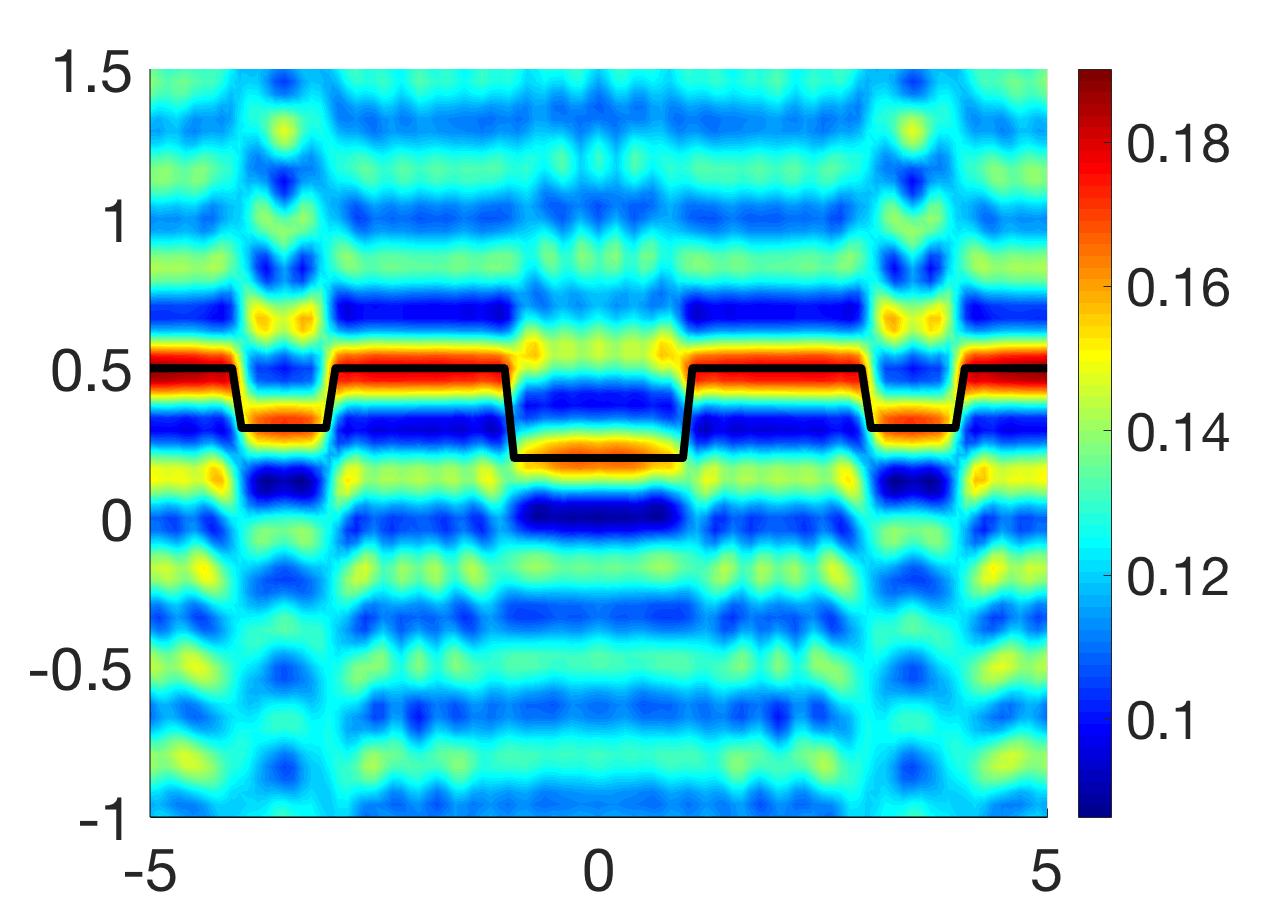}
		(d) Dirichlet
	\end{minipage}\qquad
	\begin{minipage}[t]{0.28\linewidth}
		\centering
		\includegraphics[width=1.7in]{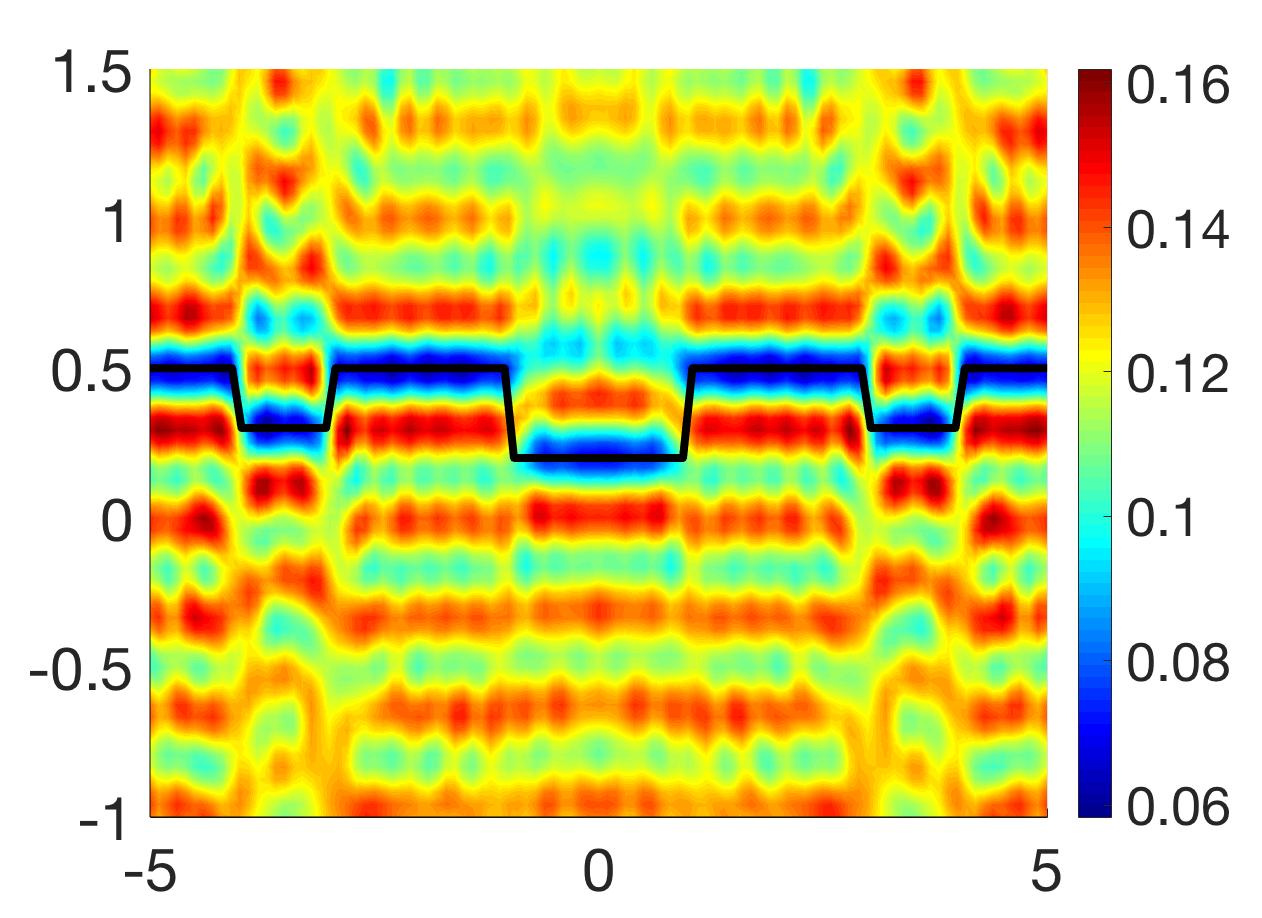}
		(e) Neumann
	\end{minipage}\qquad
	\begin{minipage}[t]{0.28\linewidth}
		\centering
		\includegraphics[width=1.7in]{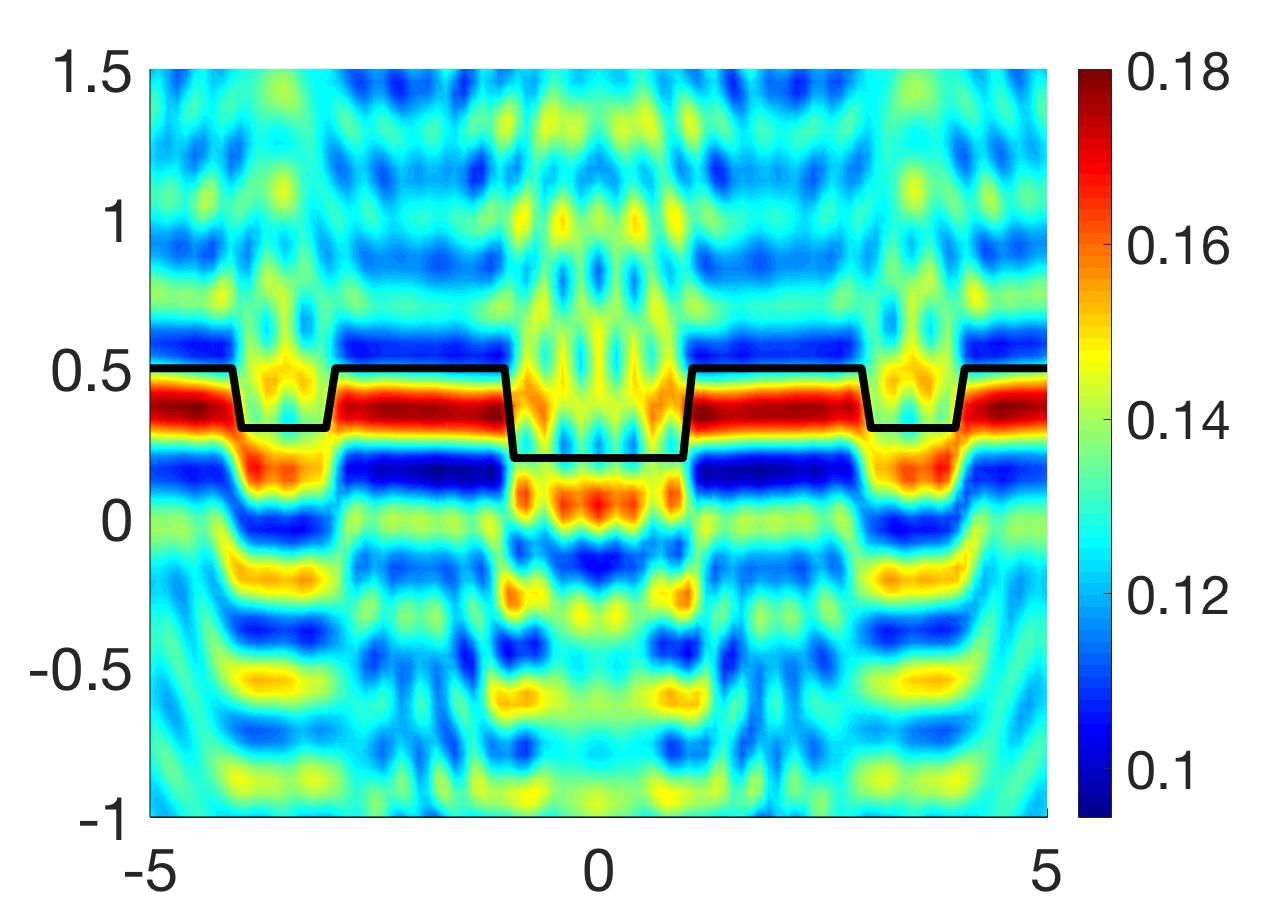}
		(f) Penetrable
	\end{minipage}
	\qquad\qquad
	\caption{Reconstructions of the locally rough surface given in Example 3 from data with 5\% noise. The first, second rows are the reconstructions from near-field, far-field, respectively.}\label{f6}
\end{figure}

From the above numerical experiments, it can be observed that the RTM method proposed in Theorem \ref{thm1} and Theorem \ref{thm3} can provide accurate and stable reconstructions for a variety of locally rough surfaces with the Dirichlet, the Neumann, and the transmission boundary conditions. In addition, it is easily seen that the RTM method could give a high quality reconstruction for some complicated locally rough surfaces such as multiscale case and piecewise continuous case. 

\section{Conclusion}
This paper proposed extended RTM methods to recover the shape and location of a locally rough surface with a Dirichlet, Neumann, or transmission boundary conditions from both the near- and far-field measurements. The idea is mainly based on constructing a modified Helmholtz-Kirchhoff identity associated with a special locally rough surface, and a novel mixed reciprocity relation. Numerical experiments demonstrated that the inversion algorithms can provide 
a stable and satisfactory reconstruction for a variety of locally rough surfaces. As far as we know, this is the first result for RTM approach to recover an unbounded rough surface. However, it is more challenging to extend the RTM method to reconstruct a diffraction grating and a non-local rough surface. We hope to report the progress on this topic in the future.

\section*{Acknowledgments}

This work was supported by the NNSF of China grants No. 12171057, 12122114, and Education
Department of Hunan Province No. 21B0299. The authors thank Prof. Haiwen Zhang for the valuable discussions.


\end{document}